 \documentclass[11pt,letterpaper]{amsart}

 \usepackage{txfonts} 

  \usepackage[colorlinks,linkcolor=blue,citecolor=blue]{hyperref}
 
 \theoremstyle{plain}

\newtheorem{theorem}{Theorem}[section]
\newtheorem{lemma}[theorem]{Lemma}

\theoremstyle{definition}
\newtheorem{definition}[theorem]{Definition}

\newtheorem{corollary}[theorem]{Corollary}
\newtheorem{proposition}[theorem]{Proposition}

\theoremstyle{remark}
\newtheorem{remark}[theorem]{Remark}

\numberwithin{equation}{section}

\begin{document}

\title[Fully nonlinear elliptic equations]
 {Fully non-linear elliptic equations on complex manifolds}

 
 \author{Rirong Yuan }
 \address{School of Mathematics, South China University of Technology, Guangzhou 510641, China}
 \email{yuanrr@scut.edu.cn}



\dedicatory{}

\begin{abstract} 
In this paper, we study a broad class of fully nonlinear elliptic equations on Hermitian manifolds. On one hand, under the optimal structural assumptions we derive $C^{2,\alpha}$-estimate for solutions of the equations on closed Hermitian manifolds.  
On the other hand, we treat the Dirichlet problem. In both cases, we prove the existence theorems with unbounded condition.
\end{abstract}

\maketitle


  
  \section{Introduction}
  \label{sec1-introduction}  
  
Let $(M, \omega)$ be a 
  Hermitian manifold of complex dimension $n$,
  $\omega=\sqrt{-1}g_{i\bar j} dz^i\wedge d\bar z^j$,
and let
  $\chi= \sqrt{-1}\chi_{i\bar j}dz^i\wedge d\bar z^j$ be a smooth real  $(1,1)$-form.  
 We consider the fully nonlinear elliptic equation
   \begin{equation}
   	\label{equ1-main}
   	\begin{aligned} 	
   		F(\chi+\sqrt{-1}\partial\overline{\partial}u)= \psi  
   	\end{aligned}
  \end{equation}  
for a given sufficiently smooth function $\psi$ on $M$. 
The operator has the form 
\begin{equation} 
	\label{Hessian-form1}
	\begin{aligned} 	F(\chi+\sqrt{-1}\partial\overline{\partial}u)\equiv f(\lambda[\omega^{-1}(\chi+\sqrt{-1}\partial\overline{\partial}u)]),
	\end{aligned}
\end{equation}
where $\lambda[\omega^{-1}(\chi+\sqrt{-1}\partial\overline{\partial}u)]=(\lambda_1,\cdots,\lambda_n)$ denotes the $n$-tuple   of eigenvalues of   $\chi+\sqrt{-1}\partial\overline{\partial}u$ with respect to $\omega$.
 Henceforth, 
   $f$ is a smooth symmetric function defined in an open symmetric convex cone  
    $\Gamma\subset \mathbb{R}^n$, 
    with vertex at the origin 
    and  
$\partial \Gamma\neq\emptyset$,  
$$\Gamma_n\equiv \{(\lambda_1,\cdots,\lambda_n)\in \mathbb{R}^n: \mbox{ each component } \lambda_i>0 \}\subseteq\Gamma.$$  
In addition, we assume that $f$ satisfies the following fundamental hypotheses:
\begin{equation}
	\label{elliptic}
	\begin{aligned} 
		f_{i}(\lambda)=f_{\lambda_i}(\lambda)
		\equiv\frac{\partial f}{\partial \lambda_{i}}(\lambda) > 0 \mbox{ in } \Gamma, \,\, \forall 1\leqslant i\leqslant n,
	\end{aligned}
\end{equation}  
\begin{equation}\label{concave}\begin{aligned}
		\mbox{$f$ is a concave function in $\Gamma$,}
\end{aligned}\end{equation} 
\begin{equation}
	\label{nondegenerate}
	\begin{aligned}
 	\inf_M\psi>\sup_{\partial\Gamma} f, 
	\end{aligned}
\end{equation}
 where  $$\underset{\partial\Gamma}{\sup} f \equiv \underset{\lambda_0\in\partial\Gamma}{\sup} \underset{\lambda\to\lambda_0}{\limsup}f(\lambda). $$ 
 %
%
 
Following \cite{CNS3} we call  $u$ an \textit{admissible} function, if 
at each point $z$
\begin{align*}
	\lambda[\omega^{-1}(\chi+\sqrt{-1}\partial\overline{\partial}u)] 
	(z)
	\in \Gamma.
\end{align*}
The condition \eqref{elliptic} ensures the ellipticity of 
\eqref{equ1-main}  at admissible solutions, while \eqref{nondegenerate} prevents the equation from degeneracy. Assumption
\eqref{concave} guarantees that 
 $F(\tilde{\chi})$ is concave 
for  $\tilde{\chi}$ 
with $\lambda(\omega^{-1}\tilde{\chi})\in\Gamma$.  
 The assumptions  allow us to apply  Evans-Krylov theorem  
 to derive $C^{2,\alpha}$ bounds from \textit{a priori}
  estimates up to second order. 

 The study of equations of this type was pioneered 
by Caffarelli-Nirenberg-Spruck \cite{CNS3}.
%
Perhaps   the most important operator of the form  \eqref{Hessian-form1} is the complex Monge-Amp\`ere operator, where we take $f=\sum_{i=1}^n\log\lambda_i$, $\Gamma=\Gamma_n$.   The complex  Monge-Amp\`ere equation has played a significant role in Yau's  \cite{Yau78}  proof of Calabi's conjecture in connection with
the Ricci curvature   and volume form over K\"ahler manifolds.
Over the past decades,
their work has fundamental impact in the study of fully nonlinear equations of elliptic and parabolic type. More recently, Guan \cite{Guan12a} and Sz\'ekelyhidi \cite{Gabor} made important contributions to 
the notion of subsolution, which is   
 a key ingredient for 
  a priori estimate and solvability of fully nonlinear 
  elliptic
 equations  on closed  manifolds. 


It is an important problem to establish $C^{2,\alpha}$-estimate  
for general fully nonlinear elliptic equations
on closed complex manifolds,  since  
there have been increasing interests in fully nonlinear PDEs  
from complex geometry. 
The problem was studied  in  existing 
  literature,   among which 
\cite{Dinew2017Kolo,HouMaWu2010,Chen04,SW08,FLM2011,FuWangWuFormtype2010,FuWangWuFormtype2015,Tosatti2013Weinkove,Tosatti2019Weinkove,STW17,Jacob-Yau17,CJY2020CJM,Gabor}, to name just a few, focusing on complex  $k$-Hessian equation, $J$-flow and extensions, 
Form-type Calabi-Yau equation, Monge-Amp\`ere equation for $(n-1)$-PSH functions,  deformed Hermitian Yang-Mills equation, 
 and more general  equations
satisfying an extra assumption that
\begin{equation}	
	\label{addistruc}	
	\begin{aligned}
		\lim_{t\rightarrow+\infty} f(t\lambda)>f(\mu), \,  \forall \lambda,\, \mu\in\Gamma.
	\end{aligned}
\end{equation} 
Such a condition or certain slightly stronger assumption was commonly imposed as a key assumption in a huge amount of literature  starting   from the landmark work of Caffarelli-Nirenberg-Spruck \cite{CNS3}.  
It is interesting 
and challenging  
to drop such additional conditions from the viewpoint of geometry and PDEs.

When $M$ admits  nonempty boundary $\partial M$,   given  a function $\psi$ and a  smooth real $(1,1)$-form  $\chi$ on $\bar M\equiv M\cup\partial M$,
it is natural to consider the Dirichlet problem
  \begin{equation}\label{mainequ-Dirichlet}
	\begin{aligned}
   		F(\chi+\sqrt{-1}\partial\overline{\partial}u)=\psi \mbox{ in } \bar M, \,\, u=\varphi  \mbox{ on }  \partial M.
	\end{aligned}
\end{equation}   
The Dirichlet problem for complex Monge-Amp\`ere equation in $\mathbb{C}^n$ was solved by  Caffarelli-Kohn-Nirenberg-Spruck
 \cite{CKNS2} on strictly pseudoconvex domains.
In \cite{Guan1998The}  Guan  extended their result   to arbitrary bounded smooth domains in $\mathbb{C}^n$, 
 assuming  existence of subsolution.  
   Under the assumption of existence of a strict subsolution, Li   \cite{LiSY2004} treated the Dirichlet problem for  general equations  on   bounded domains in unbounded case.
  Nevertheless,
the existence and regularity of Dirichlet problem for general fully nonlinear elliptic equations
are widely open  on curved complex manifolds. 
The primary obstruction is  gradient estimate (first order estimate). 
 To the best knowledge of the author, the direct proof of gradient estimate 
is  rarely known, beyond the complex Monge-Amp\`ere equation 
\cite{Hanani1996,Blocki09gradient,Guan2010Li} and   other fairly restrictive cases
considered by
\cite{Guan2015Sun,yuan2018CJM,yuan2021cvpde,ZhangXW}.
This is in contrast with  Hessian 
equations on Riemannian manifolds;
we refer to \cite{Guan-Dirichlet} for 
recent progress on  direct proof of gradient estimate 
with  some earlier work by \cite{LiYY1990,Urbas2002}. 

 The subject of this paper is to  
consider the  above-mentioned problems.  In particular, 
  we  can draw some existence theorems for the equation on closed Hermitian manifolds and the Dirichlet problem 
  in the unbounded case
  \begin{equation}
  	\label{unbounded-1}
  	\begin{aligned}
  		\lim_{t\rightarrow+\infty}  f(\lambda_1,\cdots,\lambda_{n-1},\lambda_n+t)=+\infty, \,\, \forall \lambda=(\lambda_1,\cdots,\lambda_n)\in\Gamma.
  	\end{aligned}
  \end{equation}


\subsection{Equations on closed manifolds}

The first object of this article is to derive 
\textit{a priori} estimate for equation  \eqref{equ1-main}   on closed Hermitian manifolds.  
For our purpose, we require the notion  of $\mathcal{C}$-subsolution.   
 
\begin{definition}
	[$\mathcal{C}$-subsolution] We say that a $C^2$ function $\underline{u}$ is a $\mathcal{C}$-subsolution of \eqref{equ1-main}, if  for any $z\in M$  the following set is bounded
	\[ \left\{ \mu\in\Gamma: f(\mu)=\psi(z),  \, \mu  \in \lambda [\omega^{-1}(\chi+\sqrt{-1}\partial\overline{\partial}\underline{u})](z)+\Gamma_n \right\}.\]
	
\end{definition} 

The notion of $\mathcal{C}$-subsolution was  
 developed 
 by Guan \cite{Guan-Dirichlet} and more specifically Sz\'ekelyhidi \cite{Gabor}. Subsequently,  Phong-T\^{o}  \cite{Phong-To2017} proposed a parabolic version  of $\mathcal{C}$-subsolution. 
For  $J$-equation and inverse $\sigma_k$ flow, the $\mathcal{C}$-subsolution coincides with the cone conditions   considered in earlier work  \cite{FLM2011,SW08}. 
Such a notion is very important, since   
the $\mathcal{C}$-subsolution assumption can be viewed as an indication of a priori estimate and existence  for fully nonlinear  equations  on closed manifolds. 
%
 %

Our first result  may be stated as follows:

\begin{theorem}
	\label{thm1-main}
	Let $(M, \omega)$ be a closed Hermitian manifold.  
In addition to \eqref{elliptic}, \eqref{concave} and \eqref{nondegenerate}, we assume $\psi\in C^\infty(M)$ and that 
	 \eqref{equ1-main} admits a $\mathcal{C}$-subsolution $\underline{u}\in C^2(M)$.
	Let $u\in C^4(M)$ be an admissible solution to  \eqref{equ1-main} with 
	$\sup_M u=0$,
 then for any $\alpha\in (0,1)$ we have 
	\begin{align*}
\|u\|_{C^{2,\alpha}} \leqslant C,
	\end{align*}
  where $C$ is a positive constant 
   depending on the  data  
   $(M, \omega)$,  $\alpha$, $f$, $\psi$, $\chi$
   and  $\underline{u}$.
\end{theorem}

The theorem is optimal in the sense that it fails if any of conditions \eqref{elliptic}, \eqref{concave} and \eqref{nondegenerate} is removed. This settles a problem left open by Sz\`ekelyhidi \cite{Gabor}, wherein he  derived the $C^{2,\alpha}$-estimate
under the  extra assumption   \eqref{addistruc}. 
 As we know, the additional
condition \eqref{addistruc} is a vital  ingredient for  his  approach towards  second order estimate and especially  gradient estimate via blow-up argument.  
	%

Below we sketch the proof of  Theorem \ref{thm1-main}.
It suffices to set up the estimate for complex Hessian of the solution $u$. With this bound at hand, 
similar to \cite{Guan2010Li} we can establish the bound on real Hessian of $u$ using maximum principle and then the usual Evans-Krylov theorem \cite{Evans82,Krylov83} gives  $C^{2,\alpha}$ estimate.

The $C^0$-estimate for 
 \eqref{equ1-main}
 is known by the work of Sz\'ekelyhidi \cite{Gabor}, 
generalizing the method that B{\l}ocki \cite{Blocki04-zeroestimate}
used in the case of   complex Monge-Amp\`ere equation.  The strategy is based on Alexandroff-Bakelman-Pucci 
 inequality, which is different  from  Yau's \cite{Yau78}   Moser iteration method. The $C^0$-estimate for complex Monge-Amp\`ere equation and more general 
  complex $(k,l)$-quotient  equations 
  on closed Hermitian manifolds
  can also be found  in \cite{TW10} and \cite{SunWei-cpam}, respectively.
Recently,
Guo-Phong-Tong \cite{Guo-Phong-Tong2023} 
 gave a PDE proof of sharp $L^\infty$-estimate for complex Monge-Amp\`ere equation and some other fully nonlinear   elliptic equations.
 
 The main issue  is to establish first and second order estimates. 
 In order to set up first order estimate, 
 we 
 use the blow-up method that  was 
  developed by Dinew-Ko{\l}odziej \cite{Dinew2017Kolo}  in the case of  complex $k$-Hessian  equations.  
 The  argument is based on a Liouville type theorem and the second order estimate of the form  	
 \begin{equation}
 	\label{estimate-quantitative}
 	\begin{aligned} 	
 		\sup_{M} |\partial\overline{\partial} u|_\omega\leqslant C\left(1+\sup_M|\nabla u|_\omega^2\right).
 \end{aligned} \end{equation}  
Such a second   estimate was first established by Hou-Ma-Wu \cite{HouMaWu2010} for complex $k$-Hessian equations on closed K\"ahler manifolds.
 More recently, their works were  extended by Sz\'ekelyhidi \cite{Gabor} to general fully nonlinear elliptic equations satisfying \eqref{addistruc}. 
   For more related work, we  refer to \cite{Tosatti2013Weinkove,Tosatti2019Weinkove,STW17} for   Monge-Amp\`ere equation for $(n-1)$-PSH functions, and also  to \cite{ZhangD2017} for complex $k$-Hessian equation. 

Below, we set up second order estimate   of the form \eqref{estimate-quantitative} in generic case.
\begin{theorem}
	\label{thm2-second-order} 
	Under the assumptions of Theorem \ref{thm1-main}, we have  
	\begin{equation}
		\begin{aligned} 	
			\sup_{M} |\partial\overline{\partial} u|_\omega\leqslant C\left(1+\sup_M|\nabla u|_\omega^2\right), \nonumber
	\end{aligned} \end{equation}  
where $C$  
   depends on the given data  $(M,\omega)$,  $f$, $|\psi|_{C^2(M)}$, $|\chi|_{C^2(M)}$,  $\underset{M}{\mathrm{osc \, }} u$
and $|\underline{u}|_{C^{2}(M)}$.             
\end{theorem}

With the aid of  second order estimate  
 of the form \eqref{estimate-quantitative}, 
similar to \cite{Dinew2017Kolo,Tosatti2013Weinkove}
we can construct a  function $v:\mathbb{C}^n\to \mathbb{R}$ as the limit of 
appropriate blowing up sequence. 
To apply  Liouville type theorem,
 what we need 
is to give a proper characterization of the  function $v$. 
To this end,
 for   $(f,\Gamma)$ we define the 
cone 
\begin{equation}
	\label{component1}
	\begin{aligned}
		\Gamma_{\mathcal{G}}^{f} =
		\left\{\lambda\in\Gamma: \lim_{t\to +\infty}f(t\lambda)>-\infty\right\}.
	\end{aligned}
\end{equation}
In addition, 
$\mathring{\Gamma}_{\mathcal{G}}^{f}$
denotes the interior of $\Gamma_{\mathcal{G}}^{f}$.
And then we may 
 employ  Lemma \ref{lemma2-key} to 
prove that 
$v$ is a  $\mathring{\Gamma}_{\mathcal{G}}^{f}$-solution in the sense of \cite[Definition 15]{Gabor}. 
This will contradict to Sz\'ekelyhidi's Liouville type theorem  \cite[Theorem 20]{Gabor} (in analogy with 
 Dinew-Ko{\l}odziej 
\cite[Theorem 3.2]{Dinew2017Kolo})  and then give the desired gradient bound.

 Comparing with the   literature  mentioned above,  in the proof of Theorem \ref{thm2-second-order} a new difficulty we face is that the  positive lower bound  $\sum_{i=1}^n f_i\geqslant \delta$ possibly fails   
when dropping  \eqref{addistruc}.
 In order to overcome this, 
  we propose  Lemmas \ref{lemma5-key} and   \ref{lemma3-key} below.
 The idea of exploiting such type lemma is from Guan \cite{Guan12a}.
 With the aid of this lemma, 
 we can make use of $(\underline{u}-u)$ 
 in the construction of auxiliary function in order to overcome the difficulties caused by geometric quantities and equations. 
The newly introduced terms of the form $1+\big|\sum  f_i \lambda_i\big|$
 in Lemma \ref{lemma5-key} will greatly simplify the process of managing and solving general equations under optimal conditions.
 As usual,  
  the casewise  argument of Hou-Ma-Wu \cite{HouMaWu2010} is employed. 
  This trick is 
  motivated by
 the treatment of Chou-Wang \cite{Chou2001Wang} on real  Hessian equation. 
In addition, we follow the technique 
of Szek\'elyhidi \cite{Gabor}
to treat the troublesome third order terms  arising from nontrivial torsion tensors of non-K\"ahler metrics.


Note that  in the unbounded case,   any admissible function is  a $\mathcal{C}$-subsolution to \eqref{equ1-main}.
As a consequence, 
we   obtain  the following existence result.
 \begin{theorem}
 	\label{thm1-existence}
 	Assume $\sup_{\partial\Gamma}f =-\infty$ and that $f$ satisfies \eqref{concave} and \eqref{unbounded-1}.
 	Suppose that $M$ carries  a $C^2$ admissible function $\underline{u}$. 
 	Then for any $\psi\in C^\infty(M)$, 
we can find a constant $c$ and a   smooth 
 	admissible function $u$ with $\sup_M u=0$ such that
 	\begin{equation}
 		\label{equ2-solution}
 		\begin{aligned}
 	f(\lambda[\omega^{-1}(\chi+\sqrt{-1}\partial\overline{\partial}u)])=\psi+c.
 	\nonumber
 			\end{aligned}
 	\end{equation} 
 \end{theorem}

When $f$  satisfies
\eqref{addistruc} and \eqref{elliptic}  additionally, the theorem can be found in \cite{Phong-To2017}. 
  The existence result includes among others the celebrated theorem of Yau \cite{Yau78} and the extension of Tosatti-Weinkove \cite{TW10} to Hermitian case,  with important contributions from \cite{Cherrier1987,Guan2010Li,Hanani1996}.
  Some  more special cases embraced by  Theorem \ref{thm1-existence}  are the complex $k$-Hessian equations and  Monge-Amp\`ere equation for $(n-1)$-PSH functions.  
      In addition, the above structural 
      conditions 
    allow the function 
 \begin{equation}\label{Guan-2}	
 	\begin{aligned}
 	f =\frac{\sigma_{k+1}}{\sigma_k} +\sum_{j=1}^k \beta_j\log{\sigma_j}, \,\, \Gamma=\Gamma_k, \,\, \forall\beta_j\geqslant0, \,  \sum_{j=1}^{k}  \beta_j >0,	 
 	\end{aligned}
 \end{equation}  
 where  $\sigma_k$ is the $k$-th elementary symmetric function,  
 and $\Gamma_k$ is  the $k$-G{\aa}rding cone.   
   
   The Hessian equation  with   \eqref{Guan-2} was  considered by \cite{Guan-Dirichlet}.
   Notice that the function does not satisfy \eqref{addistruc}.   More such  functions can be found in literature \cite{Guan-Dirichlet,Guan2021Zhang}.

 \subsection{The Dirichlet problem}
 The second subject of this paper is to study the Dirichlet problem. 
 To do this, we require the notion of \textit{admissible} subsolution.
 \begin{definition}
 	We say that a $C^2$ function $u$ is an admissible subsolution of  Dirichlet problem~\eqref{mainequ-Dirichlet}, if $u$ is an admissible function subject to
 \begin{equation}
 	\label{subsolution1}
 	\begin{aligned}
 		F(\chi+\sqrt{-1}\partial\overline{\partial}\underline{u})\geqslant \psi \mbox{ in } \bar M, \,\, \underline{u}=\varphi 
 		\mbox{ on } \partial M.
 	\end{aligned}
 \end{equation} 
\end{definition}

  The concept of subsolution was imposed  by \cite{Guan1998The,Guan1993Spruck,Hoffman1992Boundary} as a vital tool  to deal with  second order boundary estimate for Dirichlet problem of Monge-Amp\`ere equation.
  Also the perspective of  subsolutions has a great advantage in applications to  certain geometric problems  that it relaxes restrictions to the shape of boundary; such as   
 Guan's   \cite{GuanP2002The} proof of  the Chern-Levine-Nirenberg problem  
 and Chen's 
 \cite{Chen}   work on Donaldson's conjecture 
   (see also \cite{Guan2009Zhang}  for   Sasaki geometry counterpart).

  
In this paper, we solve the Dirichlet problem under subsolution assumption.

 \begin{theorem}	\label{thm1-1}
	Let $(\bar M, \omega)$ be a compact Hermitian manifold with smooth boundary. Let $f$ 
	satisfy the hypotheses 
	\eqref{concave}  and \eqref{unbounded-1}. Assume the 
	data $\varphi \in C^{\infty}(\partial M)$ and $\psi \in C^{\infty}(\bar{M})$ satisfy~\eqref{nondegenerate} and  support  a $C^{2,1}$-admissible subsolution. Then the Dirichlet problem~\eqref{mainequ-Dirichlet} possesses a unique smooth admissible solution $u \in C^{\infty}(\bar{M})$.
\end{theorem}

We emphasize in Theorem \ref{thm1-1} 
the importance of allowing manifolds with boundary of arbitrary geometry, assuming  existence of  subsolution.
%
	To achieve this,
a specific problem that we have in mind is to establish 
boundary estimate of the form  
\begin{equation}
	\label{bdy-sec-estimate-quar1}
	\begin{aligned}
	\sup _{\partial M}  |\partial\overline{\partial}u|_\omega \leqslant  C\left(1+\sup _{M}|\nabla u|_\omega^{2}\right).
	\end{aligned}
\end{equation}
The quantitative boundary estimate \eqref{bdy-sec-estimate-quar1}  for the Dirichlet problem of complex Monge-Amp\`ere equation was first established by Chen \cite{Chen} on some product manifolds and later   by Phong-Sturm \cite{Phong-Sturm2010} on  manifolds with holomorphically flat boundary. Subsequently, such restrictions to boundary were dropped by Boucksom \cite{Boucksom2012}.
 More recently,   the estimate \eqref{bdy-sec-estimate-quar1} for the Dirichlet problem of complex $k$-Hessian equations
was obtained by 
Collins-Picard \cite{Collins2019Picard}.
However, their methods rely  heavily 
  on the  specific 
  structure of the Monge-Amp\`ere and $k$-Hessian operators, which cannot be further  adapted to treat general equations.

Below, based on Lemma \ref{yuan's-quantitative-lemma} we establish the quantitative boundary estimate.

\begin{theorem}
	\label{thm1-bdy}
	
	Let $(\bar M, \omega)$ and $f$ be as in Theorem~\ref{thm1-1}. Assume that $\varphi \in C^{3}(\partial M)$ and $\psi \in C^{1}(\bar{M})$ satisfy~\eqref{nondegenerate} and support a $C^{2}$-admissible subsolution $\underline{u}$. Then any admissible solution $u \in C^{3}(M) \cap C^{2}(\bar{M})$ to Dirichlet problem~\eqref{mainequ-Dirichlet} satisfies
	\[
	\sup _{\partial M}  |\partial\overline{\partial}u|_\omega \leqslant  C\left(1+\sup _{M}|\nabla u|_\omega^{2}\right),
	\]
	where $C$ is a 
	positive constant depending on $|\varphi|_{C^{3}(\bar{M})}$, $|\nabla u|_{C^{0}(\partial M)}$, $|\underline{u}|_{C^{2}(\bar{M})}$, $|\psi|_{C^{1}(M)}$, $\partial M$ up to third derivatives, and other known data under control. 
	
\end{theorem}

\begin{remark}
	
	As shown by
	Lemma \ref{lemma1-unbound-yield-elliptic}, 
	the unbound and concavity of $f$
	together imply  \eqref{elliptic}.
	Thus we do not impose 
	\eqref{elliptic} in Theorems \ref{thm1-existence}, \ref{thm1-1} and \ref{thm1-bdy}.
	
\end{remark}

In \cite{CNS3}, the condition
 \eqref{unbounded-1} is necessary for second order boundary estimate for double normal derivative; 
 subsequently, it was further removed by Trudinger   \cite{Trudinger95}. However,   their boundary estimates are   not in a quantitative form. 
In order to draw \eqref{bdy-sec-estimate-quar1}, we propose Lemma
  \ref{yuan's-quantitative-lemma}   
to  follow 
the behavior of  the equation near boundary as $u_{n\bar n}$ 
grows  quadratically in terms of $u_{\alpha \bar n}$ $(\alpha<n)$. Together with Lemmas   \ref{keylemma1-yuan3}, \ref{lemma1-bdyestimate} and \ref{lemma2-key}, we may establish   quantitative boundary estimate for double normal derivative.
 Our estimate is new even  for equations in Euclidean space.  
In addition, the boundary estimate is  fairly delicate  in some case; see Theorems \ref{thm2-bdy} and \ref{thm2-bdy-leviflat}.
Such delicate  estimates enable us to study 
 degenerate equation in the sense that
\begin{equation}
 \label{degenerate-RHS}
	\begin{aligned}
		\inf _{M} \psi=\sup _{\partial \Gamma} f.
	\end{aligned}
\end{equation}  which will be shown in Sections \ref{sec1-degenerate} and \ref{sec6}.
In this   case we further assume
\begin{equation}
 	\label{continuity1}
	\begin{aligned}
		f\in C^\infty(\Gamma)\cap C(\bar{\Gamma}),
		\mbox{ where } \bar{\Gamma} =\Gamma\cup\partial\Gamma.
	\end{aligned}
\end{equation}

In conclusion, in the first part of this paper we derive global second   estimate and especially extend the blow-up argument 
 to 
 generic case. 
 As a result, under the optimal structural conditions,  we establish   $C^{2,\alpha}$-estimate  for general  Hessian type
 fully nonlinear elliptic equations 
on closed Hermitian manifolds. 
Given the assumptions \eqref{elliptic}-\eqref{nondegenerate} and the smoothness condition $\psi\in C^\infty(M)$, 
our result indicates that the sole impediment to achieving $C^{2,\alpha}$ estimate on closed manifolds arises from the existence of a $\mathcal{C}$-subsolution. 
In the second part  we first
derive the quantitative boundary estimate  and then solve the Dirichlet problem,
 possibly with degenerate right-hand side.
No matter the equations are degenerate or not,  
the 
existence results for Dirichlet problem of general fully nonlinear elliptic equations are rarely known until this work.  
Moreover, we construct subsolutions on products, which are of numerous interests from complex geometry and analysis. 
 
In addition, our method   applies with almost no change to real Hessian equations  \begin{equation}\label{equation2-realhessian}\begin{aligned}	f(\lambda[g^{-1}(\nabla^2u+A)])=\psi. 
			\nonumber
	\end{aligned}\end{equation}

In the previous draft \cite{yuan-regular-DP}   of this paper, the Dirichlet problem \eqref{mainequ-Dirichlet} was investigated under the extra assumption \eqref{addistruc}.

 The paper is organized as follows. In Section \ref{sec-preliminaries}  
 we  present  notation used in the paper. 
 In Section \ref{sec3-keylemma} we prove key lemmas for  \textit{a priori} estimates.
 In Sections  \ref{sec-second-estimate} and \ref{sec-gradient-estimate} we derive   \textit{a priori} estimate for equations on closed manifolds. More precisely, 
 we
 first derive second order estimate in Section \ref{sec-second-estimate} and 
 then  prove the gradient bound by extending  
 blow-up argument in Section \ref{sec-gradient-estimate}.   
 In Section \ref{sec1-type2} we first construct admissible functions using Morse theory and then solve the equations with type 2 cones.
 In Sections \ref{sec1-Bdy-setup} and \ref{sec1-Bdy-estimate} we establish   quantitative boundary estimate for Dirichlet problem.  
 In Section \ref{sec1-degenerate} we study  Dirichlet problem for degenerate equations. 
 In Section \ref{sec6} the subsolutions are constructed when 
 $M$ is a product of a closed complex manifold 
 with a compact Riemann surface with boundary. 
 In addition, we study Dirichlet problem on such products with less regularity assumptions on boundary.
  In Section \ref{sec-proofofquantitativelemma1}
  we complete the proof of Lemma \ref{yuan's-quantitative-lemma}.

\medskip
\section{Notation}
\label{sec-preliminaries}

Throughout this paper,  
we use the notation as follows:
\begin{itemize}
	\item   $\vec{\bf 1}=(1,\cdots,1)\in \mathbb{R}^n$. 
	
	  	\item  $\delta_{\psi,f}=\inf_{M}\psi-\sup_{\partial \Gamma} f.$
	  	
	\item 
 $\partial\Gamma^\sigma=\{\lambda\in\Gamma: f(\lambda)=\sigma\}, \, \Gamma^\sigma=\{\lambda\in\Gamma: f(\lambda)>\sigma\},$
 $\bar \Gamma^\sigma=\Gamma^\sigma\cup\partial\Gamma^\sigma.$ 
 
 \item $\Gamma_\infty=\{(\lambda_{1}, \cdots, \lambda_{n-1}): (\lambda_{1}, \cdots, \lambda_{n-1},R) \in\Gamma\mbox{ for some } R>0\}.$

\item Given a $C^2$ function $w$, for simplicity we denote $$\mathfrak{g}[w]\equiv \chi+\sqrt{-1}\partial\overline{\partial} w.$$   
In particular,
  we write 
$\mathfrak{g}=\mathfrak{g}[u]$ and $\mathfrak{\underline{g}}= 
\mathfrak{g} [\underline{u}]$ for  solution $u$ and  subsolution $\underline{u}$, respectively. 
%
 In addition, denote
 $  
 \lambda=\lambda (\omega^{-1}\mathfrak{g}), 
 \,\,
 \underline{\lambda}=\lambda (\omega^{-1}\mathfrak{\underline{g}}). $
 
\item 
In local coordinates $z=(z^1,\cdots,z^n)$, $z^i=x^i+\sqrt{-1}y^i$
\begin{align*}
	\mathfrak{g}
	= \sqrt{-1}\mathfrak{g}_{i\bar j} dz^i\wedge d\bar z^j,  \,\,
	\mathfrak{\underline{g}} 
	= 
	\sqrt{-1}\mathfrak{\underline{g}}_{i\bar j} dz^i\wedge d\bar z^j.  
\end{align*}


\item  In  computation we use derivatives with respect to the Chern connection $\nabla$ of $\omega$,
and write
$\partial_{i}=\frac{\partial}{\partial z^{i}}$, 
$\overline{\partial}_{i}=\frac{\partial}{\partial \bar z^{i}}$,
$\nabla_{i}=\nabla_{\frac{\partial}{\partial z^{i}}}$,
$\nabla_{\bar i}=\nabla_{\frac{\partial}{\partial \bar z^{i}}}$,  $\nabla_{x^i}=\nabla_{\frac{\partial}{\partial x^i}}$,  $\nabla_{y^i}=\nabla_{\frac{\partial}{\partial y^i}}$, $\nabla_{t^k}=\nabla_{\frac{\partial}{\partial t^k}}$.
Given a function $v$,  we denote
\begin{equation} \label{formula-1}\begin{aligned}
		v_i=\,&
		\partial_i v,  
		\mbox{  } v_{\bar i}= 
		\partial_{\bar i} v,  
		\, v_{x^i}=\nabla_{x^i} v, \,  v_{y^i}=\nabla_{y^i} v, \,
		 v_{t^k}=\nabla_{t^k} v, \, 	 v_{t^kt^l}=\nabla_{t^l}\nabla_{t^k} v,
		\\
		v_{i\bar j}= \,& 
		\partial_i\overline{\partial}_j v, 
		\,
		v_{ij}=
		\partial_j \partial_i v -\Gamma^k_{ji}v_k,    
		v_{i\bar j k} 
		 =   \partial_k v_{i\bar j} -\Gamma_{ki}^l v_{l\bar j},  \cdots, \mbox{etc}, 
\end{aligned}\end{equation}
where $\Gamma_{ij}^k$ are the Christoffel symbols
defined  by 
$\nabla_{\frac{\partial}{\partial z^i}} \frac{\partial}{\partial z^j}=\Gamma_{ij}^k \frac{\partial}{\partial z^k}.$  
 Moreover, for 
  $\tilde{\chi}=\sqrt{-1}\tilde{\chi}_{i\bar j}dz^i\wedge d\bar z^j$  we write $$\tilde{\chi}_{i\bar j k} =\nabla_k \tilde{\chi}_{i\bar j}, \,\, 	\tilde{\chi}_{i\bar j \bar k} =\nabla_{\bar k} \tilde{\chi}_{i\bar j}, 
 \,\,
 \tilde{\chi}_{i\bar jk\bar l}=\nabla_{\bar l}\nabla_k\tilde{\chi}_{i\bar j}.$$

\item  $F^{i\bar j}=\frac{\partial F(\mathfrak{g})}{\partial \mathfrak{g}_{i\bar j}},$  $F^{i\bar j, k\bar l}=\frac{\partial^2 F(\mathfrak{g})}{\partial \mathfrak{g}_{i\bar j}\partial \mathfrak{g}_{k\bar l}}.$ 
Moreover, 
 $F^{i\bar j}\mathfrak{g}_{i\bar j}=\sum_{i=1}^n f_i(\lambda)\lambda_i,$ $ F^{i\bar j}g_{i\bar j}=\sum_{i=1}^n f_i(\lambda).$
In view of \cite[Lemma 6.2]{CNS3} (with proper order  $\underline{\lambda}_1,\cdots, \underline{\lambda}_n$),
we have
\begin{equation} 
	\label{L1-u-ubar}
	\begin{aligned}
		F^{i\bar j}(\underline{\mathfrak{g}}_{i\bar j}-\mathfrak{g}_{i\bar j}) 
		\geqslant    \sum_{i=1}^n f_i(\lambda) (\underline{\lambda}_i-\lambda_i).	
	\end{aligned}
\end{equation}    

	\item $\mathcal{L}$ stands for the linearized operator of    \eqref{equ1-main}. Locally, it is given by
	\begin{equation}
		\label{linearoperator21}
		\begin{aligned}
			\mathcal{L}v
			=  F^{i\bar j}v_{i\bar j}. 
		\end{aligned}
	\end{equation}  
\end{itemize}

 \medskip

\section{Key lemmas}
\label{sec3-keylemma}

\subsection{A lemma for subsolution}

We prove a lemma in order to 
derive second estimate. The lemma is a refinement of  \cite[Theorem 2.18]{Guan12a}, \cite[Lemma 1.9]{Guan-Dirichlet} and
\cite[Proposition 5]{Gabor}. 

\begin{lemma}
	\label{lemma5-key}
	Assume  \eqref{elliptic} and \eqref{concave} hold.
	Fix $\sigma\in (\sup_{\partial \Gamma}f,\sup_\Gamma f)$. Suppose that $\mu\in\mathbb{R}^{n}$ 
	satisfies for some  $\delta$, $R>0$ that
	\begin{equation}
		\label{set2-bound-2}
		\begin{aligned} 
			(\mu-2\delta \vec{\bf 1}+\Gamma_n) \cap \partial\Gamma^\sigma \subset B_R(0). 
		\end{aligned}
	\end{equation}
	Then there is a constant $\varepsilon$ depending   on  $\sigma$, $\mu$, $\delta$, $R$, $f$  (explicitly see \eqref{const-lemma-choice1})  such that if $\lambda\in\partial\Gamma^\sigma$, then either
	\begin{equation}
		\label{inequ-case1-2}
		\begin{aligned}
			\sum_{i=1}^n f_i(\lambda) (\mu_i-\lambda_i) \geqslant \varepsilon \left(1+\sum_{i=1}^n f_i(\lambda)+\big|\sum_{i=1}^n f_i(\lambda)\lambda_i\big|\right),
		\end{aligned}
	\end{equation}
	or 
	\begin{equation}
		\label{inequ-case2-2}
		\begin{aligned}
			f_i(\lambda) \geqslant \varepsilon \left(1+\sum_{j=1}^n f_j(\lambda) +\big|\sum_{j=1}^n f_j(\lambda)\lambda_j\big| \right), \, \forall 1\leqslant i\leqslant n.
		\end{aligned}
	\end{equation}
	
\end{lemma}

\begin{proof}  
	
	Our proof 
	is based only  on  the
	concavity assumption. 
	Denote $\tilde{\mu}=\mu-\delta \vec{\bf 1}$. By \eqref{set2-bound-2},
	there exists a  positive constant $R_0$ depending only on $\tilde{\mu}$ and $R$ such that 
	\begin{equation}
		\label{key1-proof-lemma3.2}
		\begin{aligned}  
			f(\tilde{\mu}+R_0\vec{\bf e}_i)>\sigma, \,\, \tilde{\mu}+R_0\vec{\bf e}_i\in\Gamma, \,\, \forall 1\leqslant i\leqslant n, 
		\end{aligned}
	\end{equation} 
	where $\vec{\bf e}_i$ 
	is the $i$-th standard basis vector in $\mathbb{R}^n$. 
	
	To see this, we pick some $R_0>0$ such that $ \tilde{\mu}_i+R_0>R$, $\forall 1\leqslant i\leqslant n.$ Suppose by contradiction that $\tilde{\mu}+R_0 \vec{\bf e}_{i'} \in \mathbb{R}^n\setminus \Gamma^\sigma$ for some $i'$. Then we  take suitable (small) constant $0<\epsilon<R_0$ such that   $\tilde{\mu}+(R_0-\epsilon) \vec{\bf e}_{i'} \in \mathbb{R}^n\setminus \bar \Gamma^\sigma$ and $\tilde{\mu}_{i'}+R_0-\epsilon>R$. From this we know that there is some $t>0$ such that $\tilde{\mu}+(R_0-\epsilon) \vec{\bf e}_{i'}+t\vec{\bf1}\in\partial\Gamma^\sigma$, which contradicts to \eqref{set2-bound-2}.  
	
		
	From \eqref{key1-proof-lemma3.2} we can see that
	there is a small constant $0<\epsilon_1\ll1$  depending only on $\tilde{\mu}$ and $\delta$ such that 
	\begin{equation} 
		\label{key1-proof-lemma3.2-2}
			f((1+\epsilon_1)\tilde{\mu}+R_0\vec{\bf e}_i)>\sigma, \, 	f((1-\epsilon_1)\tilde{\mu}+R_0\vec{\bf e}_i)>\sigma, \, \forall 1\leq i\leq n.
	\end{equation}  
Fix such $\epsilon_1$,  we take  
\begin{equation}
	\label{def-c}
	\begin{aligned}
	\frac{1}{1+c}=
	\begin{cases}
		1-\epsilon_1 \,& \mbox{ if } \sum_{i=1}^n f_i(\lambda)\lambda_i\geq0,\\
			1+\epsilon_1 \,& \mbox{ otherwise.}
	\end{cases}
\end{aligned}
\end{equation}
Denote 
$\delta_0=\min\left\{\underset{1\leqslant i\leqslant n}{\min}\{f((1+\epsilon_1)\tilde{\mu}+R_0\vec{\bf e}_i)-\sigma\}, \underset{1\leqslant i\leqslant n}{\min}\{ f((1-\epsilon_1)\tilde{\mu}+R_0\vec{\bf e}_i)-\sigma\}\right\}.$  

	Below we prove that the conclusion holds for 
\begin{equation}
	\label{const-lemma-choice1}
	\begin{aligned}
		\varepsilon= \min\left\{ \frac{\delta_0}{2R_0}, \frac{\delta(1-\epsilon_1)}{2R_0}, \frac{\epsilon_1}{2R_0},
		\frac{\delta_0}{2(1+\epsilon_1)},
		\frac{\delta}{2},\frac{\epsilon_1}{2(1+\epsilon_1)}\right\}.
	\end{aligned}
\end{equation}

	Without loss of generality,  
	assume 
	$\lambda_1\leqslant \cdots\leqslant\lambda_n$. Then $f_1(\lambda)\geqslant \cdots \geqslant f_n(\lambda).$  
	If 
	$$f_n(\lambda)>\frac{\delta_0}{2R_0}  	
	+\frac{\delta}{2R_0(1+c)}\sum_{i=1}^n f_i(\lambda)
 +\frac{c}{2R_0(1+c)}\sum_{i=1}^n f_i(\lambda)\lambda_i, $$ 
	then by \eqref{def-c} $$f_n(\lambda)>\frac{\delta_0}{2R_0}  	
	+\frac{\delta(1-\epsilon_1)}{2R_0 }\sum_{i=1}^n f_i(\lambda)
	+\frac{\epsilon_1}{2R_0}
	\big|\sum_{i=1}^n f_i(\lambda)\lambda_i\big|.$$ 
	
	From now on we assume  $$ \frac{\delta_0}{2}  
	+\frac{\delta}{2(1+c)}\sum_{i=1}^n f_i(\lambda)		+\frac{c}{2(1+c)}\sum_{i=1}^n f_i(\lambda)\lambda_i	\geqslant R_0f_n(\lambda). $$
	Note $\mu=\tilde{\mu}+\delta \vec{\bf 1}.$
	By the concavity of $f$, using \eqref{concave-1} below we have
	\begin{equation} 
		\begin{aligned}
		\frac{1}{1+c}	\sum_{i=1}^n f_i(\lambda) (\mu_i-\lambda_i)    
			=\,&   	\sum_{i=1}^n f_i(\lambda) (\frac{1}{1+c}\tilde{\mu}_i-\lambda_i)  +\frac{\delta}{1+c}\sum_{i=1}^n f_i(\lambda)  +\frac{c}{1+c}\sum_{i=1}^n f_i(\lambda)\lambda_i 
		 \\=\,&   	\sum_{i=1}^{n-1} f_i(\lambda) (\frac{1}{1+c}\tilde{\mu}_i-\lambda_i)   	+ f_n(\lambda) (\frac{1}{1+c}\tilde{\mu}_n
		 +R_0-\lambda_n) 	 
		 \\\,&  
		 	-R_0f_n(\lambda)
	+\frac{\delta}{1+c}\sum_{i=1}^n f_i(\lambda)	 +\frac{c}{1+c}\sum_{i=1}^n f_i(\lambda)\lambda_i
			\\ 
			\geqslant\,& 
			f(\frac{1}{1+c}\tilde{\mu}+R_0\vec{\bf e}_n)-f(\lambda)
			+\frac{\delta}{1+c}\sum_{i=1}^n f_i(\lambda)	
			\\\,& 
			+\frac{c}{1+c}\sum_{i=1}^n f_i(\lambda)\lambda_i
				-R_0f_n(\lambda)
			\\ 
			\geqslant \,& 
		 \delta_0  
		 +\frac{\delta}{1+c}\sum_{i=1}^n f_i(\lambda)	+\frac{c}{1+c}\sum_{i=1}^n f_i(\lambda)\lambda_i		-R_0f_n(\lambda)
		 \\\geqslant \,&
		  \frac{\delta_0}{2}  	
		 +\frac{\delta}{2(1+c)}\sum_{i=1}^n f_i(\lambda)+\frac{c}{2(1+c)}\sum_{i=1}^n f_i(\lambda)\lambda_i.   \nonumber
		\end{aligned}
	\end{equation}  
That is
\begin{align*}
	\sum_{i=1}^n f_i(\lambda) (\mu_i-\lambda_i)  
	\geqslant \,& 
	\frac{\delta_0(1+c)}{2}
	+\frac{\delta}{2}
	\sum_{i=1}^n f_i(\lambda)
	+\frac{(1+c)\epsilon_1}{2}\big|\sum_{i=1}^n f_i(\lambda)\lambda_i\big|
	\\	\geqslant \,& 
	\frac{\delta_0}{2(1+\epsilon_1)}
	+\frac{\delta}{2}
	\sum_{i=1}^n f_i(\lambda)
	+\frac{ \epsilon_1}{2(1+\epsilon_1)}\big|\sum_{i=1}^n f_i(\lambda)\lambda_i\big|.
\end{align*}

\end{proof}

\begin{remark}

	In contrast to previous works  \cite{Guan12a, Guan-Dirichlet, Gabor}, Lemma \ref{lemma5-key} stands out prominently due to its inclusion of  the new terms of the form $1+\big|\sum  f_i \lambda_i\big|$
  in both   \eqref{inequ-case1-2} and \eqref{inequ-case2-2}. 
		This  feature enhances the scope 
	of the results. 
\end{remark}


As a consequence, we have
\begin{lemma} 
	\label{lemma3-key}
	Let $(\bar M,\omega)$ be a compact Hermitian manifold of complex dimension $n\geqslant2$ (possibly with boundary).
	Suppose  that there exists a $\mathcal{C}$-subsolution $\underline{u}\in C^{2}(\bar M)$ to the equation \eqref{equ1-main}.
	Then there exists a positive constant  $\varepsilon$, independent of $(\delta_{\psi,f})^{-1}$,
	such that  either  
	\begin{equation}
		\label{inequ-case1}
		\begin{aligned}
			F^{i\bar j}(\underline{\mathfrak{g}}_{i\bar j}-\mathfrak{g}_{i\bar j})
			\geqslant \varepsilon(1+ F^{i\bar j}g_{i\bar j}+ \big|F^{i\bar j}\mathfrak{g}_{i\bar j}\big|)  
		\end{aligned}
	\end{equation}
	or
	\begin{equation}
		\label{inequ-case2}
		\begin{aligned}
			F^{i\bar j}\geqslant \varepsilon (1+F^{p\bar q}g_{p\bar q}
			+\big|F^{p\bar q}\mathfrak{g}_{p\bar q}\big|)g^{i\bar j}.
		\end{aligned}
	\end{equation}
\end{lemma}

\begin{proof}

	Since $\underline{u}$ is a $\mathcal{C}$-subsolution,   $(\underline{\lambda}  +\Gamma_n) \cap \partial\Gamma^{\psi(z)}$ is bounded for any $z\in\bar M$.
	 Thus, one
	can choose $\delta$ sufficiently small such that  for some $R>0$,
	\begin{equation}
		\label{set2-bound}
		\begin{aligned} 
			(\underline{\lambda}  -2\delta \vec{\bf 1}+\Gamma_n) \cap \partial\Gamma^{\psi(z)} \subset B_R(0), \, \, \forall z\in\bar M.    \nonumber
		\end{aligned}
	\end{equation} 
Combining Lemma \ref{lemma5-key}, \eqref{L1-u-ubar} and $F^{i\bar j}\mathfrak{g}_{i\bar j}=\sum  f_i(\lambda)\lambda_i$, we have either \eqref{inequ-case1} or \eqref{inequ-case2}.
From the choice of $\varepsilon$ in \eqref{inequ-case1-2}, we see that the constant does not depend  on $(\delta_{\psi,f})^{-1}.$
	
\end{proof}

\begin{remark} 
	Note that from concavity assumption \eqref{concave} we can infer that
	\begin{equation}\label{concave-1} 	\begin{aligned}	\sum_{i=1}^n f_i(\lambda)(\mu_i-\lambda_i)\geqslant f(\mu)-f(\lambda), \,\, \forall \lambda, \, \mu\in\Gamma.	\end{aligned}\end{equation}
	In comparison, Lemmas \ref{lemma5-key} 
	and \ref{lemma3-key}
 more fully exploit the information about concavity.
  
  \end{remark}

\subsection{Monotonicity}

The following lemma asserts that, together with the unbounded condition, the  concavity assumption may imply the monotonicity. 

\begin{lemma}   
	\label{lemma1-unbound-yield-elliptic}
	If $f$ satisfies \eqref{concave} and 
	\eqref{unbounded-1} in $\Gamma$,
	then   \eqref{elliptic} holds.
	
\end{lemma}
\begin{proof}
	Suppose $\lambda_1\leqslant\cdots\leqslant\lambda_n$. Then $f_1(\lambda)\geqslant\cdots\geqslant f_n(\lambda)$.
	In view of the concavity and unbound of $f$, by setting $t\gg1$  we know 
	\begin{equation}
		\begin{aligned}
			f_n(\lambda)\geqslant\frac{f(\lambda_1,\cdots,\lambda_{n-1},\lambda_n+t)-f(\lambda)}{t}>0. \nonumber
		\end{aligned}
	\end{equation} 
\end{proof}

\subsection{Lemmas for blow-up argument}
\label{sec-gradient-estimate-sub1}
Let $\Gamma_{\mathcal{G}}^{f}$ be as in  \eqref{component1}. 
First, we verify that
\begin{lemma} 
	\label{lemma1-key}
	Suppose that $f$ satisfies \eqref{elliptic}  and \eqref{concave}  in $\Gamma$. Then
	$\mathring{\Gamma}_{\mathcal{G}}^{f}$ is an  open symmetric convex cone with vertex at origin and  $\Gamma_n\subseteq\mathring{\Gamma}_{\mathcal{G}}^{f}\subseteq\Gamma$. 
\end{lemma}

\begin{proof}
It suffices to consider  $\Gamma_{\mathcal{G}}^{f}$.
	By the definition, $\Gamma_{\mathcal{G}}^{f}$ is a cone with vertex at origin and 
	$\Gamma_{\mathcal{G}}^{f}\subseteq\Gamma$.
	By \eqref{elliptic}, $\Gamma_n\subseteq \Gamma_{\mathcal{G}}^{f}.$
	From the symmetry of $f$ and $\Gamma$, it follows that $\Gamma_{\mathcal{G}}^{f}$ is symmetric. By the concavity of $f$ and convexity of $\Gamma$, for all $\lambda,\, \mu\in\Gamma, $
	\begin{align*}
		f(t(s\lambda+(1-s)\mu)) \geqslant sf(t\lambda)+(1-s)f(t\mu), \,   \forall 0\leqslant s\leqslant 1, \forall  t>0.
	\end{align*}
	This implies that $\Gamma_{\mathcal{G}}^{f}$ is convex.  
\end{proof}

The cone $\Gamma_{\mathcal{G}}^{f}$ is a key ingredient for 
analyzing the limiting function of blow-up sequence.
 We start with an observation 
 \begin{equation}
 	\label{010} 
 	\begin{aligned}
 		\sum_{i=1}^n f_i(\lambda)\mu_i \geqslant \limsup_{t\rightarrow+\infty} f(t\mu)/t, 
 		\,\, \forall \lambda, \, \mu\in\Gamma.  
 \end{aligned}\end{equation} 
 This inequality follows from  the concavity of $f$. 
  Now we come to our key lemma. 
\begin{lemma}
	\label{lemma2-key}
	Suppose that $f$ satisfies  
	\eqref{concave} in $\Gamma$. Then 
	\begin{equation}
		\label{lemma315} 
		\begin{aligned}
			f(\lambda+\mu)\geqslant f(\lambda), 
			\,\, \forall \lambda\in\Gamma, \, 
			\forall	\mu\in  \Gamma_{\mathcal{G}}^{f}.
		\end{aligned}
	\end{equation} 
\end{lemma}
\begin{proof}
 
From \eqref{010},   for any $\lambda\in\Gamma$ and $\mu\in 	\Gamma_{\mathcal{G}}^{f}$, we have
$\sum_{i=1}^n f_i(\lambda)\mu_i \geqslant 0.$  
Combining  \eqref{concave-1},	
we obtain \eqref{lemma315}.

\end{proof}

Moreover, we have some criteria for the cone $\Gamma_{\mathcal{G}}^{f}.$
\begin{lemma}	  
 	\label{lemma3.4-strongerversion}
Given $\mu\in\Gamma$.
Suppose that $f$ satisfies 
\eqref{concave} in $\Gamma$. Then the following are mutually 
equivalent.
\begin{enumerate}
	\item 
	$\underset{t\to+\infty}{\lim} f(t\mu)>-\infty$, i.e., $\mu\in\Gamma_{\mathcal{G}}^{f}$.
	
	\item  $\underset{t\rightarrow+\infty}{\limsup} f(t\mu)/t\geqslant0.$
	
	\item  $\sum_{i=1}^n f_i(\lambda)\mu_i\geqslant 0$, $\forall \lambda\in\Gamma.$
		
	\item $f(\lambda+\mu)\geqslant f(\lambda)$, $\forall \lambda\in\Gamma$.
\end{enumerate} 
\end{lemma}
\begin{proof}
	Obviously,   $\mathrm{(3)}$ $\Rightarrow$  $\mathrm{(1)}$ $\Rightarrow$  $\mathrm{(2)}$.
	By \eqref{010}, $\mathrm{(2)}$ $\Rightarrow$  $\mathrm{(3)}$. From  \eqref{concave-1},  $\mathrm{(3)}$ $\Leftrightarrow$  $\mathrm{(4)}$.

\end{proof}

 \begin{remark}In Lemmas \ref{lemma2-key} and  \ref{lemma3.4-strongerversion},	the condition \eqref{elliptic} is not needed. \end{remark}

 \begin{remark}
 	
 If $f$ is a homogeneous function  of degree one, 	i.e., 
 	 $f(t\lambda)=tf(\lambda)$ $\forall \lambda\in\Gamma, t>0$,  
 	using \eqref{010}  
 	we then obtain a G{\aa}rding  type inequality 
 	\begin{equation}
 		\sum_{i=1}^n	f_i(\lambda)\mu_i\geqslant  f(\mu), \mbox{ } \forall\lambda, \, \mu\in\Gamma. \nonumber
 	\end{equation}
 \end{remark}

 \subsubsection*{\bf Criteria for 
 	the	function  satisfying \eqref{addistruc}}

According to Lemma \ref{lemma3.4-strongerversion},   the  assumption \eqref{addistruc}
 is equivalent to one of the following conditions: 
 \begin{equation}
	\label{addistruc-0}
	\begin{aligned}
		\text {For any } \lambda \in \Gamma,\mbox{ } \lim _{t \rightarrow+\infty} f(t \lambda)>-\infty,
	\end{aligned}
\end{equation}
\begin{equation}
	\label{addistruc-1}
	\begin{aligned}
		\mbox{For any } \lambda\in\Gamma, \mbox{  } 
		\limsup_{t\rightarrow+\infty} f(t\lambda)/t\geqslant 0,
	\end{aligned}
\end{equation}  
	\begin{equation}\label{addistruc-10}   \begin{aligned}
		f(\lambda+\mu)>f(\lambda), \mbox{ }\forall \lambda, \, \mu\in \Gamma,
\end{aligned} \end{equation}  
	\begin{equation}
		\label{addistruc-4}
		\begin{aligned}
			\sum_{i=1}^n f_i(\lambda)\mu_i>0, \mbox{   } \forall \lambda, \, \mu\in \Gamma.
		\end{aligned}
	\end{equation}

\begin{lemma}
	\label{lemma3.4}
	In the presence of  \eqref{elliptic} and \eqref{concave}, the following
	are 
	equivalent.
	\begin{enumerate}
		\item $f$ satisfies  \eqref{addistruc}.
		\item $f$ satisfies  \eqref{addistruc-0}, i.e.,  $\Gamma_{\mathcal{G}}^{f}=\Gamma$.
		\item  $f$ satisfies  \eqref{addistruc-1}.	
		\item  $f$ satisfies \eqref{addistruc-10}.	
		\item  $f$ satisfies \eqref{addistruc-4}.

	\end{enumerate}
\end{lemma}

\subsection{A quantitative lemma}
\label{refinementofCNS}

One of key ingredients for quantitative boundary estimate
for double normal derivative is how to follow the track of 
the eigenvalues of 
 $\left(\mathfrak{g}_{i\bar j}\right)$
as $\mathfrak{g}_{n\bar n}$ 
tends to infinity. To this end, we propose the following lemma.  
(Since the proof  is somewhat complicated and long, we will leave it to Section \ref{sec-proofofquantitativelemma1}).

\begin{lemma}
	\label{yuan's-quantitative-lemma}
	Let $A$ be an $n\times n$ Hermitian matrix
	\begin{equation}\label{matrix3}\left(\begin{matrix}
			d_1&&  &&a_{1}\\ &d_2&& &a_2\\&&\ddots&&\vdots \\ && &  d_{n-1}& a_{n-1}\\
			\bar a_1&\bar a_2&\cdots& \bar a_{n-1}& \mathrm{{\bf a}} \nonumber
		\end{matrix}\right)\end{equation}
	with $d_1,\cdots, d_{n-1}, a_1,\cdots, a_{n-1}$ fixed, and with $\mathrm{{\bf a}}$ variable.
	Denote the eigenvalues of $A$ by $\lambda=(\lambda_1,\cdots, \lambda_n)$.
	Let $\epsilon$ be a fixed positive constant.
	Suppose that  the parameter $\mathrm{{\bf a}}$ 
	satisfies the quadratic
	growth condition  
	\begin{equation}
		\begin{aligned}
			\label{guanjian1-yuan}
			\mathrm{{\bf a}}\geqslant \frac{2n-3}{\epsilon}\sum_{i=1}^{n-1}|a_i|^2 +(n-1)\sum_{i=1}^{n-1} |d_i|+ \frac{(n-2)\epsilon}{2n-3}.
		\end{aligned}
	\end{equation}
	Then the eigenvalues 
	(possibly with a proper permutation)
	behave like
	\begin{equation}
		\begin{aligned}
			d_{\alpha}-\epsilon 	\,& < 
			\lambda_{\alpha} < d_{\alpha}+\epsilon, \mbox{  } \forall 1\leqslant \alpha\leqslant n-1, \\ \nonumber
			\,& 	\mathrm{{\bf a}} \leqslant \lambda_{n}
			< \mathrm{{\bf a}}+(n-1)\epsilon. 
		\end{aligned}
	\end{equation}
\end{lemma}

\begin{remark}
This lemma can be viewed as a quantitative version of  \cite[Lemma 1.2]{CNS3}.
\end{remark}

\medskip
 
 
 \section{Second  estimate and proof of Theorem \ref{thm2-second-order}} 
\label{sec-second-estimate}

As above,  denote $\lambda(\omega^{-1}\mathfrak{g})=(\lambda_1,\cdots,\lambda_n)$. Assume  $\lambda_{1}:  M\rightarrow \mathbb{R}$ is the largest eigenvalue and $\lambda_1\geqslant \lambda_2\geqslant\cdots\geqslant\lambda_n$ at each point. 
We consider the quantity 
\begin{equation}
	\label{gqy-81}
	\begin{aligned}
		Q\equiv\lambda_{1}e^{H},
	\end{aligned}
\end{equation}
where the test function has the form
$	H=\phi(|\nabla u|_\omega^{2})+\varphi(u-\underline{u}).$ 
Here  \begin{equation}
	\label{def1-K}
	\begin{aligned} 
		\phi(t)=-\frac{1}{2}\log  (1-\frac{t}{2K}), 
	\,\,
		K=1+\sup_{M}  \left(|\nabla u|_\omega^{2}+ |\nabla (u-\underline{u})|_\omega^{2}\right) 
	\end{aligned}
\end{equation} 
  is the function used in \cite{HouMaWu2010},     
and  $\varphi:[\inf_{M} (u-\underline{u}),+\infty) \rightarrow \mathbb{R}$ is the function with 	$$\varphi(x)=\frac{C_{*}}{1+x -\inf_{M} (u-\underline{u})},$$   
where  $C_{*}\geqslant 1$   
is a large constant to be chosen later.
It is easy to verify that 
\begin{equation}
\label{Prop1-Phi-Psi}
\begin{aligned}
	\phi''=2\phi'^{2}, \,\, (4K)^{-1}<\phi' < (2K)^{-1} \mbox{ for } t\in [0,K],
\end{aligned}
\end{equation} 
\begin{equation}
\label{Prop2-Phi-Psi}
\begin{aligned}
\varphi'(x)=-\frac{ C_{*}}{(1+x -\inf_{M} (u-\underline{u}))^{2}}, \, 
 	\varphi''(x)= \frac{2 C_{*}}{(1+x -\inf_{M} (u-\underline{u}))^{3}}.
\end{aligned}
\end{equation} 

 Suppose $Q$ achieves its maximum at $z_0\in M$. 
It suffices to prove that at $z_0$
\begin{equation}
	\label{maingoal-2nd}
	\begin{aligned}
\lambda_1  \leqslant CK.
\end{aligned}
\end{equation} 
We choose  local coordinates $z=(z^{1},\cdots,z^{n})$ with origin at $z_0$, so that at the origin
$$g_{i\bar j}=\delta_{ij}, \, \mathfrak{g}_{i\bar j}=\delta_{ij}\lambda_{i} \mbox{ and } F^{i\bar j}=\delta_{ij}f_i.$$
Under the local coordinates, $(\lambda_1,\cdots,\lambda_n)$
are the eigenvalues of 
$A=\{A^i_{j}\}=\{g^{i\bar q }\mathfrak{g}_{j\bar q}\}$.

Since the eigenvalues  need to be distinct at $z_0$, the quantity
$Q$ may only
be continuous.
In order to use maximum principle, we employ the perturbation argument of \cite{Gabor}.
Let $B$ be a
diagonal matrix $B^{p}_{q}$ with real entries with $B^{1}_{1}=0$, $B^{n}_{ n}>2 B^{2}_{ 2}$ 
and
$B^n_n<B^{n-1}_{ n-1}<\cdots<B^{2}_{ 2}<0$  being small.
Define the matrix $ \tilde{A}=A+B$ with the eigenvalues $ \tilde{\lambda}=( \tilde{\lambda_{1}},\cdots, \tilde{\lambda_n})$.
At the origin, $ \tilde{\lambda_{1}}=\lambda_{1}=\mathfrak{g}_{1\bar 1}, \, \tilde{\lambda_{i}}=\lambda_{i}+B^{i}_{ i}\mbox{ for } i\geqslant2$
and the eigenvalues of $\tilde{A}$ define $C^{2}$-functions near the origin.

In what follows,
the discussion will be given at the origin. 
Notice that $\tilde{Q}=\tilde{\lambda_{1}}e^{H}$ also achieves its maximum at the same
point $z_{0}$ (we may assume $\lambda_{1}>K$).  
Then  
\begin{equation}
	\label{mp1} 
	\begin{aligned}
		\frac{\tilde{\lambda_{1}}_{,i}}{\lambda_{1}}+H_{i}=0, \, 	\frac{\tilde{\lambda_{1}}_{,\bar i}}{\lambda_{1}}+H_{\bar i}=0,
	\end{aligned} 
\end{equation}
\begin{equation}
	\label{mp4} 
	\begin{aligned} 
		\frac{\tilde{\lambda_{1}}_{,i\bar i}}{\lambda_{1}}
		-\frac{|\tilde{\lambda_{1}}_{,i}|^{2}}{\lambda_{1}^{2}}
		+H_{i\bar i}\leqslant 0. 
	\end{aligned} 
\end{equation}

By straightforward calculations, one obtains
\begin{equation}
	\label{mp2}
	\begin{aligned}
		\tilde{\lambda_{1}},_i\,& =\mathfrak{g}_{1\bar 1 i}+(B^{1}_{1})_{i}.   
	\end{aligned}
\end{equation}
Moreover
\begin{equation}
 \label{gqy-43}
	\begin{aligned}
		\tilde{\lambda_{1}},_{i\bar i}
		= 
		\,&  
	 \sum_{k>1} \frac{|\mathfrak{g}_{k \bar 1 i}|^{2}
			+|\mathfrak{g}_{1\bar k i}|^{2}}{\tilde{\lambda_{1}}
			-\tilde{\lambda_{k}}}
		+2\mathfrak{Re} \sum_{k>1} \frac{\mathfrak{g}_{k\bar 1 i}(B^{1}_{ k})_{\bar i}
			+ \mathfrak{g}_{1\bar k i}(B^{k}_{1})_{\bar i}}{\tilde{\lambda_{1}}-\tilde{\lambda_{k}}}
		\\ \,& 
		+ \mathfrak{g}_{1\bar 1 i\bar i }
		-(B^{1}_{1})_{i \bar i}   
		+\tilde{\lambda_{1}}^{pq,rs}(B^{p}_{ q})_{i}(B^{r}_{ s})_{\bar i},
	\end{aligned}
\end{equation}
where
$\tilde{\lambda}_{1}^{pq,rs}=(1-\delta_{1p})\frac{\delta_{1q}\delta_{1r}\delta_{ps}}{\tilde{\lambda}_{1}- \tilde{\lambda}_{p}}
+ (1-\delta_{1r})\frac{\delta_{1s}\delta_{1p}\delta_{rq}}{\tilde{\lambda}_{1}-\tilde{\lambda}_{r}}, $
see e.g. \cite{Spruck-book1}.

We need to estimate $\tilde{\lambda_{1}}-\tilde{\lambda_{p}}$ near the origin for $p>1$.
Since $\sum \tilde{\lambda_{i}} > \sum_{k=1}^{n} B^{k}_{k}$, and then
$|\tilde{\lambda_{i}}|\leqslant (n-1)\tilde{\lambda_{1}} +  \sum_k |B^{k}_{k}| \mbox{ for all } i$. We assume $\tilde{\lambda}_1$ is large such that
$\tilde{\lambda}_1\geqslant \sum |B^{k}_{k}|$, hence
$(\tilde{\lambda_{1}}-\tilde{\lambda_{k}})^{-1}\geqslant (n\tilde{\lambda}_1)^{-1}$ for $k>1$. 
On the other hand,  
$(\tilde{\lambda_{1}}-\tilde{\lambda_{k}})\leqslant 2(-B_{n}^n)^{-1}$ for $k>1$. 
Using Cauchy-Schwarz inequality,  as in \cite{Gabor} 
   from \eqref{gqy-43} we obtain
\begin{equation}
	\label{xp1}
	\begin{aligned}
		\tilde{\lambda_{1}}_{,i\bar i} \geqslant 
		 \,& 	\mathfrak{g}_{1\bar 1 i \bar i}	+\frac{1}{2n\lambda_{1}}\sum_{k>1} (|\mathfrak{g}_{k\bar 1 i}|^{2}+|\mathfrak{g}_{1\bar k i}|^{2})-1,
	\end{aligned}
\end{equation} 
when  $B$ is chosen sufficiently small. 
Moreover, we have the elementary inequality
\begin{equation}
	\label{inequ5-elementary}
	\begin{aligned}
		\frac{1}{n\lambda_1^2} \sum_{k>1} F^{i\bar i}
		|\mathfrak{g}_{1 \bar k i} |^2 
		+\frac{2}{\lambda_1}  \mathfrak{Re}(F^{i\bar i}\bar T^{j}_{1i}\mathfrak{g}_{1\bar j i}) \geqslant \frac{2}{\lambda_1}  \mathfrak{Re}(F^{i\bar i}\bar T^{1}_{1i}\mathfrak{g}_{1\bar 1 i}) -C_1\sum F^{i\bar i}.
	\end{aligned}
\end{equation}

Differentiating the equation \eqref{equ1-main}  twice (using covariant derivative), we  have
\begin{equation}
	\label{diff-equ1}
	\begin{aligned}
		F^{i\bar i}\mathfrak{g}_{i\bar i l}  =\psi_{l},   \,\,
		\nonumber
		F^{i\bar i}\mathfrak{g}_{i\bar i 1\bar 1}
		=\psi_{1\bar 1} -F^{i\bar j,l\bar m }\mathfrak{g}_{i\bar j1}\mathfrak{g}_{l\bar m\bar 1}, \nonumber
	\end{aligned}
\end{equation}
and so
\begin{equation}
	\label{mp3}
	\begin{aligned}
		F^{i\bar i}\mathfrak{g}_{1\bar 1 i\bar i} 
		\geqslant  
	-F^{i\bar j,l\bar m }\mathfrak{g}_{i\bar j1}\mathfrak{g}_{l\bar m\bar 1} 
	+ 2\mathfrak{Re}(F^{i\bar i}\bar T^{j}_{1i}\mathfrak{g}_{1\bar j i}) 
	+\psi_{1\bar 1}   -C_2\lambda_1 \sum F^{i\bar i}.    
	\end{aligned}
\end{equation} 
Here we use   the following standard formula
\begin{equation}
	\label{deco1}
	\begin{aligned}
		u_{i\bar j k}-u_{k\bar j i}  = \,& T^l_{ik}u_{l\bar j},  \\  \nonumber
		u_{1\bar 1 i\bar i}-u_{i\bar i 1\bar 1} =R_{i\bar i 1\bar p}u_{p\bar 1}-
		R_{1\bar 1 i\bar p}u_{p\bar i} \,& +2\mathfrak{Re}\{\bar T^{j}_{1i}u_{i\bar j 1}\}+T^{p}_{i1}\bar T^{q}_{i1}u_{p\bar q},   \nonumber
	\end{aligned}
\end{equation}
where  
 $$
 T^k_{ij}=g^{k\bar l} (\partial_i g_{j\bar l}-\partial_j g_{i\bar l}), \, 
R_{i\bar j k\bar l}=-\partial_{\bar j}\partial_i g_{k\bar l}
+g^{p\bar q}\partial_i g_{k\bar q}\partial_{\bar j}g_{p\bar l}. 
$$

 From   \eqref{xp1}, \eqref{mp3}, \eqref{inequ5-elementary} and \eqref{mp2}, we conclude that 
 \begin{equation}
 	\label{gblq-C90}
 	\begin{aligned}
  \mathcal{L} (\log \tilde{\lambda_{1}}) 
 		\geqslant  
 		-F^{i\bar i}\frac{|\tilde{\lambda_{1}},_i|^2}{\lambda_1^2}  -\frac{F^{i\bar j,l\bar m }\mathfrak{g}_{i\bar j1}\mathfrak{g}_{l\bar m\bar 1}}{\lambda_1} 
 		+2\mathfrak{Re} \left( F^{i\bar i}\bar T^{1}_{1i}\frac{\tilde{\lambda_{1}},_i}{\lambda_1} \right)
 		+\frac{\psi_{1\bar 1}}{\lambda_1}  -C_3\sum F^{i\bar i}.
 	\end{aligned}
 \end{equation}  
 
The straightforward computation gives that  
\begin{equation}	
	\label{L-gradient1}	
	\begin{aligned}
		\mathcal{L}(|\nabla u|_\omega^2) 	
	=\,& 	F^{i\bar i} (  |u_{ki} |^2  +|u_{ i\bar k}|^2 )   -2\mathfrak{Re}\{ F^{i\bar i}T^{l}_{ik} u_{\bar k} u_{l\bar i}\} 	
	\\ \,& 	+ 2\mathfrak{Re}\{(\psi_{k}-F^{i\bar i}\chi_{i\bar i  k})u_{\bar k}\} 	+F^{i\bar i} R_{i\bar i k\bar l}u_{l}u_{\bar k},
	\end{aligned}
\end{equation} 
and so 
\begin{equation}
	\label{L1-phi}	
	\begin{aligned}
		\mathcal{L}H 
		\geqslant \,&  
		\frac{1}{2}\phi' F^{i\bar i} \left(|u_{ki}|^{2}+|u_{k\bar i}|^{2}\right) 
		+\varphi'\mathcal{L}(u-\underline{u})
		+\phi'' F^{i\bar i}|(|\nabla u|_\omega^{2})_{i}|^{2}
		\\\,&
		+\varphi'' F^{i\bar i}|(u-\underline{u})_{i}|^{2} 
		+ 2\phi' \mathfrak{Re}\{ \psi_{ {k} } u_{\bar k} \}-C_4\sum F^{i\bar i}.
	\end{aligned}
\end{equation} 

Combining  \eqref{gblq-C90} and \eqref{L1-phi},   we derive that
\begin{equation}
	\label{key11}
	\begin{aligned}
	0\geqslant \,&
		\frac{1}{2}\phi'F^{i\bar i} \left(|u_{ki}|^{2}+|u_{k\bar i}|^{2}\right) 
		+\varphi'\mathcal{L}(u-\underline{u})
		+\phi'' F^{i\bar i}|(|\nabla u|_\omega^{2})_{i}|^{2}
		\\ \,& 
		+\varphi'' F^{i\bar i}|(u-\underline{u})_{i}|^{2}  
		-\frac{F^{i\bar j,l\bar m }\mathfrak{g}_{i\bar j1}\mathfrak{g}_{l\bar m\bar 1}}{\lambda_1}
		-F^{i\bar i}\frac{|\tilde{\lambda_{1}},_i|^2}{\lambda_1^2}  
		\\
		\,&
		+2\mathfrak{Re} \left( F^{i\bar i}\bar T^{1}_{1i}\frac{\tilde{\lambda_{1}},_i}{\lambda_1} \right) 
		+2\phi'\mathfrak{Re}\{\psi_{ {k} } u_{\bar k} \}
		+\frac{\psi_{1\bar 1}}{\lambda_1}  -C_5 \sum F^{i\bar i}.
	\end{aligned}
\end{equation}
 


\subsubsection*{\bf Case I}  Assume that 
$\delta\lambda_{1}\geqslant -\lambda_{n}$ 
($0<\delta\ll \frac{1}{2}$).
Set $$I= \{i:F^{i\bar i}>\delta^{-1}F^{1\bar 1}\} \mbox{ and } J= \{i:F^{i\bar i} \leqslant \delta^{-1}F^{1\bar 1}\}.$$
Clearly $1\in J$. 
Moreover,  the identity \eqref{mp1} implies that
for any   index $1\leqslant i\leqslant n$
\begin{equation}
	\label{zhuo22}
	\begin{aligned}
		-\frac{|\tilde{\lambda_{1}},_i|^2}{\lambda_1^2} 
		\geqslant \,&
		-2\varphi'^{2}  |(u-\underline{u})_{i}|^{2}-2\phi'^{2}  |(|\nabla u|_\omega^{2})_{i}|^{2}.
	\end{aligned}
\end{equation}
By an inequality from concavity and symmetry of $f$ 
(see    \cite{Andrews1994} or \cite{Gerhardt1996}),
 we   derive
\begin{align*} 
	-F^{i\bar j,l\bar m }\mathfrak{g}_{i\bar j1}\mathfrak{g}_{l\bar m\bar 1} \geqslant  	\sum_{i\in I}\frac{F^{i\bar i}
		-F^{1\bar 1}}{\lambda_{1}-\lambda_{i}} 
 |\mathfrak{g}_{i\bar 1 1}|^{2}
 \geqslant  	\sum_{i\in I}\frac{1-\delta}{\lambda_{1}-\lambda_{i}} 
 F^{i\bar i}|\mathfrak{g}_{i\bar 1 1}|^{2}.
\end{align*}
The assumption $\delta\lambda_{1}\geqslant -\lambda_{n}$ implies  
$\frac{1-\delta}{\lambda_{1}-\lambda_{i}}\geqslant \frac{1-2\delta}{\lambda_{1}}.$
Therefore, we can deduce that
\begin{equation}
	\begin{aligned}
		\label{key33}
 	-\frac{F^{i\bar j,l\bar m }\mathfrak{g}_{i\bar j1}\mathfrak{g}_{l\bar m\bar 1}}{\lambda_1}
		-F^{i\bar i}\frac{|\tilde{\lambda_{1}},_i|^2}{\lambda_1^2} 
	 	\geqslant  \,&	
	 	(1-2\delta)\sum_{i\in I} F^{i\bar i} \frac{|\mathfrak{g}_{i\bar 1 1}|^2}{\lambda_1^2}-F^{i\bar i}\frac{|\tilde{\lambda_{1}},_i|^2}{\lambda_1^2} \\
		\geqslant \,&
		(1-2\delta)\sum_{i\in I}F^{i\bar i}\frac{|\mathfrak{g}_{i\bar 1 1}|^2-|\tilde{\lambda_{1}},_i|^2}{\lambda_1^2} 
		\\
		\,&
		-4\delta \varphi'^2 \sum_{i\in I} F^{i\bar i}|(u-\underline{u})_{i}|^2 
		\\
		-2\varphi'^2\delta^{-1}  \,& F^{1\bar 1}K	-2\phi'^2 F^{i\bar i} |(|\nabla u|_\omega^2)_{i}|^2.
	\end{aligned}
\end{equation}
%

Write
$	\mathfrak{g}_{i\bar 1 1} 
= 
\tilde{\lambda_{1}},_i +
\tau_{i},$
where
$\tau_i=\chi_{i\bar 11}-\chi_{1\bar 1 i}
+T_{i1}^lu_{l\bar 1}-(B_1^1)_i.$   
A straightforward 
computation gives 
\begin{equation}
	\label{key119111}
	\begin{aligned}
 |\mathfrak{g}_{i\bar 1 1}|^2-|\tilde{\lambda_{1}},_i|^2 
		 =  |\tau_i|^2 +	2\mathfrak{Re}(\tilde{\lambda_{1}},_i\bar\tau_i)  
		\geqslant  
		-C_6(\lambda_{1}^2+|\tilde{\lambda_{1}},_i|)
		+
		2\lambda_{1}\mathfrak{Re}(\bar T_{i1}^1\tilde{\lambda_{1}},_i).
	\end{aligned}
\end{equation} 
 For any  $\epsilon>0$ and $1\leqslant i\leqslant n$, we have by \eqref{mp1} and \eqref{Prop1-Phi-Psi}
\begin{equation}
	\label{key119111-2}
	\begin{aligned}
		\frac{ |\tilde{\lambda_{1}},_i|}{\lambda_1}	 
		\leqslant  
		\frac{1}{ 2\epsilon}  
		+
		\frac{\epsilon}{2} \phi' \sum_{k=1}^n (|(u_{k i}|^2 +|u_{\bar k i} |^2)
		+ |\varphi' (u-\underline{u})_i|.	 
	\end{aligned}
\end{equation} 
By \eqref{Prop1-Phi-Psi}, $\phi''=2\phi'^2$. 
Together with  \eqref{key33}, \eqref{key119111}  and \eqref{key11}, we obtain 
\begin{equation}
	\label{key11-3}
	\begin{aligned} 
		0
		\geqslant \,&	\frac{1}{2}\phi'F^{i\bar i} \left(|u_{ki}|^{2}+|u_{k\bar i}|^{2}\right) 		+\varphi'\mathcal{L}(u-\underline{u}) 		+(\varphi''-4\delta \varphi'^2)	F^{i\bar i}|(u-\underline{u})_{i}|^{2}  		\\ \,& 		+2\mathfrak{Re} \left( \sum_{i\in J}F^{i\bar i}\bar T^{1}_{1i}\frac{\tilde{\lambda_{1}},_i}{\lambda_1} \right) 		+4\delta \mathfrak{Re} \left(\sum_{i\in I} F^{i\bar i}\bar T^{1}_{1i}\frac{\tilde{\lambda_{1}},_i}{\lambda_1} \right) 		-C_6 \sum_{i\in I} F^{i\bar i}\frac{|\tilde{\lambda_{1}},_i| }{\lambda_1^2} 		\\ \,& 		-2\varphi'^2\delta^{-1}  F^{1\bar 1}K 		+2\phi'\mathfrak{Re}\{ \psi_{ k} u_{\bar k} \}		+\frac{\psi_{1\bar 1}}{\lambda_1}  -C_7  \sum F^{i\bar i}. 
		\end{aligned}
	\end{equation}
Note that 
$|u_{i\bar i}|^{2}= |\mathfrak{g}_{i\bar i}-\chi_{i\bar i}|^{2}	\geqslant \frac{1}{2}|\mathfrak{g}_{i\bar i}|^{2}-|\chi_{i\bar i}|^2$ and $2\delta<1$. 
 Combining   \eqref{key119111-2} and \eqref{Prop1-Phi-Psi},
\begin{equation}
	\label{key11-2}
	\begin{aligned} 
		0
		\geqslant   \,&
	\frac{1}{32K}F^{i\bar i} \mathfrak{g}_{i\bar i}^2
	+\varphi'\mathcal{L}(u-\underline{u})
	+(\frac{1}{2}\varphi'' -4\delta \varphi'^2)
	F^{i\bar i}|(u-\underline{u})_{i}|^{2} 
	\\\,&
	-2  \delta^{-1} T |\varphi'| F^{1\bar 1}\sqrt{K}  
	-2\delta^{-1}\varphi'^2F^{1\bar 1}K	 
	-\frac{C_{9}}{\sqrt{K}}
	+\frac{\psi_{1\bar 1}}{\lambda_1}
	\\
	\,&
	- \left(C_8+\frac{8\delta^2T^2 \varphi'^2}{\varphi''} +C_6\lambda_1^{-1}|\varphi'|\sqrt{K}\right) \sum F^{i\bar i}.
	\end{aligned}
\end{equation} 
where $T=\sup|T^{j}_{ji}|$.
Here we also use the elementary inequality 
 \begin{equation}	\label{key003}	\begin{aligned}\frac{1}{2}\varphi''F^{i\bar i}|(u-\underline{u})_{i}|^{2}-4\delta T  F^{i\bar i}|\varphi'(u-\underline{u})_{i}|\geqslant 
 		-\frac{8\delta^2 T^2\varphi'^2}{\varphi''}\sum F^{i\bar i}. \end{aligned}\end{equation} 

 Let  $\varepsilon$ be the constant as in Lemma  \ref{lemma3-key}.
Pick $C_{*}\gg 1$ such that
\begin{equation}
	\begin{aligned}
		-\frac{\varepsilon}{2}\varphi'\geqslant \frac{\varepsilon C_{*}}{2(1+\sup_{M}(u-\underline{u})-\inf_{ M}(u-\underline{u}))^{2}} > 1+ C_8+C_{9}.
	\end{aligned}
\end{equation}
Fix such 
$C_*$ we  then find $\delta$ sufficiently small so that $4\delta \varphi'^{2} < \frac{1}{2}\varphi''$, $\delta T^2<\frac{1}{2}$.
Thus  
\begin{equation}
	\begin{aligned}
		\label{key001}
		0  \geqslant \,&
		\frac{1}{32K}F^{i\bar i} \mathfrak{g}_{i\bar i}^2
		+\varphi'\mathcal{L}(u-\underline{u})  
		- \left( C_8+ \delta T^2 +C_6\lambda_1^{-1}|\varphi'|\sqrt{K} \right) \sum F^{i\bar i} 
		\\\,&
		-2\delta^{-1} \varphi'^2 F^{1\bar 1}K
			-2 \delta^{-1}  T|\varphi'| F^{1\bar 1}\sqrt{K} 
		-\frac{C_{9}}{\sqrt{K}}
		+\frac{\psi_{1\bar 1}}{\lambda_1}.
	\end{aligned}
\end{equation}

 Below our discussion is based on  Lemma  \ref{lemma3-key}.
According to Lemma  \ref{lemma3-key}, 
in the first case we may assume 
\begin{equation}
	\label{bbbbb3}
	\begin{aligned}
		\mathcal{L}(\underline{u}-u)\geqslant \varepsilon \left(1+ \sum F^{i\bar i}\right).  
	\end{aligned}
\end{equation}  
Then by \eqref{key001} we get  
\begin{equation}
	\label{wukong1}
	\begin{aligned}
	\lambda_1  \leqslant CK.    \nonumber
	\end{aligned}
\end{equation} 

In the second  case, we have
\begin{equation}
	\begin{aligned}
	f_{i} \geqslant  \varepsilon \left(1+ \sum_{j=1}^n f_{j} \right), \,\, \forall 1\leqslant i\leqslant n. 
	\end{aligned}
\end{equation}
Together with \eqref{key001}, \eqref{L1-u-ubar} and \eqref{concave-1}, we thus obtain
\begin{equation}
	\label{inequality-426}
	\begin{aligned} 
		0\geqslant \,& 
		\frac{\varepsilon}{32K} |\mathfrak{g}|^2 
		\left(1+\sum F^{i\bar i}  \right)
		-  \left( C_8+\delta T^2 +C_6\lambda_1^{-1}|\varphi'|\sqrt{K}\right)\sum F^{i\bar i} 
		\\ \,& 
		+\varphi' (\psi-F(\mathfrak{\underline{g}})) 
			-2\varphi'^2\delta^{-1}F^{1\bar 1}K	
		-2 T\delta^{-1} |\varphi'| F^{1\bar 1}\sqrt{K} 
		-\frac{C_{9}}{\sqrt{K}}
		+\frac{\psi_{1\bar 1}}{\lambda_1}. 
	\end{aligned}
\end{equation}
Consequently,  
we can derive  that 
 \begin{equation}	\label{wukong2}	\begin{aligned}	\lambda_1\leqslant CK.  \nonumber	\end{aligned}\end{equation}

\subsubsection*{\bf Case II}
We assume that $\delta\lambda_{1}<-\lambda_{n}$ with the constants $C_{*}, N$ and $\delta$ fixed as in the previous case.
Then  $|\lambda_{i}|\leqslant \frac{1}{\delta}|\lambda_{n}| \mbox{ for all } i$.
Note that $F^{n\bar n} \geqslant \frac{1}{n} \sum F^{i\bar i}$, $|u_{n\bar n}|^{2}=
\frac{1}{2}|\mathfrak{g}_{n\bar n}|^2- |\chi_{n\bar n}|^{2}\geqslant  \frac{\delta^2}{2} \lambda_1^{2}
-|\chi_{n\bar n}|^{2}$ and $-F^{i\bar j,l\bar m }\mathfrak{g}_{i\bar j1}\mathfrak{g}_{l\bar m\bar 1}\geq 0$.
From \eqref{L1-u-ubar} and \eqref{concave-1} we see 
$\mathcal{L}(\underline{u}-u)\geq F(\mathfrak{\underline{g}})-\psi.$
Combining 
\eqref{key11}, \eqref{Prop1-Phi-Psi},  \eqref{zhuo22} and  \eqref{key119111-2}     
we  
obtain  
\begin{equation}	\label{inequality-427}
	\begin{aligned}
		0
   	\geqslant \,&	\frac{1}{2} \phi' F^{i\bar i}(|u_{ki}|^{2}+|u_{k\bar i}|^{2})+\varphi' \mathcal{L}(u-\underline{u}) 		+2\phi'\mathfrak{Re}\{ \psi_{k} u_{\bar k} \} 	\\		\,& +(\varphi''-2\varphi'^{2})F^{i\bar i}|(u-\underline{u})_{i}|^{2} 		+\frac{\psi_{1\bar 1} }{\lambda_1}	-C_5 \sum F^{i\bar i}		+2\mathfrak{Re} \left( F^{i\bar i}\bar T^{1}_{1i}\frac{\tilde{\lambda_{1}},_i}{\lambda_1} \right)	\\
\geqslant \,&
		\frac{\delta^2}{32nK} \lambda_1^2\sum F^{i\bar i} 
		+\varphi' (\psi-F(\mathfrak{\underline{g}}))	
		+(\varphi''-2\varphi'^{2})F^{i\bar i}|(u-\underline{u})_{i}|^{2} 
		\\\,&
		-( C_{10} +2T|\varphi'| \sqrt{K}) \sum F^{i\bar i} 
		-\frac{C_{9}}{\sqrt{K}} 
		+\frac{\psi_{1\bar 1}}{\lambda_1}.  
	\end{aligned}
\end{equation} 
 
By an observation of \cite{Guan-Dirichlet}, the concavity assumption implies that
 there exist 
  $R_1$, $c_1>0$ depending only on   $\sup_M\psi$ and $f$ such that 
\begin{equation}
	\label{Guan-sum1-inequality}
	\begin{aligned}		|\lambda|\sum_{i=1}^n f_i(\lambda)  \geqslant   c_1 \mbox{ in }  \{\lambda\in\Gamma:\inf_M \psi \leqslant f(\lambda) \leqslant \sup_M \psi\}  \mbox{ when }	|\lambda|\geqslant R_1.		\end{aligned}	\end{equation} 
Consequently, together with \eqref{inequality-427} we may conclude that 
 $\lambda_1\leqslant CK.$



 \medskip

\section{Blow-up argument and completion of proof of Theorem  \ref{thm1-main}}
\label{sec-gradient-estimate}

Under optimal structural conditions,  we derive gradient estimate.   
\begin{theorem}
	
	In the setting of Theorem \ref{thm1-main}, let $u$ be an admissible solution to 
	\eqref{equ1-main} with $\sup_M u=0$. Then
	there is a constant $C$  such that
	\begin{equation}
		\label{estimate-C1}
		\sup_{M}|\nabla u|_\omega \leqslant C.
	\end{equation}
\end{theorem}

\begin{proof}
Let  $\{\psi_j\}$ be a sequence of smooth functions such that
 \begin{equation} \label{bounds1-psi}	\begin{aligned}		|\psi_j|_{C^2(M)}\leqslant C, 
 		\,\, \sup_{\partial \Gamma}f <\inf_M\psi_j-\delta, \,\, 
 		\sup_M\psi_j+\delta<\sup_\Gamma f 
 		\end{aligned}\end{equation} 
uniformly hold  for some  $\delta$, $C>0$, and 
the corresponding equations
\begin{equation}
	\label{equ2-rescaling}
	\begin{aligned}
		f(\lambda[\omega^{-1}(\chi+\sqrt{-1}\partial\overline{\partial}u_j)]) =\psi_j  
	\end{aligned}
\end{equation}
admit  $\mathcal{C}$-subsolutions $\underline{u}_j$ with   uniform bound $|\underline{u}_j|_{C^2(M)}\leqslant C$, $\forall j$.   

Assume by contradiction that \eqref{estimate-C1} does not hold. Then there exist   $\psi_j$ as above such that the corresponding
 equations 
possess smooth admissible
solutions $u_j$ satisfying 
$$\sup_M u_j = 0, \,\, 
R_j\equiv\sup_M |\nabla u_j|_\omega \to +\infty \mbox{ as } j\to +\infty.$$ 
In addition, for any $j$  there exists some $x_j\in M$ such that 
 \begin{align*}
 	R_j = |\nabla u_j(x_j)|_\omega.
 \end{align*}
Up to passing to a subsequence, we may assume $x_j\to x\in M$ as $j\to+\infty$.

Using localization, without loss of generality, we may choose  a coordinate chart $\{U, (z^1,\cdots,z^n)\}$ 
centered at $x$, identifying with the ball $B_2(0)\subset \mathbb{C}^n$
of radius 2 centered at origin such that $\omega(x)=\beta$, where $\beta=\sqrt{-1}\sum dz^i\wedge d\bar z^i$. 
We assume 
that
 $j$ is sufficiently large so that
the points $z_j=z(x_j)\in B_1(0)$.  
Define the maps 
\begin{align*}
	\Phi_j: \mathbb{C}^n \to \mathbb{C}^n, \,\, 
	\Phi_j (z)=z_j+R_j^{-1}z,
\end{align*} 
\begin{align*}
	\tilde{u}_j: B_{R_j}(0) \to \mathbb{R}, \,\,  	\tilde{u}_j(z)= u_j\circ \Phi_j(z) = u_j(z_j+R_j^{-1}z),
	\end{align*}
\begin{align*}
	\tilde{\psi}_j: B_{R_j}(0) \to \mathbb{R}, \,\,  	\tilde{\psi}_j(z)= \psi_j\circ \Phi_j(z) = \psi_j(z_j+R_j^{-1}z).
\end{align*}

Denote
\begin{align*}
	\beta_j=R_j^2 \Phi_j^*\omega, \,\, \chi_j =\Phi_j^* \chi.
	\end{align*}
Then the equation 	
$f(\lambda[\omega^{-1}(\chi+\sqrt{-1}\partial\overline{\partial}u_j)])=\psi_j$ implies that
\begin{equation}
	\label{equ1-rescaling}
	\begin{aligned}
		f(R_j^2\lambda [\beta_j^{-1}(\chi_j+\sqrt{-1}\partial\overline{\partial}\tilde{u}_j)])
		=\tilde{\psi}_j. 
	\end{aligned}
\end{equation}

On  the chart near $x$,    
$\beta_j \to \beta$  and $\chi_j(z)\to 0$, in $C_{loc}^\infty$, as $j\to+\infty$. Then we get
\begin{equation}
\label{approximate3-quantity}
\begin{aligned}
	\lambda [\beta_j^{-1}(\chi_j+\sqrt{-1}\partial\overline{\partial}\tilde{u}_j)]
	= 
	\lambda( \tilde{u}_{j,p\bar q}) 
	+O(R_j^{-2}|z|).
\end{aligned}
\end{equation}

From  the $C^0$-estimate in \cite{Gabor}, $\sup_M|u_j|\leqslant C.$ 
By the construction 
\begin{equation}
	\begin{aligned}
		\sup_{B_{R_j}(0)} |\tilde{u}_j| + \sup_{B_{R_j}(0)} |\nabla \tilde{u}_j| \leqslant C'
	\end{aligned}
\end{equation}
for some    $C'>0$,
 where here the gradient is the Euclidean gradient.
 Moreover,
\begin{equation}
	\begin{aligned}
		|\nabla \tilde{u}_j(0)|=R_{j}^{-1} |\nabla u_j|_\omega(x_j)=1.
	\end{aligned}
\end{equation}
According to Theorem \ref{thm2-second-order}, we also have that 
 \begin{equation}
 	\begin{aligned}
\sup_{B_{R_j}(0)} |\sqrt{-1}\partial\overline{\partial} \tilde{u}_j|_{\beta} \leqslant CR_j^{-2}\sup_{M} |\partial\overline{\partial}u_j|_\omega \leqslant C''.
 	\end{aligned}
 \end{equation}
As the argument in \cite{Dinew2017Kolo,Tosatti2013Weinkove}, using the elliptic estimates for  $\Delta  $ and
the Sobolev embedding that for each given $K\subset \mathbb{C}^n$ compact,  each $0<\alpha<1$ and $p>1$, there
exists a constant $C$ such that
\begin{equation}
	\begin{aligned}
		\| \tilde{u}_j\|_{C^{1,\alpha}(K)} + \|\tilde{u}_j\|_{W^{2,p}(K)} \leqslant C. 
	\end{aligned}
\end{equation}
Therefore, a subsequence of $\tilde{u}_j$ converges uniformly in $C_{loc}^{1,\alpha}(\mathbb{C}^n)$ as well as weakly in $W_{loc}^{2,p}(\mathbb{C}^n)$
to a function 
\begin{equation}
	\label{limit-function1}
	v: \mathbb{C}^n \to \mathbb{R}.
\end{equation}
By the
construction we   have global bounds 
$\sup_{\mathbb{C}^n}(|v|+|\nabla v|)\leqslant C$ and $|\nabla v(0)|=1$.  
  In particular, $v$ is not constant.

 The rest of  proof is to 
show that $v$ is a 
$\mathring{\Gamma}_{\mathcal{G}}^{f}$-solution in the sense of \cite[Definition 15]{Gabor}.
Here 
 ${\Gamma}_{\mathcal{G}}^{f}$  is the cone defined in \eqref{component1}.
To do this,  suppose first that we have a $C^2$-function $\phi$, such that $\phi\geqslant v$, and
$\phi(z_0) = v(z_0)$ for some point $z_0$. We need to show that $\lambda[\phi_{p\bar q}(z_0)]\in \bar\Gamma_{\mathcal{G}}^{f}$. By the construction of $v$, 
similar to \cite{Gabor} 
for any $\epsilon>0$  we can find a 
$\tilde{u}_j$ as above,
corresponding to a sufficiently large $j$,
a number $c$ with $|c|<\epsilon $, and a
point  $z_0'$ with $|z_0'-z_0|<\epsilon,$  so that
\begin{equation}
	\begin{aligned}
		\phi+\epsilon |z-z_0|^2 +c \geqslant \tilde{u}_j \mbox{ on } B_1(z_0)  \nonumber
	\end{aligned}	 
\end{equation}
with equality at $z_0'$. Therefore, 
\begin{equation}
	\label{approximate4-quantity}
	\phi_{p\bar q}(z_0') + \epsilon \delta_{pq} \geqslant \tilde{u}_{j, p\bar q}(z_0'). 
\end{equation}  
Furthermore, note that for $j\gg1$ (and so $R_j$) we have  by \eqref{approximate3-quantity} that 
\begin{equation}
	\label{approximate5-quantity}
	\begin{aligned} 
\lambda[\tilde{u}_{j,p\bar q}(z_0')]+\epsilon \vec{\bf 1}	-\lambda [\beta_j^{-1}(\chi_j+\sqrt{-1}\partial\overline{\partial}\tilde{u}_j)(z_0')]\in \Gamma_n.
	\end{aligned}
\end{equation}   
Set $\tilde{\varepsilon}=\sup_{z\in B_\epsilon(z_0)}|\phi_{p\bar q}(z)-\phi_{p\bar q}(z_0)|.$   
Then  \begin{equation}
	\label{approximate7-quantity}\phi_{p\bar q}(z_0')\leqslant \phi_{p\bar q}(z_0)+\tilde{\varepsilon}\delta_{pq}.\end{equation}

Combining  \eqref{approximate7-quantity}, \eqref{approximate4-quantity},  \eqref{approximate5-quantity} and \eqref{equ1-rescaling}, we have 
\begin{equation}
	\label{inequ3-N}
	\begin{aligned}  
		f(R_j^2 [\lambda[\phi_{p\bar q}(z_0)]+(2\epsilon+\tilde{\varepsilon})\vec{\bf1}]) \geqslant \tilde{\psi}_j(z_0').
	\end{aligned}
\end{equation}
Together with the  choice of  $R_j$, $\phi$, $\tilde{\psi}_j$, $\epsilon$, $\tilde{\varepsilon}$, and the definition of $\Gamma_{\mathcal{G}}^{f}$, 
we  
infer that
\begin{equation}
	\label{goodcone-advant1}
	\begin{aligned} 
		\lambda[\phi_{p\bar q}(z_0)]+  (2\epsilon+\tilde{\varepsilon})\vec{\bf 1}\in \Gamma_{\mathcal{G}}^{f}.	\end{aligned}
\end{equation} 
Let $\epsilon\to 0$, we get $\tilde{\varepsilon}\to 0$, and also $\lambda[\phi_{p\bar q} (z_0)]\in \bar\Gamma_{\mathcal{G}}^{f}.$  We remark that this conclusion 
 is not that obvious, which is in contrast with the case  
 $\Gamma_{\mathcal{G}}^f=\Gamma$.  

Suppose now that we have a $C^2$-function $\phi$, such that $\phi\leqslant v$, and
$\phi(z_0) = v(z_0)$ for some point $z_0$. It requires to show that 
$\lambda[\phi_{p\bar q}(z_0)]\in \mathbb{R}^n\setminus \mathring{\Gamma}_{\mathcal{G}}^{f}$. As above, for any
$\epsilon$ we can find a $\tilde{u}_j$ corresponding to large $j$, and $c$, $z_0'$  with $|z_0'-z_0|<\epsilon,$ so that
\begin{equation}
	\begin{aligned}
		\phi-\epsilon |z-z_0|^2 +c \leqslant \tilde{u}_j \mbox{ on } B_1(z_0)
	\end{aligned}	 
\end{equation}
with equality at $z_0'$. This implies that 
\begin{equation}
	\phi_{p\bar q}(z_0') - \epsilon \delta_{pq} \leqslant \tilde{u}_{j, p\bar q}(z_0'). 
\end{equation}
Suppose by contradiction that 
$$\lambda(	\phi_{p\bar q}(z_0') -3 \epsilon \delta_{p\bar q})\in \mathring{\Gamma}_{\mathcal{G}}^{f}, \mbox{ i.e., } \lambda(\phi_{p\bar q}(z_0'))\in  \mathring{\Gamma}_{\mathcal{G}}^{f}+3\epsilon\vec{\bf 1}.$$ 
Then $\lambda[\tilde{u}_{j,p\bar q}(z_0') ]\in \mathring{\Gamma}_{\mathcal{G}}^{f} +2 \epsilon\vec{\bf 1}$. By \eqref{approximate3-quantity} and the openness of $\mathring{\Gamma}_{\mathcal{G}}^{f}$, if  $j$ is
sufficiently large (so does $R_j$), we have 
\begin{equation}
	\lambda[{\beta}_j^{-1} (\chi_j+\sqrt{-1}\partial\overline{\partial}\tilde{u}_j)](z_0') \in 
	\mathring{\Gamma}_{\mathcal{G}}^{f} + \epsilon\vec{\bf 1}.
\end{equation}
According to Lemma \ref{lemma2-key}, we conclude that
\begin{equation} 
	\begin{aligned}
		f(R_j^2 	\lambda[ {\beta}_j^{-1} (\chi_j+\sqrt{-1}\partial\overline{\partial}\tilde{u}_j)](z_0'))
		\geqslant f(\epsilon R_j^2\vec{\bf1}).
	\end{aligned}
\end{equation}
It contradicts to \eqref{equ1-rescaling} when $R_j$ is sufficiently large. Thus
\begin{align*}
	\lambda[\phi_{p\bar q}(z_0')-3 \epsilon \delta_{pq}]\in
	\mathbb{R}^n\setminus \mathring{\Gamma}_{\mathcal{G}}^{f}.
\end{align*}
Let $\epsilon\to 0$, we will have
$z_0' \to z_0$, and so 
$\lambda[\phi_{p\bar q}(z_0)]\in \mathbb{R}^n\setminus \mathring{\Gamma}_{\mathcal{G}}^{f}.$ 

Therefore, $v$ is a
$\mathring{\Gamma}_{\mathcal{G}}^{f}$-solution.
%
From  Lemma \ref{lemma1-key}, $\mathring{\Gamma}_{\mathcal{G}}^{f}$  is an open symmetric convex cone, containing
the positive cone $\Gamma_n$ and not equal to all of 
$\mathbb{R}^n$. 
According to Sz\'ekelyhidi's Liouville type theorem  \cite[Theorem 20]{Gabor}, 
 $v$ is a constant. This is  a  contradiction.

\end{proof}

 \begin{remark}	\label{remark-52}
  
 The definition of $\Gamma_\mathcal{G}^{f}$  in
 \eqref{component1} is important, since it can   straightforwardly  utilize \eqref{inequ3-N}. 
 	\end{remark}

\medskip

\section{Fully nonlinear elliptic equations with type 2 cone} 
\label{sec1-type2}



  For the equation \eqref{equ1-main}, we say that $\chi$ is {\em quasi-admissible}, if 
  \begin{equation}
  	\label{key4}
  	\begin{aligned}
  		\lambda(\omega^{-1}\chi)\in \bar\Gamma \mbox{ in } M  
  \end{aligned}\end{equation}	 
  and 
  \begin{equation} \label{key5}	\begin{aligned}
  		\lambda(\omega^{-1}\chi)(p_0)\in  \Gamma \mbox{ for some } p_0\in M. 
  	\end{aligned}
  \end{equation}	 
Throughout this section, we assume that $(M,\omega)$ is a connected closed Hermitian manifold, and $\chi$ is a smooth quasi-admissible real $(1,1)$-form. 

\begin{proposition}
	\label{construc2-prop}
	Suppose  $\Gamma$ is of type 2 in the sense of \cite{CNS3}, i.e.,  $(0,\cdots,0,1)\in\Gamma$. 
	 Then there exists a  smooth admissible function $\underline{u}$, i.e.,  
 $\lambda[\omega^{-1}(\chi+\sqrt{-1}\partial\overline{\partial}  \underline{u})]\in\Gamma$ in $M$. 
\end{proposition}
 
 \begin{proof}
 	
We use an idea of \cite{yuan-PUE-conformal}, using some technique from Morse theory.
 By  \eqref{key5} 
 and the openness of $\Gamma$,  
 \begin{equation}
 	\label{key1-metric}
 	\begin{aligned}
 		\lambda(\omega^{-1}\chi)\in \Gamma \mbox{ in } \overline{B_{r_0}(p_0)} 
 	\end{aligned}
 \end{equation}
 for some positive constant $r_0$.
We take a smooth Morse function $w$ with
  critical set  
 $$\mathcal{C}(w)=\{p_1,\cdots, p_m, p_{m+1}\cdots p_{m+k}\}$$
 among which $p_1,\cdots, p_m$ are all the critical points  being in $M\setminus \overline{B_{r_0/2}(p_0)}$. 
 Pick $q_1, \cdots, q_m\in {B_{r_0/2}(p_0)}$ but not the critical point of $w$. By the homogeneity lemma 
  (see e.g. \cite{Milnor-1997}), 
 one can find a diffeomorphism
 $h: M\to M$, which is smoothly isotopic to the identity, such that  
  \begin{itemize} 	\item $h(p_i)=q_i$, $1\leqslant i\leqslant  m$. 	\item $h(p_i)=p_i$, $m+1\leqslant  i\leqslant  m+k$. \end{itemize}
 Then we obtain a   Morse function $v=w\circ h^{-1}$
 with the
 critical set  
 \begin{equation}
 	\label{key2}
 	\begin{aligned}
 		\mathcal{C}(v)=\{q_1,\cdots, q_m, p_{m+1}\cdots p_{m+k}\}\subset \overline{B_{r_0/2}(p_0)}.
 	\end{aligned}
 \end{equation}
 
 Next we complete the proof. Assume $v\leqslant  -1$.   	Take $\underline{u}=e^{tv}$, 
 then \begin{equation}
 	\label{key3-2}
 	\begin{aligned}
 	\chi+\sqrt{-1}\partial\overline{\partial}  \underline{u}=\chi+te^{tv}(\sqrt{-1}\partial\overline{\partial} v+t\sqrt{-1}\partial v\wedge \overline{\partial} v).  
 	\end{aligned}
 \end{equation}	 
 Notice that  
 \begin{equation}
 	\label{key3}
 	\begin{aligned}
  \lambda[\omega^{-1}(\sqrt{-1}\partial v\wedge \overline{\partial} v) ] =|\nabla v|^2(0,\cdots,0,1).
 	\end{aligned}
 \end{equation}   
From \eqref{key2} there is a  positive constant $m_0$ such that 
 $|\nabla v|^2\geqslant  m_0  \mbox{ in }  M\setminus\overline{B_{r_0}(p_0)}.$ 
   By \eqref{key4},  
    \eqref{key3-2}, \eqref{key3}  
 and the openness of $\Gamma$,
  there exists $t\gg1$ such that
 \begin{equation}
 	\begin{aligned}
 	 		\lambda[\omega^{-1}(\chi+\sqrt{-1}\partial\overline{\partial}  \underline{u})] \in\Gamma \mbox{ in } M\setminus\overline{B_{r_0}(p_0)}, \nonumber
 	\end{aligned}
 \end{equation}
and moreover by \eqref{key1-metric}
\begin{equation}
	\begin{aligned}
	\lambda[\omega^{-1}(	\chi+te^{tv}\sqrt{-1}\partial\overline{\partial} v)] \mbox{ in } \overline{B_{r_0}(p_0)}.\nonumber
	\end{aligned}
\end{equation} 
 In conclusion, we know $\lambda[\omega^{-1}(\chi+\sqrt{-1}\partial\overline{\partial}  \underline{u})] \in \Gamma$  in $M.$

 \end{proof}

 \begin{theorem}
 	\label{thm1-existence-2} 
 	In addition to \eqref{concave}, 
  $\sup_{\partial \Gamma} f=-\infty$ and $\sup_\Gamma f=+\infty,$ we assume   $(0,\cdots,0,1)\in 	\mathring{\Gamma}_{\mathcal{G}}^{f}$.
 	Then for any $\psi\in C^\infty(M)$,  
 there exists a constant $c$ and a   smooth 
 	admissible function $u$ with $\sup_M u=0$ such that
 	\begin{equation}
 		\begin{aligned}
 	 	f(\lambda[\omega^{-1}(\chi+\sqrt{-1}\partial\overline{\partial}u)])=\psi+c.
 			\nonumber
 		\end{aligned}
 	\end{equation} 
 \end{theorem}
\begin{proof}
By Proposition \ref{construc2-prop}, we have a smooth admissible function $\underline{u}$. According to Lemma 
\ref{lemma2-key},
   \eqref{unbounded-1} automatically holds if $(0,\cdots,0,1)\in \mathring{\Gamma}_{\mathcal{G}}^{f}$ and $\sup_\Gamma f=+\infty.$ 
\end{proof}
 
 
  \medskip

 \section{Quantitative boundary estimate: Set-up}
 \label{sec1-Bdy-setup}

 Denote $\sigma(z)$  the distance  from $z$ to $\partial M$ with respect to $\omega$.  Let    \begin{equation} \label{M-delta}	\begin{aligned} 		M_\delta\equiv\{z\in M: \sigma(z)<\delta\}.	\end{aligned} \end{equation} 
Given a boundary point $z_0\in \partial M$.  
 We choose local 
complex coordinates  with origin at $z_0$
 \begin{equation}
 	\label{goodcoordinate1}
 	\begin{aligned}
 		(z^1,\cdots, z^n), \mbox{  } z^i=x^i+\sqrt{-1}y^i 
 	\end{aligned}
 \end{equation}
 in a neighborhood which we assume to be contained in
 $M_\delta$, 
 such that at the origin $z_0$ $(z=0)$,
 $g_{i\bar j}(0)=\delta_{ij}$, $\frac{\partial}{\partial x^{n}} $ is the interior normal to $\partial M$. 
 (Here and in what follows we identify the boundary point  $z_0$ with $z=0$).
For convenience 
we set
 \begin{equation}
 	\label{t-coordinate}
 	\begin{aligned}
 		t^{2\alpha-1}=x^{\alpha}, \ t^{2\alpha}=y^{\alpha},\ 1\leqslant \alpha\leqslant n-1;\ t^{2n-1}=y^{n},\ t^{2n}=x^{n}. 
 	\end{aligned}
 \end{equation}
 In the computation, we use the notation  in  \eqref{formula-1}.
  Also we denote
  \begin{equation}	\begin{aligned}		\Omega_{\delta}\equiv\{z\in M: |z|<\delta\}.	\end{aligned} \end{equation} 
 Throughout this section and Section \ref{sec1-Bdy-estimate}, the Greek letters $\alpha, \beta$ run from $1$ to $n-1$.

 Since $u-\varphi=0$ on $\partial M$, we can derive  
 \begin{equation}	\label{ineq2-bdy}	\begin{aligned}		|u_{t^k t^l}(z_0)|\leqslant C  \mbox{ for } k, l<2n,	\end{aligned} \end{equation} 
 where $C$ depends on $\sup_{\partial M}|\nabla (u-\varphi)|$ and geometric quantities of $\partial M$.
 
To show Theorem \ref{thm1-bdy}, it suffices to prove
  the following two propositions:

 \begin{proposition}
 	\label{proposition-quar-yuan2}
 	Suppose 
 	$f$ satisfies	\eqref{concave},  \eqref{nondegenerate} and  \eqref{unbounded-1}.
 	Assume that Dirichlet problem \eqref{mainequ-Dirichlet} admits a $C^2$ {\em admissible} subsolution $\underline{u}$. 
 	Let $u\in  C^{2}(\bar M)$ be an \textit{admissible} solution  to the Dirichlet problem, then
 	\begin{equation}
 		\begin{aligned}
 			\mathfrak{g}_{n\bar n}(z_0)
 			\leqslant  C\left(1 +  \sum_{\alpha=1}^{n-1} |\mathfrak{g}_{\alpha\bar n}(z_0)|^2\right),  \,\,\forall z_0\in\partial M.  \nonumber
 		\end{aligned}
 	\end{equation}
 	Here the constant $C$ depends on  $(\delta_{\psi,f})^{-1}$,  $\sup_M\psi$, $|u|_{C^0(\bar M)}$,    	$|\nabla u|_{C^0(\partial M)}$,
 	$|\underline{u}|_{C^{2}(\bar M)}$, $\partial M$ up to third order derivatives 
 	and other known data
 	(but not on $\sup_{M}|\nabla u|$). 
 	
 \end{proposition}

\begin{proposition}
	\label{mix-general}

	Assume  \eqref{elliptic}, \eqref{concave} and  \eqref{nondegenerate}  hold. 
	Suppose that Dirichlet problem \eqref{mainequ-Dirichlet} admits an {\em admissible} subsolution $\underline{u}\in C^2(\bar M)$. 
	Then 
	for any admissible solution $u\in C^3(M)\cap C^2(\bar M)$ to the  Dirichlet problem,  we have
	\begin{equation}
		\label{quanti-mix-derivative-00} 
		|\mathfrak{g}_{\alpha \bar n}(z_0)|\leqslant C\left(1+\sup_{M}|\nabla u|\right)\, \mbox{ for any } z_0\in\partial M, 
	\end{equation}
	where $C$ a positive constant depending on $|\varphi|_{C^{3}(\bar M)}$, 
	$|\psi|_{C^1(M)}$, $|\nabla u|_{C^0(\partial M)}$, 
	$|\underline{u}|_{C^{2}(\bar M)}$,  
	$\partial M$ up to third derivatives
	and other known data (but neither on $(\delta_{\psi,f})^{-1}$ nor on $\sup_{M}|\nabla u|$).
\end{proposition}

 \subsection{$C^0$-estimate and boundary gradient estimate}

 Now we derive  $C^0$-estimate, and  boundary gradient  estimates.
 Let $\check{u}$ be the solution to 
 \begin{equation}
 	\label{supersolution-1}
 	\begin{aligned}
 		\,& \Delta \check{u} + \mathrm{tr}_\omega(\chi)=0 \mbox{ in } M,   \,& \check{u}=\varphi \mbox{ on } \partial M.
 	\end{aligned}
 \end{equation}
 The existence of $\check{u}$ follows from standard theory of elliptic equations of second order.
 Such $\check{u}$ is a supersolution of  \eqref{mainequ-Dirichlet}.
 By the maximum principle 
 one derives
 \begin{equation}
 	\begin{aligned}
 		\underline{u}\leqslant u\leqslant \check{u} \mbox{ in } M, 	
 	\end{aligned}
 \end{equation}	
 \begin{equation}
 	\label{key-14-yuan3}
 	\begin{aligned}
 		\underline{u}_{x^n}(0) \leqslant u_{x^n}(0) \leqslant \check{u}_{x^n}(0), \mbox{ }  u_\beta(0)=\underline{u}_\beta(0), \mbox{ } 1\leqslant\beta\leqslant n-1.
 	\end{aligned}
 \end{equation}
 This simply gives the following lemma.
 \begin{lemma}
 	\label{lemma-c0-bc1}
 	There is a uniform positive constant $C$ such that 
 	\begin{equation}
 		\begin{aligned}
 			\sup_M |u|+\sup_{\partial M} |\nabla u| \leqslant C.
 		\end{aligned}
 	\end{equation}
 \end{lemma}

 \subsection{A key lemma} 
 The boundary value condition yields 
 \begin{equation}
 	\label{yuan3-buchong5}
 	\begin{aligned}
 		u_{\alpha\bar\beta}(0)=\underline{u}_{\alpha\bar\beta}(0)+(u-\underline{u})_{x^n}(0)\sigma_{\alpha\bar\beta}(0).
 	\end{aligned}
 \end{equation} 
 At  $z_0\in\partial M$ ($z=0$),  by \eqref{yuan3-buchong5} we have  
 \begin{equation}
 	\label{410-buchong}
 	\begin{aligned}
 		\mathfrak{g}_{\alpha\bar\beta}= (1-t)\underline{\mathfrak{g}}_{\alpha\bar\beta}
 		+\{t\underline{\mathfrak{g}}_{\alpha\bar\beta}+ (u-\underline{u})_{x^n}\sigma_{\alpha\bar\beta}\}.
 	\end{aligned}
 \end{equation}
 For simplicity, 
 denote
 \begin{equation}
 	\label{A_t}
 	\begin{aligned}
 		A_t=\sqrt{-1} \left[t\underline{\mathfrak{g}}_{\alpha\bar\beta}+ (u-\underline{u})_{x^n}\sigma_{\alpha\bar\beta}\right]dz^\alpha\wedge d\bar z^\beta.
 	\end{aligned}
 \end{equation}
 Clearly, $(A_{1})_{\alpha\bar\beta}=\mathfrak{g}_{\alpha\bar\beta}$. 
 Let $t_0$ be the first  
 $t$ as we decrease $t$ from $1$  
 such that
 \begin{equation}
 	\label{key0-yuan3}
 	\begin{aligned}
 		\lambda_{\omega'}(A_{t_0})\in\partial\Gamma_\infty 
 	\end{aligned}
 \end{equation}
at $z_0$,
where $\partial\Gamma_\infty$ is the boundary of $\Gamma_\infty$, as above $$\Gamma_\infty=\{(\lambda_{1}, \cdots, \lambda_{n-1}): (\lambda_{1}, \cdots, \lambda_{n-1},R)\in\Gamma \mbox{ for some } R>0\}.$$
 Henceforth, $\lambda_{\omega'}(\chi')$ denotes the eigenvalues of $\chi'$ with respect to $\omega'=\sqrt{-1}g_{\alpha\bar\beta} dz^\alpha\wedge d\bar z^\beta.$
 (Note that $g_{\alpha\bar \beta}(0)=\delta_{\alpha\beta}$ under the coordinate \eqref{goodcoordinate1}).
 Such $t_0$ exists, since 
 $\lambda_{\omega'}(A_1)\in\Gamma_\infty$ and
 $\lambda_{\omega'}(A_t)\in \mathbb{R}^{n-1}\setminus\Gamma_\infty$ for $t\ll -1$. Furthermore, 
 \begin{equation}
 	\label{1-yuan3}
 	\begin{aligned}
 		-T_0< t_0<1 
 	\end{aligned}
 \end{equation}
  for some uniform positive constant $T_0$ under control.
 Denote  
 \begin{equation} 	\label{yuan3-buchong6} 
 	\begin{aligned} 	
 	\underline{\lambda}'=(\underline{\lambda}'_1,\cdots,\underline{\lambda}'_{n-1})	\equiv  \lambda_{\omega'}( \underline{\mathfrak{g}}_{\alpha\bar\beta}),
 	\end{aligned}
 \end{equation} 
 \begin{equation}
 	\label{yuan3-buchong3} 	\begin{aligned}  
 	\tilde{\lambda}'=(\tilde{\lambda}_1',\cdots,\tilde{\lambda}_{n-1}') \equiv 	\lambda_{\omega'} ( A_{t_0}). 	\end{aligned}
 \end{equation}
 
 To complete the proof of Proposition \ref{proposition-quar-yuan2}, it   requires to prove the following lemma.
 
 \begin{lemma}
 	\label{keylemma1-yuan3}
 	Let $t_0$ be as defined in \eqref{key0-yuan3}.  
 	There  is a uniform positive constant $C$ depending on 
 	$|u|_{C^0(M)}$, $|\nabla u|_{C^0(\partial M)}$, $|\underline{u}|_{C^2(M)}$, $\sup_M\psi$,
 	$(\delta_{\psi,f})^{-1}$, $\partial M$ up to third derivatives and other known data, such that
 	\begin{equation}
 		\begin{aligned}
 			(1-t_0)^{-1}\leqslant C.  \nonumber
 		\end{aligned}
 	\end{equation}

 \end{lemma}

 \subsection{Proof of  Lemma \ref{keylemma1-yuan3}}
 \label{proof-keylemma1}
 The proof is based on some idea of Caffarelli-Nirenberg-Spruck \cite{CNS3}. 
 %
 We assume that $\Gamma$ is of type 1. Then $\Gamma_\infty$ is a symmetric convex cone as noted in \cite{CNS3}. 
 (The  type 2 cone case is much more simpler since $\Gamma_\infty=\mathbb{R}^{n-1}$).    
 
 We follow closely  \cite{LiSY2004}. 
 For simplicity, we set 
 \begin{equation}
 	\label{def1-eta}
 	\begin{aligned}
 	\eta=(u-\underline{u})_{x^n}(0).
 	\end{aligned}
 \end{equation}
From  \eqref{key-14-yuan3}, we know $\eta\geqslant0$.
 Without loss of generality, we assume $$t_0>\frac{1}{2} \mbox{ and } \tilde{\lambda}_1'\leqslant \cdots \leqslant \tilde{\lambda}_{n-1}'.$$ 

 It was proved in \cite[Lemma 6.1]{CNS3} that for $\tilde{\lambda}'\in\partial\Gamma_\infty$
 there is a supporting plane for
 $\Gamma_\infty$ and one can choose $\mu_j$ with 
 $\mu_1\geqslant \cdots\geqslant \mu_{n-1}\geqslant0$ so that
 \begin{equation}
 	\label{key-18-yuan3}
 	\begin{aligned}
 		\Gamma_\infty\subset \left\{\lambda'\in\mathbb{R}^{n-1}: \sum_{\alpha=1}^{n-1}\mu_\alpha\lambda'_\alpha>0 \right\}, \mbox{  }
 		\mbox{  } \sum_{\alpha=1}^{n-1} \mu_\alpha=1, \mbox{  } \sum_{\alpha=1}^{n-1}\mu_\alpha \tilde{\lambda}_\alpha'=0.
 	\end{aligned}
 \end{equation} 
 
 Since $\underline{u}$ is an admissible function, there is $\varepsilon_0>0$ small such that
 \begin{equation} 
 	\label{varepsilon-0-a0}
 	\begin{aligned}
 		\underline{\lambda}-\varepsilon_0\vec{\bf 1}\in\Gamma.
 	\end{aligned}
 \end{equation}
 Then we know that
 \begin{equation}
 	\label{adf}
 	\begin{aligned}
 		\sum_{\alpha=1}^{n-1}\mu_\alpha \underline{\lambda}'_\alpha
 		\geqslant \varepsilon_0>0. 
 	\end{aligned}
 \end{equation} 
 Combining \cite[Lemma 6.2]{CNS3}
 (without loss of generality,  assume $\underline{\lambda}_1'\leqslant \cdots\leqslant\underline{\lambda}_{n-1}'$), 
 \begin{equation}
 	\begin{aligned}
 		\sum_{\alpha=1}^{n-1} \mu_\alpha \underline{\mathfrak{g}}_{\alpha\bar\alpha}\geqslant \sum_{\alpha=1}^{n-1}\mu_\alpha \underline{\lambda}'_\alpha  
 		\geqslant \varepsilon_0. 
 	\end{aligned}
 \end{equation}
 Without loss of generality,  we may assume $({A_{t_0}})_{\alpha\bar\beta}=t_0\underline{\mathfrak{g}}_{\alpha\bar\beta}+\eta\sigma_{\alpha\bar\beta}$ is diagonal at  the origin $z_0$ $(z=0)$. From \eqref{key-18-yuan3} one has at the origin
 \begin{equation}
 	\begin{aligned}
 		0=  t_0 \sum_{\alpha=1}^{n-1}\mu_\alpha\underline{\mathfrak{g}}_{\alpha\bar\alpha}+\eta\sum_{\alpha=1}^{n-1}\mu_\alpha  {\sigma}_{\alpha\bar\alpha} 
 		\geqslant \varepsilon_0 t_0 +\eta \sum_{\alpha=1}^{n-1} \mu_\alpha \sigma_{\alpha\bar\alpha}
 		>\frac{\varepsilon_0}{2} +\eta \sum_{\alpha=1}^{n-1} \mu_\alpha \sigma_{\alpha\bar\alpha}.
 	\end{aligned}
 \end{equation}
 Together with \eqref{key-14-yuan3},
 we see that 
    $\eta>0$  and 
 \begin{equation}
 	\label{key-1-yuan3}
 	\begin{aligned}
 		-\sum_{\alpha=1}^{n-1} \mu_\alpha \sigma_{\alpha\bar\alpha}\geqslant
 		\frac{\varepsilon_0}{2\sup_{\partial M}|\nabla (\check{u}-\underline{u})|}=:a_1>0,
 	\end{aligned}
 \end{equation}
 where $\check{u}$ and $\underline{u}$ are respectively supersolution and subsolution.
 
 On $\Omega_\delta=M\cap B_{\delta}(0)$, 
 we take
 \begin{equation}
 	\begin{aligned}
 		d(z)=\sigma(z)+\tau |z|^2
 	\end{aligned}
 \end{equation}
 where $\tau$ is a positive constant 
 to be determined. Let 
 \begin{equation}
 	\label{w-buchong1}
 	\begin{aligned}
 		w(z)=\underline{u}(z)+({\eta}/{t_0})\sigma(z)+l(z)\sigma(z)+Ad(z)^2,
 	\end{aligned}
 \end{equation}
 where $l(z)=\sum_{i=1}^n (l_iz^i+\bar l_{i}\bar z^{i})$, 
 $l_i\in \mathbb{C}$, $\bar l_i=l_{\bar i}$,
 to be chosen as in \eqref{chosen-1}, and $A$ is a positive constant to be determined. 
 As above $\eta=(u-\underline{u})_{x^n}(0)$.
 Furthermore, on $\partial M\cap\bar\Omega_\delta$, $u(z)-w(z)=-A\tau^2|z|^4$.
 On $M\cap\partial B_{\delta}(0)$,
 \begin{equation}
 	\begin{aligned}
 		u(z)-w(z) 
 		\leqslant 
 		-(2A\tau \delta^2+\frac{\eta}{t_0}-2n \sup_{i}|l_i| \delta)\sigma(z)-A\tau^2\delta^4 
 		\leqslant 
 		-\frac{A\tau^2 \delta^4}{2} \nonumber
 	\end{aligned}
 \end{equation}
 provided $A\gg1$. 

 Let $T_1(z),\cdots, T_{n-1}(z)$ be an orthonormal basis for holomorphic tangent space 
 of level hypersurface
 $ \{w: d(w)=d(z)\}$ at $z$, so that  for each 
 $1\leqslant\alpha\leqslant n-1$, $T_\alpha$ is of $ C^1$ class and  
 $T_\alpha(0)= \frac{\partial }{\partial z^\alpha}$. 
 Such a basis exists and
 the holomorphic tangent space 
 can be characterized as  $\left\{\xi=\xi^i\frac{\partial}{\partial z^i}: (\sigma_i+\tau \bar z^i ) \xi^i=0\right\}$,
 see e.g. \cite{Demailly-CADG-book}.

 By \cite[Lemma 6.2]{CNS3}, we have the following lemma.
 \begin{lemma}
 	\label{lemma-yuan3-buchong1}
 	Let $T_1(z),\cdots, T_{n-1}(z)$ be as above, and  let  $T_n=\frac{\partial d}{|\partial d|}$.
 	For a real $(1,1)$-form $\Theta=\sqrt{-1}\Theta_{i\bar j}dz^i\wedge d\bar z^j$,
 	we denote by $\lambda(\omega^{-1}\Theta)=(\lambda_1(\Theta),\cdots,\lambda_n(\Theta))$  
 	with $\lambda_1(\Theta)\leqslant \cdots\leqslant \lambda_n(\Theta)$. Then for any 
 	$\mu_1\geqslant\cdots\geqslant\mu_n$,
 	\[\sum_{i=1}^n \mu_i \lambda_{i}(\Theta)\leqslant \sum_{i=1}^n\mu_i\Theta(T_i,J\bar T_i).\]
 	Here   $J$ is the complex structure.
 \end{lemma}
 
 Let $\mu_1,\cdots,\mu_{n-1}$ be as in \eqref{key-18-yuan3}, and we set $\mu_n=0$. Let's denote $T_\alpha=\sum_{k=1}^nT_\alpha^k\frac{\partial }{\partial z^k}$.  
 For $\Theta=\sqrt{-1}\Theta_{i\bar j}dz^i\wedge d\bar z^j$,  define
 \begin{equation}
 	\begin{aligned}
 		\Lambda_\mu(\Theta) \equiv \sum_{\alpha=1}^{n-1}\mu_{\alpha} T_{\alpha}^i \bar T_{\alpha}^j \Theta_{i\bar j}.  \nonumber
 	\end{aligned}
 \end{equation}
 
 \begin{lemma}
 	\label{lemma-key2-yuan3}
 	Let $w$ be as in \eqref{w-buchong1}.
 	There are parameters $\tau$, $A$, $l_i$, $\delta$ depending only on $|u|_{C^0(M)}$, 
 	$|\nabla u|_{C^0(\partial M)}$,  
 	$|\underline{u}|_{C^2(M)}$, 
 	$\partial M$ up to third derivatives and other known data, such that 
 	\begin{equation}
 		\begin{aligned}
 		    \Lambda_\mu (\mathfrak{g}[w])   \leqslant0  \mbox{ in } \Omega_\delta, \,\, u\leqslant w \mbox{ on } \partial \Omega_\delta. \nonumber
 		\end{aligned}
 	\end{equation}
 \end{lemma}
 
 \begin{proof}
 	By direct computation
 	\begin{equation}
 		\begin{aligned}
 			\Lambda_\mu (\mathfrak{g}[w])
 			=\,&\sum_{\alpha=1}^{n-1} \mu_\alpha T_{\alpha}^i \bar T_{\alpha}^j
 			(\chi_{i\bar j}+\underline{u}_{i\bar j}+\frac{\eta}{t_0}\sigma_{i\bar j}) 
 			+ 2Ad(z)\sum_{\alpha=1}^{n-1} \mu_\alpha T_{\alpha}^i \bar T_{\alpha}^j d_{i\bar j}
 			\\\,&
 			+ \sum_{\alpha=1}^{n-1} \mu_\alpha T_{\alpha}^i \bar T_{\alpha}^j (l(z)\sigma_{i\bar j}
 			+l_i\sigma_{\bar j}+\sigma_i l_{\bar j}).  \nonumber
 		\end{aligned}
 	\end{equation}
 	Here we use $T_\alpha d=0$  for $1\leqslant\alpha\leqslant n-1.$
 	
 	Next, we plan to estimate the quantity $\Lambda_\mu (\mathfrak{g}[w])$ in $\Omega_\delta$.
 	\begin{itemize}
 		\item
 		At the origin $(z=0)$, $T_{\alpha}^i=\delta_{\alpha i}$, 
 		\begin{equation}
 			\begin{aligned}
 				\sum_{\alpha=1}^{n-1} \mu_\alpha T_{\alpha}^i \bar T_{\alpha}^j
 				(\chi_{i\bar j}+\underline{u}_{i\bar j}+\frac{\eta}{t_0}\sigma_{i\bar j}) (0)
 				=\frac{1}{t_0}\sum_{\alpha=1}^{n-1}\mu_\alpha (A_{t_0})_{\alpha\bar\alpha}=0.  \nonumber
 			\end{aligned}
 		\end{equation}
 		So there are complex constants $k_i$ 
 		such that
 		\begin{equation}
 			\begin{aligned}
 				\sum_{\alpha=1}^{n-1} \mu_\alpha T_{\alpha}^i \bar T_{\alpha}^j
 				(\chi_{i\bar j}+\underline{u}_{i\bar j}+\frac{\eta}{t_0}\sigma_{i\bar j}) (z)=\sum_{i=1}^n (k_i z^i+ \bar k_{i} \bar z^i)+O(|z|^2). \nonumber
 			\end{aligned}
 		\end{equation}

 		\item    
 		Note that	\begin{equation}
 			\label{yuan3-buchong2}
 			\begin{aligned}
 				\sum_{\alpha=1}^{n-1} \mu_\alpha T_{\alpha}^i \bar T_{\alpha}^j(z)
 				=\sum_{\alpha=1}^{n-1} \mu_\alpha T_{\alpha}^i \bar T_{\alpha}^j(0)+O(|z|) =\sum_{\alpha=1}^{n-1} \mu_\alpha \delta_{\alpha i} \delta_{\alpha j}+O(|z|),
 			\end{aligned}
 		\end{equation}
 		\begin{equation}
 			\label{c3-yuan3}
 			\begin{aligned}
 				\sum_{\alpha=1}^{n-1}\mu_\alpha\sigma_{\alpha\bar\alpha}(z)=\sum_{\alpha=1}^{n-1}\mu_\alpha\sigma_{\alpha\bar\alpha}(0)+O(|z|).
 			\end{aligned}
 		\end{equation}  
 		Combining with \eqref{key-1-yuan3}, one can pick $\delta, \tau$ sufficiently small such that 
 		\begin{equation}
 			\begin{aligned}
 				\sum_{\alpha=1}^{n-1}\mu_\alpha T_{\alpha}^i \bar T_{\alpha}^j d_{i\bar j}
 				=\,&   \sum_{\alpha=1}^{n-1} \mu_\alpha(\sigma_{\alpha\bar\alpha}(z)+\tau)
 				+\sum_{\alpha=1}^{n-1}\mu_\alpha \left(T_{\alpha}^i \bar T_{\alpha}^j (z)- T_{\alpha}^i \bar T_{\alpha}^j (0)\right)d_{i\bar j} \\
 				\leqslant\,& -a_1+\tau+O(|z|) \leqslant -\frac{a_1}{4}.  \nonumber
 			\end{aligned}
 		\end{equation}
 		Consequently, we get
 		\begin{equation}
 			\begin{aligned}
 				2Ad(z)\sum_{\alpha=1}^{n-1} \mu_\alpha T_{\alpha}^i \bar T_{\alpha}^j d_{i\bar j} \leqslant -\frac{a_1A}{2}d(z). \nonumber
 			\end{aligned}
 		\end{equation}

 		\item  From 
 		$ \sum_{i=1}^n T_\alpha^i \sigma_i=-\tau\sum_{i=1}^n T_\alpha^i
 		\bar z^i $
 		we have
 		\begin{equation}
 			\begin{aligned}
 				\sum_{\alpha=1}^{n-1}\mu_\alpha T^i_\alpha \bar T^j_\alpha(l_i \sigma_{\bar j} + \sigma_i l_{\bar j})
 				=-\tau \sum_{\alpha=1}^{n-1} \mu_\alpha(\bar l_{\alpha} \bar z^\alpha+l_\alpha z^\alpha)
 				+O(|z|^2).   \nonumber
 			\end{aligned}
 		\end{equation}
 		On the other hand, by \eqref{yuan3-buchong2}, 
 		\begin{equation}	\begin{aligned}	l(z) \sum_{\alpha=1}^{n-1}\mu_\alpha T^i_\alpha \bar T^j_\alpha \sigma_{i\bar j} 	= l(z) \sum_{\alpha=1}^{n-1}\mu_\alpha\sigma_{\alpha\bar\alpha}(0)	+O(|z|^2).   \nonumber	\end{aligned}  	\end{equation}
 		Thus 
 		\begin{equation}
 			\begin{aligned}
 				\,& 
 				l(z)\sum_{\alpha=1}^{n-1}\mu_\alpha T^i_\alpha \bar T^j_\alpha
 				\sigma_{i\bar j}+\sum_{\alpha=1}^{n-1}\mu_{\alpha}T^i_\alpha \bar T^j_\alpha(l_i\sigma_{\bar j}
 				+\sigma_i l_{\bar j}) \\
 				= \,&
 				l(z) \sum_{\alpha=1}^{n-1}\mu_\alpha\sigma_{\alpha\bar\alpha}(0)
 				- \tau\sum_{\alpha=1}^{n-1} \mu_\alpha (z^\alpha l_\alpha+\bar z^\alpha \bar l_{\alpha})
 				+O(|z|^2). \nonumber
 			\end{aligned}
 		\end{equation}
 		
 	\end{itemize}
 	
 	Putting these together, 
 	\begin{equation}
 		\label{together1}
 		\begin{aligned}
 			\Lambda_\mu(\mathfrak{g}[w])\leqslant\,&
 			\sum_{\alpha=1}^{n-1} 2\mathfrak{Re}
 			\left\{z^\alpha \left(k_\alpha 
 			-\tau\mu_\alpha l_\alpha
 			+l_\alpha\sum_{\beta=1}^{n-1}\mu_\beta \sigma_{\beta\bar\beta}(0)\right)\right\}
 			\\ \,& + 2\mathfrak{Re}\left\{z^n
 			\left(k_n+l_n\sum_{\beta=1}^{n-1}\mu_\beta \sigma_{\beta\bar\beta}(0)\right)\right\} 
 			-\frac{Aa_1}{2}d(z) + O(|z|^2).   \nonumber
 		\end{aligned}
 	\end{equation}
 	Let $l_n=-\frac{k_n}{\sum_{\beta=1}^{n-1} \mu_\beta \sigma_{\beta\bar\beta}(0)}$.
 	For $1\leqslant \alpha\leqslant n-1$, we set
 	\begin{equation}
 		\label{chosen-1}
 		\begin{aligned}
 			l_\alpha=-\frac{k_\alpha}{\sum_{\beta=1}^{n-1}\mu_\beta \sigma_{\beta\bar\beta}(0)-\tau \mu_\alpha}.
 		\end{aligned}
 	\end{equation}
 	From $\mu_\alpha\geqslant 0$ and \eqref{key-1-yuan3}, we see
 	such $l_i$ (or equivalently the  $l(z)$) are all well-defined and uniformly bounded. 
 	We thus complete the proof if $0<\tau, \delta\ll1$, $A\gg1$.
 \end{proof}

 \subsubsection*{\bf Completion of the proof of Lemma \ref{keylemma1-yuan3}}
 Let $w$ be as in Lemma \ref{lemma-key2-yuan3}. From the construction above, we know that there is a uniform positive constant $C_1'$ such that
 $$ |\mathfrak{g}[w]|_{C^0(\Omega_\delta)}\leqslant C_1'.$$
 Denote $\lambda[w]=\lambda (\omega^{-1}\mathfrak{g}[w])$ and assume 
 $\lambda_1[w]\leqslant \cdots\leqslant \lambda_n[w]$.
 Together with Lemma \ref{lemma-yuan3-buchong1}, Lemma \ref{lemma-key2-yuan3} implies
 \begin{equation}
 	\begin{aligned}
 		\sum_{\alpha=1}^{n-1} \mu_\alpha \lambda_\alpha[w]\leqslant 0 \mbox{ in } \Omega_\delta.  \nonumber
 	\end{aligned}
 \end{equation}
 So  by \eqref{key-18-yuan3} $(\lambda_1[w],\cdots,\lambda_{n-1}[w])\notin\Gamma_\infty$. In other words, $\lambda[w]\in X$, where
 \[X=\{\lambda\in\mathbb{R}^{n}\setminus \Gamma:   |\lambda|\leqslant C_1'\}.\]
 
 Notice that  
 $X$ is a compact subset; furthermore $X\cap \bar\Gamma^{\inf_M\psi}=\emptyset$. (Note $\bar\Gamma^{\inf_M\psi}=\{\lambda\in\Gamma: f(\lambda)\geqslant \inf_M\psi\}$).
 So we can deduce that the distance between $\bar\Gamma^{\inf_M\psi}$ and $X$ 
 is greater than some positive constant depending on  
 $\delta_{\psi,f}$ and other known data.
 Therefore, there exists a positive constant $\epsilon_0$  such that for any $z\in\Omega_\delta$
 \begin{equation}
 	\begin{aligned}
 		\epsilon_0\vec{\bf 1} +\lambda[w]\notin \bar\Gamma^{\inf_M\psi}.  \nonumber
 	\end{aligned}
 \end{equation}
 
 Near the origin $z_0$ $(z=0)$, 
 under the coordinates
 \eqref{goodcoordinate1} 
 the distance can be expressed as 
 \begin{equation}
 	\label{asym-sigma1}
 	\begin{aligned}
 		\sigma(z)=x^n+\sum_{ i,j=1}^{2n} a_{ij} t^it^j+O(|t|^3). 
 	\end{aligned}
 \end{equation}
 Thus one can choose a  positive constant $C'$ such that $ x^n\leqslant C'|z|^2$ on 
 $\partial M\cap \bar{\Omega}_\delta$. As a result,  
 there is a positive constant $C_2$ depending only on 
 $\partial M$ 
 and $\delta$ so that  
 $$x^n\leqslant C_2 |z|^2 \mbox{ on }\partial\Omega_\delta.$$

 Let 
 $C_2$ be as above,  define ${h}(z)=w(z)+\epsilon (|z|^2-\frac{x^n}{C_2})$. 
 Thus
 \begin{equation}
 	\begin{aligned}
 		u\leqslant {h} \mbox{ on } \partial \Omega_\delta.  \nonumber
 	\end{aligned}
 \end{equation}
 From \eqref{asym-sigma1} we can choose 
 $0<\epsilon\ll1$ 
 so that $\lambda[{h}]\notin \bar\Gamma^{\inf_M\psi}$ in $\bar \Omega_\delta.$
 By \cite[Lemma B]{CNS3}, we have $$u\leqslant {h} \mbox{ in }\Omega_\delta.$$
 Notice 
 $u(0)=\varphi(0)={h}(0)$, we have $u_{x^n}(0)\leqslant h_{x^n}(0)$. Thus
 \begin{equation}
 	\begin{aligned}
 		t_0\leqslant \frac{1}{1+\epsilon/(\eta C_2)}, \mbox{ i.e., }  (1-t_0)^{-1}\leqslant 1+\frac{\eta C_2}{\epsilon}.  \nonumber
 	\end{aligned}
 \end{equation}

  \medskip
  
  \section{Quantitative boundary estimate and proof of Theorem \ref{thm1-bdy}}
   \label{sec1-Bdy-estimate}

 \subsection{Double normal derivative case}
 \label{sec1-Bdy-normalnormal}


  Fix $z_0\in\partial M$.
We use local coordinate \eqref{goodcoordinate1} with origin at $z_0 $ $(z=0)$. 
 Let $t_0$ be as defined in \eqref{key0-yuan3}.
 Fix $\varepsilon_0$ as in \eqref{varepsilon-0-a0}, i.e.,  
 \begin{equation} 
 	\begin{aligned}
 		\underline{\lambda}-\varepsilon_0\vec{\bf 1}\in\Gamma. \nonumber
 	\end{aligned}
 \end{equation}
In what follows the discussion will be given at $z_0$, where $g_{i\bar j}(0)=\delta_{ij}$, $\frac{\partial}{\partial x^n}\big|_{z=0}$ is the interior normal to $\partial M$.
For $R>0$,  define the matrix  
\begin{equation}
	\label{def1-AR}
	\begin{aligned}
		{A}(R) =
		\left( \begin{matrix}
			\mathfrak{{g}}_{\alpha\bar \beta} &\mathfrak{g}_{\alpha\bar n}\\
			\mathfrak{g}_{n\bar \beta}& R  
		\end{matrix}\right).
		\end{aligned}
\end{equation}
 By 
 \eqref{410-buchong} we can decompose $A(R)$ into
 \begin{equation}
 	\begin{aligned}
 		{A}(R) = A'(R)+A''(R) 
 	\end{aligned}
 \end{equation}
 where 
 \[A'(R)=\left(\begin{matrix}
 	(1-t_0)(\mathfrak{\underline{g}}_{\alpha\bar \beta}-\frac{\varepsilon_0}{4}\delta_{\alpha\beta}) &\mathfrak{g}_{\alpha\bar n}\\
 	\mathfrak{g}_{n\bar \beta}& R/2  \nonumber
 \end{matrix}\right), \mbox{  }
 A''(R)=\left(\begin{matrix} (A_{t_0})_{\alpha\bar \beta}+\frac{(1-t_0)\varepsilon_0}{4} \delta_{\alpha\beta} &0\\ 0& R/2  \end{matrix}\right).\]
 
 In order to complete the proof, we 
 need to prove that
 \begin{lemma}
 	\label{lemma1-bdyestimate}
 	There exist uniform constants $R_3$ and $C_{A''}$ depending on
 	$(1-t_0)^{-1}$, $\varepsilon_0^{-1}$,   $f$ and other known data under control, such that 
 	\begin{equation}
 		\label{at0-2}
 		\begin{aligned}
 			f(2\lambda[A''(R_3)] )\geqslant -C_{A''}.
 		\end{aligned}
 	\end{equation}

 \end{lemma}
 
 \begin{proof}

 	Let     	
 		$\underline{\lambda}'$ and 
 	$\tilde{\lambda}'$ be as in \eqref{yuan3-buchong6} and \eqref{yuan3-buchong3}, respectively.
 	One can see that there is a uniform constant $C_0>0$ depending on $|t_0|$,
 	$\sup_{\partial M}|\nabla u|$ and other known data, such that $|\tilde{\lambda}'|\leqslant C_0$, that is
 	$\tilde{\lambda'}$ is contained in a compact subset of $\overline{\Gamma}_\infty$, i.e.,
 	\begin{equation}
 		\label{yuan3-buchong4v}
 		\begin{aligned}
 			\tilde{\lambda}'\in {K}\equiv \{\lambda'\in \overline{\Gamma}_\infty: |\lambda'|\leqslant C_0\}.  \nonumber
 		\end{aligned}
 	\end{equation}
 	Thus  there is a 
 positive constant $R_2$ depending on $((1-t_0)\varepsilon_0)^{-1}$, $K$ and other known data, 
 	such that  
 	\begin{equation}
 		\label{key-74}
 		\begin{aligned}
 			\lambda	\left[\begin{pmatrix} (A_{t_0})_{\alpha\bar \beta}+\frac{(1-t_0)\varepsilon_0}{8} \delta_{\alpha\beta} &0\\ 0& R_2/2  \end{pmatrix}\right] \in \Gamma.
 		\end{aligned}
 	\end{equation}
 	Let $R_3=R_2+ \frac{(1-t_0)\varepsilon_0}{4}$, we have
 	\begin{equation}
 		\label{at0-1}
 		\begin{aligned}
 			\lambda[A''(R_3)] \in \Gamma+\frac{(1-t_0)\varepsilon_0}{8} \vec{\bf1}.
 		\end{aligned}
 	\end{equation} 
 	Notice that $A''(R_3)$ is bounded.
There exists a uniform constant $C_{A''}$ depending on  $(1-t_0)^{-1}$, $\varepsilon_0^{-1}$, $R_3$  and $f$, such that
 	\eqref{at0-2} holds.
 	
 \end{proof}
 
\begin{proof}
	[Proof of Proposition \ref{proposition-quar-yuan2}]
 By the unbounded condition \eqref{unbounded-1},
 there is a  positive constant $R_1$ depending on $(1-t_0)^{-1}$,  
 $\varepsilon_0$, $\underline{\lambda}'$ and 
 $\sup_{\partial M}\psi$, $C_{A''}$
 such that  	 
 \begin{equation}
 	\label{key-03-yuan3}
 	\begin{aligned} 	\frac{1}{2}f\left(2(1-t_0)(\underline{\lambda}'_{1}-{\varepsilon_0}/2),\cdots, 2(1-t_0)(\underline{\lambda}'_{n-1}
 		-{\varepsilon_0}/2), 2R_1\right) \geqslant  f(\underline{\lambda})+\frac{C_{A''}}{2}, 
 	\end{aligned}
 \end{equation}
 and  $(\underline{\lambda}'_1-\varepsilon_{0}/2, \cdots, \underline{\lambda}'_{n-1}-\varepsilon_0/2, {(1-t_0)^{-1}}{R_1})\in \Gamma.$

 Let's pick  $\epsilon=\frac{(1-t_0)\varepsilon_0}{4}$ in 
 Lemma  \ref{yuan's-quantitative-lemma}. And then 
 we set
 \begin{equation}
 	\label{def1-R-c}
 	\begin{aligned}
 		R_c\equiv \,& \frac{8(2n-3)}{(1-t_0)\varepsilon_0}\sum_{\alpha=1}^{n-1} | \mathfrak{g}_{\alpha\bar n}|^2 	
 		+ 2(n-1) (1-t_0) \sum_{\alpha=1}^{n-1}|{\underline{\lambda}_\alpha'}|
 		\\\,&	
 		+\frac{n(n-1)(1-t_0)\varepsilon_0}{2}
 		+2R_1+2R_2,  
 	\end{aligned}
 \end{equation}
 where $\varepsilon_0$, $R_1$ and $R_2$ are the constants as we 
 fixed in \eqref{key-03-yuan3} and \eqref{at0-1}.
 
 According to Lemma \ref{yuan's-quantitative-lemma},
 the eigenvalues $\lambda[{A}'(R_c)]$ of ${A}'(R_c)$
 (possibly with an appropriate order)
 shall behave like
 \begin{equation}
 	\label{lemma12-yuan}
 	\begin{aligned}
  \lambda_{\alpha}[{A}'(R_c)]\geqslant \,& (1-t_0)  (\underline{\lambda}'_{\alpha}-{\varepsilon_0}/{2}), \mbox{  } 1\leqslant \alpha\leqslant n-1,  	
\end{aligned}
\end{equation} 
\begin{equation}
\label{lemma12-yuan-5}
\begin{aligned}
 		\lambda_{n}[{A}'(R_c)]\geqslant R_c/2.
 	\end{aligned}
 \end{equation}
 In particular, $\lambda[{A}'(R_c)]\in \Gamma$. So $\lambda[A(R_c)]\in \Gamma$.  
 By the concavity, Lemma  \ref{lemma1-bdyestimate}
 and \eqref{key-03-yuan3},  
 \eqref{lemma12-yuan}, \eqref{lemma12-yuan-5}, we deduce that
 \begin{equation}
 	\label{yuan-k1}
 	\begin{aligned}
 		f(\lambda[A(R_c)])\geqslant 
 		 \frac{1}{2}f(2\lambda[A'(R_c)])+\frac{1}{2}f(2\lambda[A''(R_c)]) 
 	\geqslant f(\underline{\lambda})\geqslant \psi.
 	\end{aligned}
 \end{equation}
Thus $\mathfrak{g}_{n\bar n} \leqslant R_c.$  
By Lemma \ref{keylemma1-yuan3}, 
 $(1-t_0)^{-1}$ is bounded. We complete the proof.

 \end{proof}

When $f$ satisfies \eqref{addistruc},
applying Lemma \ref{lemma3.4}  we have  immediately 
\begin{lemma} 
	\label{lemma3.4-2} 
	Let $R_c$ be as in \eqref{def1-R-c}. Assume in addition that $f$ satisfies \eqref{addistruc}. Then
	\begin{equation}
		\label{yuan-k2}
		\begin{aligned}
			f(\lambda[A(R_c)])\geqslant f(\lambda[A'(R_c)]). 
		\end{aligned}
	\end{equation}
	
\end{lemma}

\begin{corollary}
	\label{rmk-10}
	When $f$ further satisfies \eqref{addistruc}, 
	the unbounded condition \eqref{unbounded-1} imposed in Theorems \ref{thm1-1} and \ref{thm1-bdy} can be replaced by a slightly broader condition  
	\begin{equation}
		\label{unbounded-2}
		\begin{aligned}
			\lim_{t\rightarrow+\infty}  f(\lambda_1,\cdots,\lambda_{n-1},\lambda_n+t)=\sup_\Gamma f, \,\, \forall \lambda=(\lambda_1,\cdots,\lambda_n)\in\Gamma.
		\end{aligned}
	\end{equation}  
\end{corollary}

 \begin{proof}
	The unbounded condition \eqref{unbounded-2} implies that 
	\begin{equation}
		\label{key-03-yuan5}
		\begin{aligned}
			f\left((1- t_0)(\underline{\lambda}'_{1}-{\varepsilon_0}),\cdots, (1-t_0)(\underline{\lambda}'_{n-1}
			-{\varepsilon_0}), R_1\right)\geqslant  f(\underline{\lambda}).
	\end{aligned} 	\end{equation}
	So when $f$ satisfies \eqref{addistruc},  we can use \eqref{key-03-yuan5} and \eqref{yuan-k2}, in place of  \eqref{key-03-yuan3} and  \eqref{yuan-k1} respectively.
 	Consequently, in this case
	\eqref{unbounded-1} can be replaced by  
	\eqref{unbounded-2}.  
\end{proof}

 
 \subsection{Tangential-Normal derivatives case}
 \label{sec-bdyestimate-mixed-derivatives}
 
 
We use the usual
 subsolution method
 of \cite{Guan1993Spruck,Guan1998The}  (further refined by  \cite{Guan12a,Guan-Dirichlet}). 
 In order to derive \eqref{quanti-mix-derivative-00}, we shall construct more delicate local barriers near boundary.  
 The  specific instance of such local barriers was investigated  by
 \cite{Boucksom2012,Chen,Phong-Sturm2010} for   complex Monge-Ampère equation, and further by \cite{Collins2019Picard}
 for more general  complex $k$-Hessian equations.  
 The utilization of local barriers for general equations   
  was introduced by
  \cite{yuan2020PAMQ,yuan2017}  in some case. 
 
 
We will highlight that the expression  
 $\big|\sum f_i\lambda_i\big|$ in Lemma \ref{lemma5-key} offers a simplified method for managing the problematic terms that incorporate $\sum f_i|\lambda_i|$. 
This enhances some of the techniques mentioned earlier.


 Under the local coordinate \eqref{goodcoordinate1} with origin at $z_0 \in\partial M$. 
  As in \eqref{t-coordinate} we set
 \begin{equation}
  \begin{aligned}
 	t^{2\alpha-1}=x^{\alpha}, \ t^{2\alpha}=y^{\alpha},\ 1\leqslant \alpha\leqslant n-1;\ t^{2n-1}=y^{n},\ t^{2n}=x^n.  \nonumber
 	\end{aligned}
 \end{equation}
As we denoted in \eqref{def1-K},  
 $$K=1+\sup_{M}  (|\nabla u|_\omega^{2}+ |\nabla (u-\underline{u})|_\omega^{2}).$$ 
 
  We use the notation  in  \eqref{formula-1}.
 By direct calculations, one derives  
 \begin{equation}
 	\begin{aligned}
 		(u_{x^k})_{\bar j} 
 		=u_{\bar j x^k}+\overline{\Gamma_{kj}^l} u_{\bar l},  \,\,
 		(u_{y^k})_{\bar j}
 		=u_{\bar j y^k}-{\sqrt{-1}}\overline{\Gamma_{kj}^l} u_{\bar l}, \nonumber
 	\end{aligned}
 \end{equation}
 \begin{equation}
 	\begin{aligned}
 		(u_{x^k})_{i\bar j}  
 		=	 	u_{i\bar j x^k}+\Gamma_{ki}^lu_{l\bar j}+\overline{\Gamma_{kj}^l} u_{i\bar l}, 
 		\,\, \nonumber
 		(u_{y^k})_{i\bar j} 
 		=u_{i\bar j y^k}
 		+\sqrt{-1}		\left(\Gamma_{ki}^l u_{l\bar j}-\overline{\Gamma_{kj}^l} u_{i\bar l}\right).	\nonumber  	\end{aligned} \end{equation}

 We set
 \begin{equation}
 	\label{barrier1}
 	\begin{aligned}
 		w= (\underline{u}-u)
 		- t\sigma
 		+N\sigma^{2}   \mbox{  in  } \Omega_{\delta}.  \nonumber
 	\end{aligned}
 \end{equation}
 Here $N$ is a positive constant to be determined, $\delta$ and $t$ are small enough such that   
 \begin{equation}
 	\label{bdy1}
 	\begin{aligned}
 		\mbox{  $\sigma$ is smooth, }
 		\frac{1}{4} \leqslant |\nabla\sigma|\leqslant 2,  \,\,
 		|\mathcal{L}\sigma | \leqslant   C_\sigma\sum_{i=1}^n f_i 
 		\mbox{ in }  \Omega_{\delta},
 	\end{aligned}
 \end{equation} 
 \begin{equation}
 	\label{yuanbd-11}
 	\begin{aligned}
 		N\delta-t\leqslant 0, \mbox{ }	 \max\{|2N\delta-t|, t\}
 		\leqslant \frac{\varepsilon}{32C_\sigma},
 	\end{aligned}
 \end{equation}
 for some  constant $C_\sigma>0,$
 where 
 $\varepsilon$ is the constant asserted in Lemma \ref{lemma3-key}.
 Clearly,
 \begin{equation}
 	\label{yiqi-1}
 	\begin{aligned}
 		w\leqslant 0    \mbox{ in } \Omega_{\delta},  
 \end{aligned}  \end{equation}
 \begin{equation}	\begin{aligned}
 		\mathcal{L}w
 		= F^{i\bar j}(\mathfrak{\underline{g}}_{i\bar j}-\mathfrak{g}_{i\bar j}) 	+2NF^{i\bar j}\sigma_i\sigma_{\bar j}	+(2N\sigma-t)\mathcal{L}\sigma
 		\mbox{ in } \Omega_{\delta}.  
 	\end{aligned}
 \end{equation}

Define the tangential operator on the boundary
 \begin{equation}
 	\label{tangential-oper-general1}
 	\begin{aligned} 
 		\mathcal{T}=\nabla_{\frac{\partial}{\partial t^{\alpha}}}- \widetilde{\eta}\nabla_{\frac{\partial}{\partial x^n}}, 
 		\mbox{ for each fixed }
 		1\leqslant \alpha< 2n,  \nonumber
 	\end{aligned}
 \end{equation}
 where $\widetilde{\eta}= {\sigma_{t^{\alpha}}}\big/{\sigma_{x^n}}$. 
 One has $\mathcal{T}(u-\varphi)=0$ on $\partial M\cap \bar\Omega_\delta$. 
 By 
 $\widetilde{\eta}(0)=0$ one derives $|\widetilde{\eta}|\leqslant C'|z|$ on $\bar\Omega_\delta$. 
 Since  $(u-\varphi)\big|_{\partial M}=0$  we obtain
 $\mathcal{T}(u-\varphi)\big|_{\partial M}=0$. Together with the boundary gradient estimate contained in Lemma \ref{lemma-c0-bc1},
 one has
 \begin{equation}\label{bdr-t}\begin{aligned}
 		|(u-\varphi)_{t^{\alpha}}|\leqslant C|z| \mbox{ on } \partial M\cap\bar \Omega_\delta,
 		\mbox{  } \forall 1\leqslant \alpha<2n. 
 	\end{aligned}
 \end{equation}

 Take   
 \begin{equation}
 	\label{Phi-def1}
 	\begin{aligned}
 		\Phi=\pm \mathcal{T}(u-\varphi)+\frac{1}{\sqrt{K}}(u_{y^{n}}-\varphi_{y^{n}})^2 \,\, \mbox{ in } \Omega_\delta.  \nonumber
 	\end{aligned}
 \end{equation}
 Note that 
 $u_{x_n}=2u_n+\sqrt{-1}u_{y_n}
 $ and $u_{x_n}=2u_{\bar n}-\sqrt{-1}u_{y_n}$. 
 Combining Cauchy-Schwarz inequality, we can prove 
 \begin{equation}
 	\label{yuan-1}
 	\begin{aligned}
 		\mathcal{L}\Phi \geqslant 
 		-C_{\Phi}  \sqrt{K}  \sum_{i=1}^n f_i  - C_{\Phi}  \sum_{i=1}^n f_i|\lambda_i|
 		-C_\Phi  \mbox{ on } \Omega_{\delta}. 
 	\end{aligned}
 \end{equation}
In the proof of \eqref{yuan-1} we use 
 \begin{equation}	\begin{aligned}		2|\mathfrak{Re}(F^{i\bar j}(\widetilde{\eta})_i (u_{x^n})_{\bar j})|	\leqslant  \,&  \frac{1}{\sqrt{K}}F^{i\bar j}(u_{y^n})_{i} (u_{y^n})_{\bar j} +C\sum_{i=1}^nf_i|\lambda_i|+C\sqrt{K}\sum_{i=1}^n f_i.		\nonumber	\end{aligned}\end{equation}

 By straightforward computations and \cite[Proposition 2.19]{Guan12a}, there is an index $r$  with $1\leqslant r\leqslant n$ such that
 \begin{equation}
 	\label{inequ-buchong1}
 	\begin{aligned}
 		\mathcal{L} \left(\sum_{\tau<n}|(u-\varphi)_{\tau}|^2 \right)  
 		\geqslant \,& 	\frac{1}{2}\sum_{\tau<n}F^{i\bar j} \mathfrak{g}_{\bar \tau i} \mathfrak{g}_{\tau \bar j} 	-C_1\sqrt{K}  \left(\sum_{i=1}^n f_{i}|\lambda_{i}|+\sum_{i=1}^n f_{i}	\right)  +   2\sum_{\tau<n}\mathfrak{Re} (u_\tau \psi_{\bar \tau}) \\  
 		\geqslant  \,&
 		\frac{1}{4}\sum_{i\neq r} f_{i}\lambda_{i}^{2}
 		-C_1'\sqrt{K} \left( 1+\sum_{i=1}^n f_{i} +\sum_{i=1}^n f_{i}|\lambda_{i}| \right). 
 		\nonumber
 	\end{aligned}
 \end{equation}

In order to treat the terms involving $\sum f_i|\lambda_i|$,
applying Cauchy-Schwarz inequality,  we can derive  
	\begin{equation}
 \label{inequality1-bdy-esti}
	\begin{aligned} 
		\sum f_i |\lambda_i|
	 =\,& \sum_{i=1}^n f_i\lambda_i-2\sum_{\lambda_i<0} f_i\lambda_i	=  2\sum_{\lambda_i\geq 0} f_i\lambda_i-\sum_{i=1}^n f_i\lambda_i	\\ 
		\leqslant \,&
		\epsilon \sum_{i\neq r} f_i \lambda_i^2+   \epsilon^{-1}  \sum f_i  + \big|\sum_{i=1}^n f_i \lambda_i\big|, \, \forall \epsilon>0. 
	\end{aligned}
\end{equation} 
 
The local barrier function in $\Omega_\delta$ is as follows:
\begin{equation}
 	\label{Psi}
 	\begin{aligned} 
 		\widetilde{\Psi} =A_1 \sqrt{K}w -A_2 \sqrt{K} |z|^2 + A_3 \Phi+ \frac{1}{\sqrt{K}} \sum_{\tau<n}|(u-\varphi)_{\tau}|^2.   \nonumber
 	\end{aligned}
 \end{equation}
 Putting the above inequalities and \eqref{L1-u-ubar} together,  
 we obtain in $\Omega_\delta$
  \begin{equation}
 	\label{bdy-main-inequality}
 	\begin{aligned}
 		\mathcal{L}\widetilde{\Psi} \geqslant \,&
 		 A_1 \sqrt{K}  F^{i\bar j}(\mathfrak{\underline{g}}_{i\bar j}-\mathfrak{g}_{i\bar j})  
 		+ 2A_1N \sqrt{K} F^{i\bar j}\sigma_i \sigma_{\bar j} 
 		\\ \,&
 		-  \left( A_2+A_3C_\Phi  +A_1C_\sigma |2N\sigma-t|
 		+4(C_1'+A_3C_\Phi)^2 \right) \sqrt{K}\sum_{i=1}^n f_i 
 		\\ \,&
 		-(C_1'+A_3 C_\Phi ) 
 		\left(1+\big|\sum f_i\lambda_i\big| \right)
 		- C_1'   \sum_{i=1}^n f_i.
 	\end{aligned}
 \end{equation}
 Here we set $\epsilon=\frac{1}{4\sqrt{K}(A_3C_
 	\Phi+C_1')}$ in \eqref{inequality1-bdy-esti}.

 Proposition \ref{mix-general} then follows from
 the following lemma.
 \begin{lemma}
 	There are constants $A_1\gg A_2\gg A_3^2\gg1$,  $0<\delta\ll1$, $N\gg1$ such that  $\widetilde{\Psi}(0)=0$,
 	$\widetilde{\Psi}\big|_{\partial {\Omega_\delta}}\leqslant 0$, and 
 	\begin{equation}
 		\label{mainiequality-1} \mathcal{L}\widetilde{\Psi}> 0 \mbox{ on } \Omega_\delta.
 	\end{equation}
 	
 \end{lemma}

 \begin{proof}
 	Obviously $\widetilde{\Psi}(0)=0$.
 Using \eqref{yiqi-1}, \eqref{bdr-t} and  $\mathcal{T}(u-\varphi)\big|_{\partial M}=0$ we can infer that 
 $\widetilde{\Psi}\big|_{\partial {\Omega_\delta}}\leqslant 0$.
 
Next we will prove \eqref{mainiequality-1}.
 	The discussion can be divided into two cases according to Lemma \ref{lemma3-key}.  

 	\noindent
 	{\bf Case 1}:  In the first case we have
 	\begin{equation}
 		\begin{aligned}
 			F^{i\bar j}(\underline{\mathfrak{g}}_{i\bar j}-\mathfrak{g}_{i\bar j})
 			\geqslant  \varepsilon+\varepsilon F^{i\bar j}g_{i\bar j}+\varepsilon 
 			\big|F^{i\bar j}\mathfrak{g}_{i\bar j}\big|.  \nonumber
 		\end{aligned}
 	\end{equation}
Furthermore, by \eqref{yuanbd-11} we see that $C_\sigma |2N\sigma-t| \leqslant \frac{\varepsilon}{32}.$

 	From \eqref{bdy-main-inequality} we know that \eqref{mainiequality-1} holds if $A_1\gg A_3^2\gg1$ and $A_1\gg A_2\gg1$.

 	\noindent
 	{\bf Case 2}:  
 	In the rest case, by Lemma \ref{lemma3-key}
 	 we have
 	\begin{equation}
 		\begin{aligned}
 			F^{i\bar j}\geqslant \varepsilon (1+F^{p\bar q}g_{p\bar q}+	\big|F^{p\bar q}\mathfrak{g}_{p\bar q}\big|)g^{i\bar j}.
 			\nonumber
 		\end{aligned}
 	\end{equation} 
  From \eqref{bdy1}, we have $|\nabla \sigma|\geqslant\frac{1}{4}$ in $\Omega_\delta$, and then 	\begin{equation}		\label{bbvvv}		\begin{aligned}			F^{i\bar j}\sigma_i \sigma_{\bar j} \geqslant \frac{ \varepsilon }{16} \left(1+\sum_{i=1}^n f_i +\big|\sum_{i=1}^nf_i \lambda_i\big| \right) \mbox{ on } \Omega_\delta  \nonumber		\end{aligned} 	\end{equation}
  for sufficiently small $\delta>0$.
 	Note that
 	$	F^{i\bar j}(\underline{\mathfrak{g}}_{i\bar j}-\mathfrak{g}_{i\bar j})\geqslant 0$.
 	From \eqref{bdy-main-inequality} and  \eqref{yuanbd-11},  
 	 we can conclude that
  \begin{equation}  		
  	\label{bdy-main-inequality-2}
 	\begin{aligned}
 		\mathcal{L}\widetilde{\Psi} 
 		\geqslant \,& 
 		\frac{A_1N \varepsilon}{8}  
 		\sqrt{K}  \left(1+\sum_{i=1}^n f_i+\big|\sum_{i=1}^nf_i \lambda_i\big| \right)
 		-(C_1'+A_3 C_\Phi ) (1+\big|\sum_{i=1}^nf_i \lambda_i\big|)
 		\\\,&
 		- C_1' \sum_{i=1}^n f_i  
 		-  \left( A_2+A_3C_\Phi  + \frac{A_1 \varepsilon }{32}
 		+4(C_1'+A_3C_\Phi)^2 \right) \sqrt{K}\sum_{i=1}^n f_i 
 	 >0 \nonumber
 	\end{aligned}
 \end{equation}
  on $\Omega_\delta$,
   provided $A_1\gg A_3^2\gg1$,  $A_1\gg A_2\gg1$,    $0<\delta\ll1$ and $N\geqslant 1$.

 \end{proof}

\subsection{Tangential-Normal derivatives case revisited}
\label{sec-bdyestimate-mixed-derivatives-revisited}

 We say that $\partial M$ is \textit{holomorphically flat},  
if  for any $z_0\in \partial M$  one can pick local 
complex coordinates  
\begin{equation}
	\begin{aligned}
		\label{holomorphic-coordinate-flat}
		(z^1,\cdots, z^n), \mbox{  } z^i=x^i+\sqrt{-1}y^i, 
	\end{aligned}
\end{equation}
centered at $z_0$ such that
$\partial M$ is locally of the form $\mathfrak{Re}(z^n)=0$. 

 In this case,  we have a delicate result:  
 The estimate depends on $\partial M$ up to second derivatives, 
 rather than third derivatives.

\begin{proposition}
	\label{mix-Leviflat}
	In addition to  \eqref{elliptic}, \eqref{concave}, 
	\eqref{nondegenerate} and \eqref{subsolution1},  
	we assume that $\partial M$ is holomorphically flat.
	Then  for any  admissible solution $u\in C^3(M)\cap C^2(\bar M)$ to the Dirichlet problem  \eqref{mainequ-Dirichlet} satisfies
	\begin{equation}
		\label{quanti-mix-derivative-1}
		|\mathfrak{g}_{\alpha \bar n}(z_0)|\leqslant C(1+\sup_{M}|\nabla u|), \, \forall z_0\in\partial M. 
	\end{equation}
	where $C$ depends on 
	$|\psi|_{C^{1}(\bar M)}$, $|\underline{u}|_{C^{2}(\bar M)}$, $|\varphi|_{C^{3}(\bar M)}$, 
	$\partial M$ 
	up to second derivatives
	and other known data (but neither on $(\delta_{\psi,f})^{-1}$ nor on $\sup_{M}|\nabla u|$). 
\end{proposition}

\begin{proof}[Sketch of the proof]

	Under the 
 	complex coordinates \eqref{holomorphic-coordinate-flat}, we can take
	\begin{equation}
		\label{tangential-oper-Leviflat1}
		\begin{aligned}
			D=   \frac{\partial}{\partial x^\alpha}, \mbox{ } \frac{\partial}{\partial y^\alpha},  
			\,\, 1\leqslant  \alpha\leqslant  n-1.  \nonumber
		\end{aligned}
	\end{equation}
	(Furthermore, we may assume $g_{i\bar j}=\delta_{ij}$ at $z_0$).
	%
We use the local barriers  
	$$\widetilde{\Psi} =A_1 \sqrt{K}w -A_2 \sqrt{K} |z|^2 + \frac{1}{\sqrt{K}} \sum_{\tau<n}|({u}-\varphi)_{\tau}|^2\pm A_3 {D}(u-\varphi).$$  
	The rest proof is similar to that 
	of 
	 Subsection \ref{sec-bdyestimate-mixed-derivatives}.
\end{proof}

\medskip
\section{The Dirichlet problem for degenerate equations}
\label{sec1-degenerate} 

 In this section, we consider the Dirichlet problem for degenerate equations. 
 \begin{definition}
 	 We call $\underline{u}$ a \textit{strict  subsolution} to   \eqref{mainequ-Dirichlet}, if for some $\delta>0$
 	 \begin{equation} 
 	 	\label{existenceofsubsolution-2}
 	 	\begin{aligned} 
 	 		F(\mathfrak{g}[\underline{u}])
 	 		\geqslant \psi+\delta \text { in } \bar{M}, \,\, \underline{u}=\varphi \text { on } \partial M. 
 	 	\end{aligned}
 	 \end{equation}
 \end{definition}

Let $\Gamma_{\mathcal{G}}^{f}$ be as in \eqref{component1}. 
As denoted in Introduction, 
Let
 $\Gamma_{\mathcal{G},\infty}^{f}$ be the projection of $\mathring{\Gamma}_{\mathcal{G}}^{f}$ to the subspace of former $n-1$ subscripts, that is
\begin{equation} 
	\begin{aligned}
		\Gamma_{\mathcal{G},\infty}^{f}=\left\{\left(\lambda_{1}, \cdots, \lambda_{n-1}\right): \left(\lambda_{1}, \cdots, \lambda_{n-1}, R\right) \in \mathring{\Gamma}_{\mathcal{G}}^{f} \mbox{ for some } R>0
\right\}.  \nonumber
	\end{aligned}
\end{equation}  
 Also,  $\overline{\Gamma}_{\mathcal{G},\infty}^{f}$ denotes the closure of $\Gamma_{\mathcal{G},\infty}^{f}$.   
 
Throughout this section we assume that the boundary   satisfies
 \begin{equation}\label{bdry-assumption1}
 	\begin{aligned}
 		\left(-\kappa_{1}, \cdots,-\kappa_{n-1}\right) \in \overline{\Gamma}_{\mathcal{G},\infty}^{f} \mbox{ on } \partial M,
 	\end{aligned}
 \end{equation} 
where $\kappa_{1}, \cdots, \kappa_{n-1}$ denote the eigenvalues of Levi form $L_{\partial M}$ of $\partial M$ with respect to $\omega^{\prime}=\left.\omega\right|_{T_{\partial M} \cap J T_{\partial M}}$.  As above, $J$ denotes the underlying complex structure.

 We first drop the unbounded condition imposed in Theorem~\ref{thm1-1}.

\begin{theorem}
	\label{thm1-2}
	In Theorem~\ref{thm1-1} the hypothesis~\eqref{unbounded-1} on $f$ can be dropped if
	the Levi form $L_{\partial M}$ 
	satisfies 
	\eqref{bdry-assumption1} and $f$ obeys \eqref{elliptic}.
\end{theorem}

It is a remarkable fact that the boundary estimate 
is independent of $(\delta_{\psi,f})^{-1}$ when $\partial M$ satisfies~\eqref{bdry-assumption1}, as shown in Theorem \ref{thm2-bdy}.
As a result,
we can treat the Dirichlet problem for 
degenerate equations.  

\begin{theorem} \label{thm2-diri-de}
	Let $(M,\omega)$ be a compact Hermitian manifold with smooth boundary subject to \eqref{bdry-assumption1}, and let $f$ satisfy \eqref{elliptic},  \eqref{concave}   and~\eqref{continuity1}.
	Assume  $\varphi \in C^{2,1}(\partial M)$ and $\psi \in C^{1,1}(\bar{M})$ satisfy \eqref{degenerate-RHS} and support a strict,  admissible subsolution $\underline{u}\in C^{2,1}(\bar{M})$. Then there exists a (weak) solution  $u \in C^{1, \alpha}(\bar{M})$ to the Dirichlet problem \eqref{mainequ-Dirichlet} with $\forall 0<\alpha<1$ such that
	\[
	\lambda(\omega^{-1}\mathfrak{g}[u]) \in \bar{\Gamma} \mbox{ in }  \bar{M}, \,\,    \Delta u \in L^{\infty}(\bar{M}).
	\]
	 
\end{theorem}

 
This shows that  one may solve the Dirichlet problem for degenerate equations on manifolds with $C^{2,1}$ boundary data. 
 Moreover, the weak solution to Dirichlet problem of degenerate equations on products was also obtained in Theorem \ref{mainthm-10-degenerate} below, where both boundary and boundary data are in class of $C^{2,1}$.
 This is in contrast with bounded domain $\Omega\subsetneq\mathbb{R}^n$ case, 
in which  
 Caffarelli-Nirenberg-Spruck  \cite{CNS-deg} showed that  $C^{3,1}$-regularity assumptions on boundary and boundary value are optimal for $C^{1,1}$ global regularity of weak solution to homogeneous real Monge-Amp\`ere equation.  
Note also that the principal curvatures of bounded domain do not satisfy one similar to \eqref{bdry-assumption1}.
Hence we believe that  the assumption 
\eqref{bdry-assumption1} 
is reasonable.

\subsection{Quantitative boundary estimate revisited}
\label{sec1-Bdy-estimate-revisited}


The boundary estimate for double normal derivative is fairly delicate when the boundary satisfies   \eqref{bdry-assumption1}.

\begin{proposition}
	\label{proposition-quar-yuan1} 
	Let $(M, \omega)$ be a compact Hermitian manifold  with $C^2$ 
	boundary satisfying \eqref{bdry-assumption1}.
	Suppose, in addition to \eqref{elliptic}, 
	\eqref{concave} and  \eqref{nondegenerate}, that the Dirichlet problem \eqref{mainequ-Dirichlet} admits a $C^2$ {\em admissible} subsolution $\underline{u}$. 
	Let $u\in  C^{2}(\bar M)$ be an \textit{admissible} solution  to the Dirichlet problem, then
	\begin{equation}
		\begin{aligned}
			\mathfrak{g}_{n\bar n}(z_0)
			\leqslant  C\left(1 +  \sum_{\alpha=1}^{n-1} |\mathfrak{g}_{\alpha\bar n}(z_0)|^2\right),  \,\,\forall z_0\in\partial M.  \nonumber
		\end{aligned}
	\end{equation}
	where $C$  depends only on   $|u|_{C^0(\bar M)}$, 	$|\nabla u|_{C^0(\partial M)}$,	$|\underline{u}|_{C^{2}(\bar M)}$, $\partial M$ up to second order derivatives 	and other known data. Moreover, the constant is independent of $(\delta_{\psi,f})^{-1}$.
	
\end{proposition}

\begin{proof}
	The proof is based on Lemma \ref{lemma2-key}.
	Near $z_0\in\partial M$, we choose the local complex coordinate \eqref{goodcoordinate1} with origin at $z_0$ ($z=0$). In what follows, the discussion will be given at $z_0$. 
	As in \eqref{def1-eta}, we denote $\eta=(u-\underline{u})_{x^n}(0)$. The maximum principle tells  us that $\eta\geqslant 0.$ 
	
First, there exist  uniform positive constants $\varepsilon_1$ and $R_1$
such that 	
\begin{equation}
	\label{key-03-yuan6}
	\begin{aligned}
f(\underline{\lambda}'_{1}-{\varepsilon_1},\cdots,  \underline{\lambda}'_{n-1}
		-{\varepsilon_1}, R_1)\geqslant  f(\underline{\lambda}),
\end{aligned} 	\end{equation}
and $( \underline{\lambda}'_{1}-{\varepsilon_1},\cdots,  \underline{\lambda}'_{n-1}
-{\varepsilon_1}, R_1 )\in\Gamma.$  
Take $\epsilon_1>0$ sufficiently small such that 
\begin{equation}
	\begin{aligned}
		\epsilon_1 \sup_{\partial M}|\nabla(u-\underline{u})| \leqslant \frac{\varepsilon_1}{4}. 
		\end{aligned}
\end{equation}
Such an $\epsilon_1$ exists according to Lemma \ref{lemma-c0-bc1}. 
Fix such constants above.
 Denote
\begin{equation}
	\begin{aligned}
	B'(R)=	\begin{pmatrix}
			\mathfrak{\underline{g}}_{\alpha\bar \beta}-\epsilon_1\eta \delta_{\alpha\beta}  &\mathfrak{g}_{\alpha\bar n}\\
			\mathfrak{g}_{n\bar \beta}& R 
		\end{pmatrix}, \,\,
		B''(R)=	\begin{pmatrix}
	\sigma_{\alpha\bar\beta}+\epsilon_1\delta_{\alpha\beta} &0 \\
	0& R  
	\end{pmatrix}.
		\end{aligned}
\end{equation}
Under the assumption \eqref{bdry-assumption1}, 
  there exists a uniform   constant $R_4$ such that
\begin{equation}
	\label{in1-goodcone}
	\lambda [B''(R_4)]\in \mathring{\Gamma}_{\mathcal{G}}^{f}. 
\end{equation}

Let $A(R)$ be as in \eqref{def1-AR}, i.e., 
$$	{A}(R) =
\begin{pmatrix}
	\mathfrak{{g}}_{\alpha\bar \beta} &\mathfrak{g}_{\alpha\bar n}\\
	\mathfrak{g}_{n\bar \beta}& R  \nonumber
\end{pmatrix}.$$ 
By 
the boundary value condition we have  
\eqref{yuan3-buchong5}, i.e., $	\mathfrak{g}_{\alpha\bar\beta}=  
\underline{\mathfrak{g}}_{\alpha\bar\beta} + \eta \sigma_{\alpha\bar\beta}.$
Obviously, 
\begin{equation}
	\label{identity2}
	\begin{aligned}
		A(R) = B'(R-\eta R_4)+\eta B''(R_4). 
		\end{aligned}
\end{equation}

Let's pick  $\epsilon=\frac{ \varepsilon_1}{4}$ in 
Lemma  \ref{yuan's-quantitative-lemma}. 
In analogy with \eqref{def1-R-c}, we take
\begin{equation}
	\label{def1-R-c-2}
	\begin{aligned}
		R_c' \equiv 
		\frac{4(2n-3)}{ \varepsilon_1}\sum_{\alpha=1}^{n-1} | \mathfrak{g}_{\alpha\bar n}|^2 	 	  
		+ (n-1)  \sum_{\alpha=1}^{n-1}|{\underline{\lambda}_\alpha'}| 
		+\frac{n(n-1) \varepsilon_1}{2}+ R_1
			+ \eta R_4. 
	\end{aligned}
\end{equation}  
According to Lemma \ref{yuan's-quantitative-lemma},
the eigenvalues  
of ${B}'(R_c-\eta R_4)$
 behave like
\begin{equation}
	\label{lemma12-yuan-2}
	\begin{aligned}
	  \lambda_{\alpha}[{B}'(R_c-\eta R_4)]\geqslant  (\underline{\lambda}'_{\alpha}-{\varepsilon_1}/{2}), \mbox{  } 1\leqslant \alpha\leqslant n-1, 	
	\end{aligned}
	\end{equation}
\begin{equation}
	\label{lemma12-yuan-3}
	\begin{aligned}
		\lambda_{n}[{B}'(R_c-\eta R_4)]\geqslant R_c'-\eta R_4.
	\end{aligned}
\end{equation}
In particular, $\lambda[{B}'(R_c'-\eta R_4)]\in \Gamma$ and so  $\lambda[A(R_c')]\in \Gamma$.  Combining with \eqref{key-03-yuan6}, we 
conclude that $$f(\lambda[{B}'(R_c'-\eta R_4)]) \geqslant f(\underline{\lambda})\geqslant\psi.$$ 
Together with \eqref{in1-goodcone}, \eqref{identity2} and $\eta\geqslant0$, Lemma \ref{lemma2-key} implies that
\begin{equation}
	f(\lambda[A(R_c')]) \geqslant f(\lambda[{B}'(R_c'-\eta R_4)]).
	\nonumber
\end{equation}
This gives that
$\mathfrak{g}_{n\bar n}\leqslant R_c'.$
	
	\end{proof}

 Combining 
 Propositions   \ref{proposition-quar-yuan1} and \ref{mix-general}, we 
 conclude that
 \begin{theorem}
 	\label{thm2-bdy}
 	In Theorem \ref{thm1-bdy}
 	the constant $C$ is independent of $(\delta_{\psi, f})^{-1}$ and 	the hypothesis~\eqref{unbounded-1} can be dropped, provided that 
 	  $\partial M$ satisfies~\eqref{bdry-assumption1}  and $f$ obeys \eqref{elliptic}.
 \end{theorem}
 
When $\partial M$ is  \textit{holomorphically flat},
combining Propositions \ref{proposition-quar-yuan1} and \ref{mix-Leviflat},
 we obtain some more delicate result.  
 \begin{theorem}
 	\label{thm2-bdy-leviflat}
In addition to \eqref{elliptic}, \eqref{concave},  
 	\eqref{nondegenerate},	\eqref{subsolution1}, 
 	  $\varphi\in C^3(\partial M)$
 and $\psi\in C^1(\bar M)$, we assume
  that $\partial M$ is \textit{holomorphically flat}.
 	Then for any admissible solution $u\in C^3(M)\cap C^2(\bar M)$ to the Dirichlet problem \eqref{mainequ-Dirichlet}, we have
 	\[\sup _{\partial M}  |\partial\overline{\partial} u|_\omega \leqslant  C\left(1+\sup _{M}|\nabla u|_\omega^{2}\right).\] 
 	Here $C$ is a uniform positive constant depending only on 
 	$|\psi|_{C^{1}(\bar M)}$, $|\nabla u|_{C^0(\partial M)}$, 
 	$|\underline{u}|_{C^{2}(\bar M)}$, $|\varphi|_{C^{3}(\bar M)}$, 
 	$\partial M$
 	up to second derivatives 
 	and other known data. In addition, the constant is independent of $(\delta_{\psi,f})^{-1}$.
 	
 \end{theorem}

\subsection{Uniqueness of weak solution} \label{uniqueness-weak-solution}
Following Chen \cite{Chen}, we define 
\begin{definition}
	\label{def-c0-weak}
	A continuous function $u\in C(\bar M)$ is   a weak $C^0$-solution to the degenerate equation \eqref{mainequ-Dirichlet}  with prescribed boundary data $\varphi$, if
	for any sufficiently small positive constant $\epsilon$ there is a $C^2$-{admissible} function $\widetilde{u}$ such that $$\left|u-\widetilde{u}\right|<C(\epsilon),
	$$ 
	where $\widetilde{u}$ solves
	\begin{equation}
		\begin{aligned}
			\,&	F(\mathfrak{g}[\widetilde{u}])= 
			\psi+\rho_{\epsilon} 
			\mbox{ in } M, \,& \widetilde{u}=\varphi \mbox{ on } \partial M. \nonumber
		\end{aligned}
	\end{equation}
	Here $\rho_{\epsilon}$ is a function satisfying $0<\rho_{\epsilon}<\epsilon$, and $C(\epsilon)\rightarrow 0$ as $\epsilon\rightarrow 0$.
\end{definition}

We first obtain the stability of weak $C^0$-solutions.   The proof is almost parallel to that of   \cite[Theorem 4]{Chen}.
We  omit the proof here.
\begin{theorem}
	\label{weakc0comparison}
	Suppose $u^1$, $u^2$ are two  weak $C^0$-solutions to the degenerate equation  
 \eqref{mainequ-Dirichlet}
	with 
	boundary data $\varphi^1$, $\varphi^2$. 
	Then $$\sup_{M}|u^1-u^2|\leqslant \sup_{\partial M}|\varphi^1-\varphi^2|.$$
\end{theorem}

As a corollary, we derive the uniqueness of weak $C^0$-solution.
\begin{corollary}
	\label{unique-weak-solution}
	The weak $C^0$-solution to the Dirichlet problem \eqref{mainequ-Dirichlet} for degenerate equation 
	is unique, provided the boundary data is fixed.
\end{corollary}

\medskip
\section{Dirichlet problem on products}
\label{sec6}

The Dirichlet problem for degenerate  
	equations have important applications in complex geometry and analysis.	There are many challenging open problems. 
	 
In this section we study a special problem on certain product of the form  
\begin{equation}
	\label{const1-product}
	\begin{aligned} 
		(M,\omega)=(X\times S,\omega),
	\end{aligned}
\end{equation} 
where and hereafter, 
$(X,  \omega_X)$
 is a closed 
Hermitian manifold of complex dimension $n-1$, and
$(S,   \omega_S)$
 is a compact Riemann surface with sufficiently smooth boundary $\partial S$, and  $\omega$ is a Hermitian metric being compatible with  the induced complex structure  $J$.  (We remark that $\omega$ is not necessarily the product metric). 

Let $\vec{\bf \nu}$ be  the unit inner normal vector along the boundary. 

\subsection{Construction of subsolutions}
\label{sec7}

The Dirichlet problem is 
insolvable 
without subsolution assumption. 
Therefore, one natural 
question is to construct subsolutions.
Unfortunately, 
beyond on certain domains 
in Euclidean spaces 
\cite{CKNS2,CNS3} (see also \cite{LiSY2004}), few progress has been made on curved manifolds. 

In this section,  on the product $M=X\times S$ we construct strict admissible subsolutions with 
$\frac{\partial\underline{u}}{\partial \vec{\bf \nu}} \big|_{\partial M}< 0$,
 provided that
\begin{equation}
	\label{cone-condition0} \begin{aligned}	\lambda[\omega^{-1}(\mathfrak{g}[\varphi]+ t\pi_2^*\omega_S)]\in \Gamma \mbox{ in } \bar M,  
	\end{aligned}
\end{equation}
 for some $t>0$, and
\begin{equation}
	 \label{cone-condition1}
	\begin{aligned}
		\lim_{t\rightarrow +\infty} f (\lambda[\omega^{-1}(\mathfrak{g}[\varphi]+ t\pi_2^*\omega_S)] )>\psi 
	 \mbox{ in } \bar M.
	\end{aligned}
\end{equation}
Henceforth $\pi_1: X\times S\rightarrow X$ and $\pi_2: X\times S\rightarrow S$ stand for the natural projections,
(Here $\varphi$ is an appropriate extension of boundary data to $\bar{M}$).

To do this, we begin with the solution $h$ to 
\begin{equation}  
	\label{possion-def}
	\begin{aligned}
		\Delta_S h =1 \mbox{ in } S, \,\,h=0 \mbox{ on } \partial S,
	\end{aligned}
\end{equation} 
where $\Delta_S$ is the Laplacian operator of $(S,J_S,\omega_S)$.
The existence and regularity can be found in  standard monographs; 
  see e.g. \cite{GT1983}. 
More precisely, when $\partial S\in C^\infty$,   $h\in C^{\infty}(\bar S)$; 
while  $h\in C^\infty(S)\cap C^{2,\beta}(\bar S)$ if $\partial S\in C^{2,\beta}$ $(0<\beta<1)$. Moreover,    $\left.\frac{\partial h}{\partial\vec{\bf \nu}}\right|_{\partial S}<0$.
%
The subsolution is precisely given by 
\begin{equation} 
	\label{construct-subsolution}  \begin{aligned}
		\underline{u}=\varphi+t\pi_2^* h
\end{aligned}\end{equation}
for  $t\gg1$
($\pi_2^* h=h\circ\pi_2$, still denoted by $h$  
for simplicity),  
then
\begin{equation}  
	\begin{aligned}
		\mathfrak{g}[\underline{u}]=\mathfrak{g}[\varphi]+t\pi_2^*\omega_S, \,\,
		\frac{\partial \underline{u}}{\partial \vec{\bf \nu}}\Big|_{\partial M}<0 \mbox{ for } t\gg1. \nonumber
	\end{aligned}
\end{equation}
Therefore,   $\underline{u}$ is the subsolution if  \eqref{cone-condition0} and \eqref{cone-condition1} hold.
Furthermore, 
the condition \eqref{cone-condition1} always holds when $f$ satisfies
\eqref{unbounded-1} and  $\varphi$ is admissible.  

Significantly, if $\omega=\pi_1^*\omega_X+\pi_2^*\omega_S$ and $\chi$ splits by $\chi=\pi_1^*\chi_1+ \pi_2^*\chi_2,$   then 
according to  
Lemma \ref{yuan's-quantitative-lemma}, 
conditions \eqref{cone-condition0} and  \eqref{cone-condition1} reduce  to 
\begin{equation}
	\label{cone-condition1-1}
	\begin{aligned}
		\lim_{t\rightarrow +\infty} f(\lambda (\omega_X^{-1}\chi_1), t)>\psi 
		\mbox{ and  } \lambda (\omega_X^{-1}\chi_1)\in \Gamma_\infty  \mbox{ in } \bar M,   \nonumber
	\end{aligned}
\end{equation}
where    
$\chi_1$ is a real $(1,1)$-form on $X$, 
$\chi_2$ is a real $(1,1)$-form on $S$.
This indicates that, 
for $\chi=\pi_1^*\chi_1+ \pi_2^*\chi_2$ the solvability of Dirichlet problem is  determined by 
$\chi_1$. 

\begin{remark}
	
	Let 
 $(Y, \omega_Y)$ 	be a closed Hermitian manifold of complex dimension $n-m$, let $\Omega$ be a smooth bounded strictly pseudoconvex domain in $\mathbb{C}^m$, $2\leqslant  m\leqslant  n-1$. Under the assumption that \begin{equation}\label{unbounded-buchong-m}\begin{aligned}
			\lim _{t \rightarrow+\infty} f(\lambda_{1}, \cdots, \lambda_{n-m},\lambda_{n-m+1}+t,  \cdots \lambda_{n}+t)=\sup_{\Gamma} f, \mbox{ } \forall    \left(\lambda_{1}, \cdots, \lambda_{n}\right) \in \Gamma,  \nonumber
		\end{aligned}
	\end{equation}
we can construct subsolutions on more general product 
$(M,\omega)=(Y\times \Omega,\omega).$

	\end{remark}

\subsection{The Dirichlet problem with less regular boundary assumption} 
\label{sec8}

  Theorem \ref{thm2-bdy-leviflat} implies that in the product case, 
the quantitative boundary estimate \eqref{bdy-sec-estimate-quar1}  
depends only on $\partial M$ up to 
\textit{second derivatives} 
and other known data. 
With the bound  of
complex Hessian of $u$ 
at hand,  the equation becomes uniformly elliptic. And then as in \cite{Guan2010Li} we can
  bound the real Hessian 
by   maximum principle.   (Also, 
as shown in \cite{TWWYEvansKrylov2015} using a trick from \cite{WangYu12}, 
one can  
convert  \eqref{mainequ-Dirichlet} to a real uniformly elliptic concave equation).
Thus a result due to Silvestre-Sirakov \cite{Silvestre2014Sirakov} allows one to derive
$C^{2,\alpha}$ boundary regularity with only assuming $C^{2,\beta}$ boundary. 
Consequently, we can study the Dirichlet problem on the products
with less regular boundary\renewcommand{\thefootnote}{\fnsymbol{footnote}}\footnote{We emphasize that the geometric quantities of $(M,\omega)$ (curvature $R_{i\bar j k\bar l}$ and the torsion $T^k_{ij}$) keep bounded as approximating to $\partial M$, and all derivatives of ${\chi}_{i\bar j}$ has continues extensions to $\bar M$, whenever $M$ has less regularity boundary. 
	Typical examples are as follows: $M\subset M'$, 	$\mathrm{dim}_{\mathbb{C}}M'=n$, 	$\omega=\omega_{M'}|_{M}$ and the given data ${\chi}$ can be smoothly defined on $M'$.}
and obtain some delicate result in analogy with a theorem of Savin \cite{Savin2013} on real non-degenerate Monge-Amp\`ere equation with homogeneous boundary value condition on bounded $C^{2,\alpha}$ strictly convex domains. 

\begin{theorem}
	\label{mainthm-09}
	Let  $(M,  \omega)=(X\times S,  \omega)$ be as in \eqref{const1-product} with
	$\partial S\in C^{2,\beta}$ for some $0<\beta<1$. 
 Let $\varphi\equiv0$.
	Suppose, in addition to    \eqref{elliptic} and \eqref{concave}, that 
	\begin{equation}
		\label{cone-condition1-0}
		\begin{aligned}
			f(\lambda[\omega^{-1}(\chi+ t\pi_2^*\omega_S)])>\psi,
			\mbox{ } \lambda[\omega^{-1}(\chi+ t\pi_2^*\omega_S)]\in \Gamma
			\mbox{ in } \bar M, \mbox{ } t\gg1. \nonumber
		\end{aligned}
	\end{equation}
	Then we have two conclusions:
	\begin{itemize}
		\item  The 
		Dirichlet problem \eqref{mainequ-Dirichlet} has a 
		unique $C^{2,\alpha}$ admissible solution
		for some 
		$0<\alpha\leqslant\beta$, provided that 
		$\psi\in C^2(\bar M)$, $\inf_{M} \psi> \sup_{\partial \Gamma}f$.  
		\item Suppose in addition that
		$f\in C^\infty(\Gamma)\cap C(\bar\Gamma)$,   $\psi\in C^{1,1}(\bar M)$  and   
		$\inf_{M} \psi=sup_{\partial \Gamma}f$. Then the Dirichlet problem  \eqref{mainequ-Dirichlet}   has a weak solution   with 
		$u\in C^{1,\alpha}(\bar M)$, $\forall 0<\alpha<1$  $\lambda(\omega^{-1}\mathfrak{g}[u])\in \bar \Gamma$ and $\Delta u \in L^{\infty}(\bar M)$.
	\end{itemize}
\end{theorem}

\begin{proof}
	It  suffices to consider the nondegenerate case: There is some  $\delta_0>0$ so that
	\begin{equation}
		\begin{aligned}
			\psi\geqslant\sup_{\partial\Gamma}f+\delta_0 \mbox{ in } \bar M.
		\end{aligned}
	\end{equation}

	The first step is to construct  approximate Dirichlet problems with constant boundary  data.
	Let $h$ be the solution to \eqref{possion-def}. For $t\gg1$, $\underline{u}=th$ satisfies
	\begin{equation}
		\begin{aligned}
		 F(\mathfrak{g}[\underline{u}])=
			f(\lambda(\omega^{-1}\mathfrak{g}[\underline{u}]))\geqslant \psi+\delta_1 \mbox{ in }\bar M 
		\end{aligned}
	\end{equation}
	for some $\delta_1>0$, and $-t\mathrm{tr}_\omega(\pi_2^*\omega_S)+\mathrm{tr}_\omega\chi\leqslant 0   \mbox{ in } \bar M.$
	Note that $$h\in C^\infty(S)\cap C^{2,\beta}(\bar S) \mbox{ and } \left.\frac{\partial h}{\partial\vec{\bf \nu}}\right|_{\partial S}<0,$$
	we get a sequence of level sets of $h$, say $\{h=-\alpha_k\}$,
	and  a family of smooth Riemann surfaces $S_{k}$ enclosed by  $\{h=-\alpha_k\}$,
	such that 
	$\cup S_{k}=S$ and $\partial S_{k}$ converge to $\partial S$ in the norm of $C^{2,\beta}$. 
	Denote $M_{k}=X\times S_{k}$. 
	
	For any $k\geqslant 1$, there exists a $\psi^{(k)}\in C^\infty(\bar M_k)$ 
	such that
	\begin{equation}
		\begin{aligned}
			|\psi-\psi^{(k)}|_{C^2(\bar M_k)}\leqslant 1/k. 
		\end{aligned}
	\end{equation}
	For $k\gg1$ 
	we have
	\begin{equation}
		\label{App-Diri1-subsolution}
		\begin{aligned}
  F(\mathfrak{g}[\underline{u}])
			\geqslant \psi^{(k)} +{\delta_1}/{2} \mbox{ in } M_{k},  \,\,
			\underline{u}=-t\alpha_k \mbox{ on }   \partial M_{k}
	\end{aligned}\end{equation} 
	which is a strict, admissible subsolution to approximate Dirichlet problem
	\begin{equation}
		\label{App-Diri1}
		\begin{aligned}
	  F(\mathfrak{g}[u])
			= \psi^{(k)} \mbox{ in } M_{k},  \,\,
			u =-t\alpha_k \mbox{ on }   \partial M_{k}.
	\end{aligned}\end{equation} 
	According to 
	Theorem  
	 	\ref{thm1-2}, 
	 the Dirichlet problem \eqref{App-Diri1} admits a unique  smooth  admissible solution $u^{(k)}\in C^{\infty}(\bar M_k)$. 
	Notice that $M_k$ is a product  and the boundary data in \eqref{App-Diri1} is a constant, Theorem \ref{thm2-bdy-leviflat} applies. Therefore,
	\begin{equation}
		\label{uniform-00}
		\begin{aligned}
			\sup_{M_k}|\partial\overline{\partial} u^{(k)}|_\omega \leqslant C_k \left(1+\sup_{M_{k}}|\nabla u^{(k)}|_\omega^2\right)
		\end{aligned}
	\end{equation}
	where $C_k$ is a constant depending on $|u^{(k)}|_{C^0(M_{k})}$, $|\nabla u^{(k)}|_{C^0(\partial M_{k})}$, $|\underline{u}|_{C^{2}(M_{k})}$,   $|\psi^{(k)}|_{C^{2}(M_{k})}$,
	$\partial S_{k}$ up to second order derivatives  and other known data (but not on $(\delta_{\psi^{(k)},f})^{-1}$).

	It requires to prove  
	\begin{equation}
		\label{uniform-c0-c1}
		\begin{aligned}
			\sup_{M_{k}}|u^{(k)}| +\sup_{\partial M_{k}}|\nabla u^{(k)}|\leqslant C, \mbox{ independent of } k.
		\end{aligned}
	\end{equation}
	Let $w^{(k)}$ be the solution of
	\begin{equation}
		\label{supersolution-k}
		\begin{aligned}
		 \Delta w^{(k)} +\mathrm{tr}_\omega{\chi} =0 \mbox{ in } M_{k},  \,\,
			 w^{(k)}= -t\alpha_k \mbox{ on } \partial M_{k}. \nonumber
		\end{aligned}
	\end{equation}
	By the maximum principle and the boundary value condition, 
	one has
	\begin{equation}
		\label{approxi-boundary1}
		\begin{aligned}
		  \underline{u} \leqslant u^{(k)}\leqslant w^{(k)} \mbox{ in } M_{k},  \,\,
			\frac{\partial \underline{u}}{\partial \vec{\bf \nu}} \leqslant  \frac{\partial u^{(k)}}{\partial \vec{\bf \nu}} \leqslant \frac{\partial w^{(k)}}{\partial \vec{\bf \nu}} 
			\mbox{ on } \partial M_{k}.  \nonumber
		\end{aligned}
	\end{equation}
Here $\vec{\bf \nu}$ denotes  the unit inner normal vector along the boundary $\partial M_k$.

	On the other hand,  
	 we have
	\begin{equation}
		\begin{aligned}
			\Delta (-\underline{u}-2t\alpha_k)+\mathrm{tr}_\omega \chi=\,& -t\mathrm{tr}_\omega(\pi_2^*\omega_S)+\mathrm{tr}_\omega\chi\leqslant 0 \,& \mbox{ in } M_k,\\
			-\underline{u}-2t\alpha_k=\,& -t\alpha_k \,& \mbox{ on } \partial M_k.  \nonumber
		\end{aligned}
	\end{equation}
	As a result, we have
	\begin{equation}
		\begin{aligned}
		 w^{(k)}  \leqslant -\underline{u} -2t\alpha_k \mbox{ in } M_{k},   
		 \,\, w^{(k)}  =-\underline{u}-2t\alpha_k   \mbox{ on } \partial M_{k}, \nonumber
		\end{aligned}
	\end{equation}
	which further implies
	\[\frac{\partial w^{(k)}}{\partial\vec{\bf \nu}} \leqslant -\frac{\partial \underline{u}}{\partial \vec{\bf \nu}} \mbox{ on } \partial M_k\] 
	as required. 
	Consequently, 
	\eqref{uniform-00} holds for a uniform constant $C'$ which does not depend on $k$. 
	Thus $$|u|_{C^2(M_{k})}\leqslant C, \mbox{ independent of } k.$$ 
	
	To complete the proof, we apply Silvestre-Sirakov's \cite{Silvestre2014Sirakov} result to 
	derive $C^{2,\alpha'}$ estimates on the boundary, while
	the convergence of  $\partial S_{k}$ in the norm $C^{2,\beta}$ 
	allows  us to take a limit ($\alpha'$ can be uniformly chosen).

\end{proof}

	In the proof of Theorem \ref{mainthm-09}, we use the domains enclosed by level sets 
	to approximate the background space. Such an approximation gives 
	a bound of \eqref{uniform-c0-c1} explicitly.  
Similarly, we can prove 
 that 

 \begin{theorem}
 	\label{mainthm-10-degenerate}  
 	Let $(M,  \omega)=(X\times S,  \omega)$ be as above 
 	with $\partial S\in C^{2,1}$, and let
 	$f$ satisfy \eqref{elliptic}, 
 	\eqref{concave} 
 	and \eqref{continuity1}.  
 	Given the data $\varphi\in C^{2,1}(\bar M)$, $\psi\in C^{1,1}(\bar M)$   satisfying 
 	\eqref{cone-condition0}, \eqref{cone-condition1} and \eqref{degenerate-RHS}, the Dirichlet problem \eqref{mainequ-Dirichlet}  
 	has a weak solution $u$ with   
 	\begin{equation}
 		\begin{aligned}
 			u\in C^{1,\alpha}(\bar M), 
 			\mbox{  } \forall 0<\alpha<1, \mbox{ } \Delta u \in L^{\infty}(\bar M),
 			\mbox{  }
 			\lambda[\omega^{-1}(\chi+\sqrt{-1}\partial\overline{\partial}u)]\in \bar \Gamma \mbox{ in } \bar M. \nonumber
 		\end{aligned}
 	\end{equation} 

 \end{theorem}

 \begin{remark}

 	The Dirichlet problem for homogeneous complex Monge-Amp\`ere equation 
 	on $X\times S$ is closely connected with the geodesics in the space of K\"ahler potentials $\mathcal{H}_{\omega_X}$ \cite{Donaldson99,Semmes92,Mabuchi87}; see e.g.  \cite{Blocki09geodesic,Chen,Phong-Sturm2010}.
 	We also refer to the monograph \cite{Guedj2017Zeriahi}
 	 for more results regarding to degenerated   Monge-Amp\`ere equation. 
 \end{remark}

\medskip
 
 \section{Proof of  Lemma \ref{yuan's-quantitative-lemma}}
 \label{sec-proofofquantitativelemma1}
 
 In this section we  complete the proof of Lemma \ref{yuan's-quantitative-lemma}.
 We start with the case $n=2$. 
 For $n=2$, the eigenvalues of $A$ are precisely as follows
 \begin{equation}
 	\begin{aligned}
 		\lambda_{1}=\frac{\mathrm{{\bf a}}+d_1- \sqrt{(\mathrm{{\bf a}}-d_1)^2+4|a_1|^2}}{2},
 		\,\,
 		\lambda_2=\frac{\mathrm{{\bf a}}+d_1+\sqrt{(\mathrm{{\bf a}}-d_1)^2+4|a_1|^2}}{2}. \nonumber
 	\end{aligned}
 \end{equation}
 We assume $a_1\neq 0$; otherwise we are done.
 If $\mathrm{{\bf a}} \geqslant \frac{|a_1|^2}{ \epsilon}+ d_1$ then one has
 \begin{equation}
 	\begin{aligned}
 		0\leqslant d_1- \lambda_1 =\lambda_2-\mathrm{{\bf a}}
 		= \frac{2|a_1|^2}{\sqrt{ (\mathrm{{\bf a}}-d_1)^2+4|a_1|^2 } +(\mathrm{{\bf a}}-d_1)}
 		< \frac{|a_1|^2}{\mathrm{{\bf a}}-d_1 } \leqslant \epsilon.   \nonumber
 	\end{aligned}
 \end{equation}
 Here we use $a_1\neq 0$ to confirm the strictly inequality in the above formula.

 The following lemma enables us to count  the eigenvalues near the diagonal elements
 via a deformation argument. It is an essential  ingredient in the proof of
 Lemma \ref{yuan's-quantitative-lemma} for general $n$.
 \begin{lemma}
 	\label{refinement}
 	Let $A$ be 
 	a Hermitian $n$ by  $n$  matrix
 	\begin{equation}\label{matrix2}\left(\begin{matrix}d_1&&  &&a_{1}\\&d_2&& &a_2\\&&\ddots&&\vdots \\&& &  
 			d_{n-1}& a_{n-1}\\ \bar a_1&\bar a_2&\cdots& \bar a_{n-1}& \mathrm{{\bf a}} \nonumber
 		\end{matrix}\right)\end{equation} with $d_1,\cdots, d_{n-1}, a_1,\cdots, a_{n-1}$ fixed, and with $\mathrm{{\bf a}}$ variable.
 	Denote
 	$\lambda=(\lambda_1,\cdots, \lambda_n)$ the  eigenvalues of $A$ with the order
 	$\lambda_1\leqslant \lambda_2 \leqslant \cdots \leqslant \lambda_n$.
 	Fix  $\epsilon$ a positive constant.
 	Suppose that the parameter $\mathrm{{\bf a}}$
 	satisfies  the following quadratic growth condition
 	\begin{equation}
 		\label{guanjian2}
 		\begin{aligned}
 			\mathrm{{\bf a}} \geqslant \frac{1}{\epsilon}\sum_{i=1}^{n-1} |a_i|^2+\sum_{i=1}^{n-1}  \left(d_i+ (n-2) |d_i|\right)+ (n-2)\epsilon.
 		\end{aligned}
 	\end{equation}
 	Then for any $\lambda_{\alpha}$ $(1\leqslant \alpha\leqslant n-1)$ there exists a 
 	$d_{i_{\alpha}}$
 	with 
 	$1\leqslant i_{\alpha}\leqslant n-1$ such that
 	\begin{equation}
 		\label{meishi}
 		\begin{aligned}
 			|\lambda_{\alpha}-d_{i_{\alpha}}|<\epsilon,
 		\end{aligned}
 	\end{equation}
 	\begin{equation}
 		\label{mei-23-shi}
 		0\leqslant \lambda_{n}-\mathrm{{\bf a}} <(n-1)\epsilon + \left|\sum_{\alpha=1}^{n-1}(d_{\alpha}-d_{i_{\alpha}})\right|.
 	\end{equation}
 \end{lemma}

 \begin{proof}
 	Without loss of generality, we assume $\sum_{i=1}^{n-1} |a_i|^2>0$ and  $n\geqslant 3$
 	(otherwise we are done). 
 	Note that  the eigenvalues have
 	the order $\lambda_1\leqslant \lambda_2\leqslant \cdots \leqslant \lambda_n$, as in the assumption of lemma.
 	It is well known that, 
 	for a Hermitian matrix, any diagonal element is   less than or equals to   the  largest eigenvalue.
 	In particular,
 	\begin{equation}
 		\label{largest-eigen1}
 		\lambda_n \geqslant \mathrm{{\bf a}}.
 	\end{equation}
 	
 	We only need to prove   \eqref {meishi}, since  \eqref{mei-23-shi} is a consequence of  \eqref{meishi}, \eqref{largest-eigen1}  and
 	\begin{equation}
 		\label{trace}
 		\sum_{i=1}^{n}\lambda_i=\mbox{tr}(A)=\sum_{\alpha=1}^{n-1} d_{\alpha}+\mathrm{{\bf a}}.
 	\end{equation}

 	Let's denote   $I=\{1,2,\cdots, n-1\}$. We divide the index set   $I$ into two subsets:
 	$${\bf B}=\{\alpha\in I: |\lambda_{\alpha}-d_{i}|\geqslant \epsilon, \mbox{   }\forall i\in I \}, $$
 	$$ {\bf G}=I\setminus {\bf B}=\{\alpha\in I: \mbox{There exists $i\in I$ such that } |\lambda_{\alpha}-d_{i}| <\epsilon\}.$$
 	
 	To complete the proof, it suffices to prove ${\bf G}=I$ or equivalently ${\bf B}=\emptyset$.
 	It is easy to see that  for any $\alpha\in {\bf G}$, one has
 	\begin{equation}
 		\label{yuan-lemma-proof1}
 		\begin{aligned}
 			|\lambda_\alpha|< \sum_{i=1}^{n-1}|d_i| + \epsilon.
 		\end{aligned}
 	\end{equation}
 	
 	Fix $ \alpha\in {\bf B}$,  we are going to give the estimate for $\lambda_\alpha$.
 	The eigenvalue $\lambda_\alpha$ satisfies
 	\begin{equation}
 		\label{characteristicpolynomial}
 		\begin{aligned}
 			(\lambda_{\alpha} -\mathrm{{\bf a}})\prod_{i=1}^{n-1} (\lambda_{\alpha}-d_i)
 			= \sum_{i=1}^{n-1} |a_{i}|^2 \prod_{j\neq i} (\lambda_{\alpha}-d_{j}).
 		\end{aligned}
 	\end{equation}
 	By the definition of ${\bf B}$, for  $\alpha\in {\bf B}$ one then has $|\lambda_{\alpha}-d_i|\geqslant \epsilon$ for any $i\in I$.
 	We therefore derive
 	\begin{equation}
 		\begin{aligned}
 			|\lambda_{\alpha}-\mathrm{{\bf a}} |\leqslant \sum_{i=1}^{n-1} \frac{|a_i|^2}{|\lambda_{\alpha}-d_{i}|}\leqslant 
 			\frac{1}{\epsilon}\sum_{i=1}^{n-1} |a_i|^2, \,\, \mbox{ if } \alpha\in {\bf B}.
 		\end{aligned}
 	\end{equation}
 	Hence,  for $\alpha\in {\bf B}$, we obtain
 	\begin{equation}
 		\label{yuan-lemma-proof2}
 		\begin{aligned}
 			\lambda_\alpha \geqslant \mathrm{{\bf a}}-\frac{1}{\epsilon}\sum_{i=1}^{n-1} |a_i|^2.
 		\end{aligned}
 	\end{equation}

 	For a set ${\bf S}$, we denote $|{\bf S}|$ the  cardinality of ${\bf S}$.
 	We shall use proof by contradiction to prove  ${\bf B}=\emptyset$.
 	Assume by contradiction ${\bf B}\neq \emptyset$.
 	Then $|{\bf B}|\geqslant 1$, and so $|{\bf G}|=n-1-|{\bf B}|\leqslant n-2$. 

 	We compute the trace of the matrix $A$ as follows:
 	\begin{equation}
 		\begin{aligned}
 			\mbox{tr}(A)=\,&
 			\lambda_n+
 			\sum_{\alpha\in {\bf B}}\lambda_{\alpha} + \sum_{\alpha\in  {\bf G}}\lambda_{\alpha}\\
 			\geqslant \,&
 			\lambda_n+
 			|{\bf B}| (\mathrm{{\bf a}}-\frac{1}{\epsilon}\sum_{i=1}^{n-1} |a_i|^2 )-|{\bf G}| (\sum_{i=1}^{n-1}|d_i|+\epsilon ) \\
 			> \,&
 			2\mathrm{{\bf a}}-\frac{1}{\epsilon}\sum_{i=1}^{n-1} |a_i|^2 -(n-2) (\sum_{i=1}^{n-1}|d_i|+\epsilon )
 			\\
 			\geqslant \,& \sum_{i=1}^{n-1}d_i +\mathrm{{\bf a}}= \mbox{tr}(A),
 		\end{aligned}
 	\end{equation}
 	where we use  \eqref{guanjian2},   \eqref{largest-eigen1}, \eqref{yuan-lemma-proof1} and \eqref{yuan-lemma-proof2}.
 	This is a contradiction.
 	We have ${\bf B}=\emptyset$.
 	Therefore, ${\bf G}=I$ and  the proof is complete.
 \end{proof}
 
Consequently we obtain
 
 \begin{lemma}
 	\label{refinement111}
 	Let $A(\mathrm{{\bf a}})$ be an $n\times n$ Hermitian  matrix
 	\begin{equation}
 		A(\mathrm{{\bf a}})=\left(
 		\begin{matrix}
 			d_1&&  &&a_{1}\\
 			&d_2&& &a_2\\
 			&&\ddots&&\vdots \\
 			&& &  d_{n-1}& a_{n-1}\\
 			\bar a_1&\bar a_2&\cdots& \bar a_{n-1}& \mathrm{{\bf a}} \nonumber
 		\end{matrix}
 		\right)
 	\end{equation}
 	with
 	$d_1,\cdots, d_{n-1}, a_1,\cdots, a_{n-1}$ fixed, and with $\mathrm{{\bf a}}$ variable.
 	Assume that $d_1, d_2, \cdots, d_{n-1} $
 	are distinct with each other, i.e. $d_{i}\neq d_{j}, \forall i\neq j$.
 	Denote
 	$\lambda=(\lambda_1,\cdots, \lambda_n)$ as the the eigenvalues of $A(\mathrm{{\bf a}})$.
 	Given a positive constant $\epsilon$ with
 	$0<\epsilon \leqslant \frac{1}{2}\min\{|d_i-d_j|: \forall i\neq j\}$, 
 	if the parameter  $\mathrm{{\bf a}}$ satisfies the quadratic growth condition 
 	\begin{equation}
 		\begin{aligned}
 			\label{guanjian1}
 			\mathrm{{\bf a}}\geqslant \frac{1}{\epsilon}\sum_{i=1}^{n-1}|a_i|^2 +(n-1)\sum_{i=1}^{n-1} |d_i|+(n-2)\epsilon,
 		\end{aligned}
 	\end{equation}
 	then the eigenvalues behave like
 	\begin{equation}
 		\begin{aligned}
 			|d_{\alpha}-	\,&\lambda_{\alpha}|
 			<   \epsilon, \mbox{  }\forall 1\leqslant \alpha\leqslant n-1, \\ \nonumber
 			\,& 0\leqslant \lambda_{n}-\mathrm{{\bf a}}
 			<  (n-1)\epsilon.  \nonumber
 		\end{aligned}
 	\end{equation}
 \end{lemma}
 
 \begin{proof}
 	The proof is based on Lemma  \ref{refinement} and a deformation argument.
 	Without loss of generality,  we assume $n\geqslant 3$ and  $\sum_{i=1}^{n-1} |a_i|^2>0$
 	(otherwise  
 	we are done).
 	Moreover, we assume in addition that $d_1<d_2\cdots<d_{n-1}$ and the eigenvalues have the order
 	$\lambda_1\leqslant \lambda_2\leqslant \cdots \leqslant \lambda_{n-1}\leqslant \lambda_n.$

 	Fix $\epsilon\in (0,\mu_0]$, where $\mu_0=\frac{1}{2}\min\{|d_i-d_j|: \forall i\neq j\}$.
 	Denote
 	$$I_{i}=(d_i-\epsilon,d_i+\epsilon),$$
 	and $$P_0=\frac{1}{\epsilon}\sum_{i=1}^{n-1} |a_i|^2+ (n-1)\sum_{i=1}^{n-1} |d_i|+ (n-2)\epsilon.$$
 	Since $0<\epsilon \leqslant \mu_0$,  the intervals disjoint each other
 	\begin{equation}
 		\label{daqin122}
 		I_\alpha\bigcap I_\beta=\emptyset \mbox{ for }  1\leqslant \alpha<\beta\leqslant n-1.
 	\end{equation}
 	
 	In what follows,
 	we assume that  the parameter  $\mathrm{{\bf a}}$ satisfies \eqref{guanjian1} and  the Greek letters $\alpha,  \beta$
 	range from $1$ to $n-1$.
 Define a function
 	$$\mathrm{{\bf Card}}_\alpha: [P_0,+\infty)\rightarrow \mathbb{N}$$
 	to count the eigenvalues which lie in $I_\alpha$.
 	(Note that when the eigenvalues are not distinct,  the function $\mathrm{{\bf Card}}_\alpha$ means the summation of all the algebraic
 	multiplicities
 	of  distinct eigenvalues 
 	which  lie in $I_\alpha$). 
 	This function measures the number of the  eigenvalues which lie in $I_\alpha$.
 	
 	We are going to prove that $\mathrm{{\bf Card}}_\alpha$ is continuous on $[P_0,+\infty)$.
 	%
 	First, Lemma \ref{refinement} asserts that if $\mathrm{{\bf a}} \geqslant P_0$  then
 	\begin{equation}
 		\label{daqin111}
 		\begin{aligned}
 			\lambda_{\alpha}\in  \bigcup_{i=1}^{n-1} I_i, \mbox{   } \forall 1\leqslant \alpha\leqslant n-1.
 		\end{aligned}
 	\end{equation}

 	It is well known that the largest eigenvalue $\lambda_n\geqslant \mathrm{{\bf a}}$,
 	while the smallest eigenvalue $\lambda_1 \leqslant d_1$.
 	Combining  with \eqref{daqin111}
 	one has
 	\begin{equation}
 		\label{largest1}
 		\begin{aligned}
 			\lambda_{n} \geqslant  \mathrm{{\bf a}}
 			>\,& \sum_{i=1}^{n-1}|d_i| +\epsilon.
 		\end{aligned}
 	\end{equation}
 	Thus $\lambda_n\in \mathbb{R}\setminus (\bigcup_{i=1}^{n-1} \overline{I_i})$
 	where $\overline{I_i}$ denotes the closure of $I_i$.
 	Therefore, 
 	the function $\mathrm{{\bf Card}}_\alpha$
 	is continuous (thus it is a constant), since  
 	\eqref{daqin111}, \eqref{daqin122},  $\lambda_n\in \mathbb{R}\setminus (\overline{\bigcup_{i=1}^{n-1} I_i})$
 	and  the eigenvalues of $A(\mathrm{{\bf a}})$   depend  continuously on the parameter  $\mathrm{{\bf a}}$.
 	
 	The continuity of $\mathrm{{\bf Card}}_\alpha(\mathrm{{\bf a}})$ plays a crucial role in this proof.
 	Following the outline of the proof of  
  \cite[Lemma 1.2]{CNS3},
 	in the setting of Hermitian matrices,  one can show that for $1\leqslant \alpha\leqslant n-1$,
 	\begin{equation}
 		\label{yuanrr-lemma-12321}
 		\lim_{\mathrm{{\bf a}}\rightarrow +\infty} \mathrm{{\bf Card}}_{\alpha}(\mathrm{{\bf a}}) \geqslant 1.
 	\end{equation} 
 	It follows from \eqref{largest1},  \eqref{yuanrr-lemma-12321}
 	and the  continuity  of $\mathrm{{\bf Card}}_\alpha $ that
 	\begin{equation}
 		\label{daqin142224}
 		\mathrm{{\bf Card}}_{\alpha}(\mathrm{{\bf a}})= 1,
 		\mbox{   } \forall \mathrm{{\bf a}}\in [P_0, +\infty),  \mbox{   } 1\leqslant \alpha\leqslant n-1.  \nonumber
 	\end{equation}
 	Together with \eqref{daqin111}, we prove that,
 	for any   $1\leqslant \alpha\leqslant n-1$, the interval $I_\alpha=(d_{\alpha}-\epsilon,d_{\alpha}+\epsilon)$
 	contains the eigenvalue $\lambda_\alpha$.
 	We thus complete  the proof.
 \end{proof}

 Suppose that there are two distinct indices $i_0, j_0$ ($i_0\neq j_0$) such that $d_{i_{0}}= d_{j_{0}}$.
 Then the characteristic polynomial of $A$
 can be rewritten as the following
 \begin{equation}
 	\begin{aligned}
 		(\lambda-d_{i_{0}})\left[(\lambda-\mathrm{{\bf a}})\prod_{i \neq i_{0}} (\lambda-d_i)-|a_{i_{0}}|^{2}\prod_{j\neq j_0, j\neq i_0} (\lambda-d_j)
 		-\sum_{i\neq i_0}|a_i|^2\prod_{j\neq i, j\neq i_{0}} (\lambda-d_j)\right]. \nonumber
 	\end{aligned}
 \end{equation}
 So $\lambda_{i_0}=d_{i_{0}}$ is an eigenvalue of $A$ for any $a\in \mathbb{R}$.
 Notice that 
 $$(\lambda-\mathrm{{\bf a}})\prod_{i \neq i_{0}} (\lambda-d_i)
 -|a_{i_{0}}|^{2}\prod_{j\neq j_0, j\neq i_0} (\lambda-d_j)
 -\sum_{i\neq i_0}|a_i|^2\prod_{j\neq i, j\neq i_{0}} (\lambda-d_j)$$
 is the characteristic polynomial of the $(n-1)\times (n-1)$ Hermitian  matrix
 \begin{equation}
 	\begin{pmatrix}
 		d_1&&  &&&&a_{1}\\
 		&\ddots && &&&\vdots \\
 		&&\widehat{d_{i_{0}}}&& &&\widehat{a_{i_{0}}}\\
 		&&&\ddots&&&\vdots \\
 		&&& &  d_{j_{0}}&& (|a_{j_{0}}|^{2}+|a_{i_{0}}|^{2})^{\frac{1}{2}}\\
 		&&&& &   \ddots              \\
 		\bar a_1&\cdots&\widehat{ \bar a_{i_{0}}}&\cdots& (|a_{j_{0}}|^{2}+|a_{i_{0}}|^{2})^{\frac{1}{2}} &\cdots & \mathrm{{\bf a}}   \nonumber
 	\end{pmatrix}
 \end{equation}
 where $\widehat{*}$ indicates deletion.
 Therefore, $(\lambda_1, \cdots, \widehat{\lambda_{i_{0}}}, \cdots, \lambda_{n})$ are the eigenvalues of the above $(n-1)\times (n-1)$
 Hermitian  matrix. 
 Hence,  we obtain  
 \begin{lemma}
 	\label{refinement3}
 	Let $A$ be as in Lemma \ref{yuan's-quantitative-lemma} an $n\times n$ Hermitian matrix. 
 	Let
 	\[\mathcal{I}=
 	\begin{cases}
 		\mathbb{R}^{+}=(0,+\infty) \,& \mbox{ if } d_{i}=d_{1}, \forall 2\leqslant i \leqslant n-1,\\
 		\left(0,\mu_0\right),  \mbox{  } \mu_0=\frac{1}{2}\min\{|d_{i}-d_{j}|: d_{i}\neq d_{j}\} \,& \mbox{ otherwise.}
 	\end{cases}
 	\] 
 	Denote
 	$\lambda=(\lambda_1,\cdots, \lambda_n)$   the the eigenvalues of $A$. Fix  $\epsilon\in \mathcal{I}$.
 	Suppose that  the parameter $\mathrm{{\bf a}}$ in $A$ satisfies  \eqref{guanjian1}. 
 	Then the eigenvalues behave like
 	\begin{equation}
 		\begin{aligned}
 		 |d_{\alpha}-\lambda_{\alpha}|
 			< \epsilon, \mbox{    } \forall 1\leqslant \alpha\leqslant n-1, 
 			\,\, 0 \leqslant \lambda_{n}-\mathrm{{\bf a}}
 			< (n-1)\epsilon. \nonumber
 		\end{aligned}
 	\end{equation}
 \end{lemma}

 Applying Lemmas \ref{refinement} and \ref{refinement3},
 we  complete the proof of Lemma \ref{yuan's-quantitative-lemma}  without restriction to the applicable
 scope of $\epsilon$.

 \begin{proof}
 	[Proof of Lemma \ref{yuan's-quantitative-lemma}]
 	We follow the outline of the proof of Lemma \ref{refinement111}.
 	Without loss of generality, we may assume
 	$$n\geqslant 3, \mbox{ } \sum_{i=1}^{n-1} |a_i|^2>0,  \mbox{ }  d_1\leqslant d_2\leqslant \cdots \leqslant d_{n-1}  \mbox{ and } \lambda_1\leqslant \lambda_2 \leqslant \cdots \lambda_{n-1}\leqslant \lambda_n.$$

 	Fix $\epsilon>0$.  Let $I'_\alpha=(d_\alpha-\frac{\epsilon}{2n-3}, d_\alpha+\frac{\epsilon}{2n-3})$ and
 	$$P_0'=\frac{2n-3}{\epsilon}\sum_{i=1}^{n-1} |a_i|^2+ (n-1)\sum_{i=1}^{n-1} |d_i|+ \frac{(n-2)\epsilon}{2n-3}.$$
 	In what follows we assume \eqref{guanjian1-yuan} holds.
 	The connected components of $\bigcup_{\alpha=1}^{n-1} I_{\alpha}'$ are denoted as in the following:
 	$$J_{1}=\bigcup_{\alpha=1}^{j_1} I_\alpha', \mbox{ }
 	J_2=\bigcup_{\alpha=j_1+1}^{j_2} I_\alpha', \mbox{ }  \cdots, J_i =\bigcup_{\alpha=j_{i-1}+1}^{j_i} I_\alpha', \mbox{ } \cdots,
 	J_{m} =\bigcup_{\alpha=j_{m-1}+1}^{n-1} I_\alpha'.$$
 	Moreover
 	\begin{equation}
 		\begin{aligned}
 			J_i\bigcap J_k=\emptyset, \mbox{ for }   1\leqslant i<k\leqslant m. \nonumber
 		\end{aligned}
 	\end{equation}
 	It plays formally the role of \eqref{daqin122} in the proof of Lemma \ref{refinement111}.
 	
 	As in the proof of Lemma \ref{refinement111},
   let
 	$$ \mathrm{{\bf \widetilde{Card}}}_k:[P_0',+\infty)\rightarrow \mathbb{N}$$
 	be the function   counting the eigenvalues that lie in $J_k$.
 	(Note that when the eigenvalues are not distinct,  the function $\mathrm{{\bf \widetilde{Card}}}_k$ denotes  the summation of all the algebraic
 	multiplicities of distinct eigenvalues which
 	lie in $J_k$).
 	By Lemma \ref{refinement} and  
 	$$\lambda_n \geqslant {\bf a}\geqslant P_0'>\sum_{i=1}^{n-1}|d_i|+\frac{\epsilon}{2n-3},$$ 
 	we conclude that
 	if the parameter $\mathrm{{\bf a}}$ satisfies 
 	the 
 	growth condition \eqref{guanjian1-yuan} then
 	\begin{equation}
 		\label{yuan-lemma-proof5}
 		\begin{aligned}
 			\lambda_n \in \mathbb{R}\setminus (\bigcup_{k=1}^{n-1} \overline{I_k'})
 			=\mathbb{R}\setminus (\bigcup_{i=1}^m \overline{J_i}), \mbox{ }
 			\lambda_\alpha \in \bigcup_{i=1}^{n-1} I_{i}'=\bigcup_{i=1}^m J_{i} \mbox{ for } 1\leqslant \alpha\leqslant n-1.
 		\end{aligned}
 	\end{equation}
 	Similarly, $\mathrm{{\bf \widetilde{Card}}}_i(\mathrm{{\bf a}})$ is a continuous function
 	with respect to the variable $\mathrm{{\bf a}}$ when ${\bf a}\geqslant P_0'$. So it is a constant.
 	Combining it with Lemma \ref{refinement3}, 
 	we see that 
 	$$ \mathrm{{\bf \widetilde{Card}}}_i(\mathrm{{\bf a}})
 	=j_i-j_{i-1}$$
 	for ${\bf a}\geqslant P_0'$. Here we denote $j_0=0$ and $j_m=n-1$.
 	We thus know that 
 	$$\lambda_{j_{i-1}+1}, \lambda_{j_{i-1}+2}, \cdots, \lambda_{j_i}$$
 	lie in the connected component $J_{i}$.
 	Thus, for any $j_{i-1}+1\leqslant \gamma \leqslant j_i$,  we see $I_\gamma'\subset J_i$ and  $\lambda_\gamma$
 	lies in the connected component $J_{i}$.
 	Therefore,
 	$$|\lambda_\gamma-d_\gamma| < \frac{(2(j_i-j_{i-1})-1) \epsilon}{2n-3}\leqslant \epsilon.$$
 	Here we use the fact that $d_\gamma$ is the midpoint of  $I_\gamma'$ and $J_i\subset \mathbb{R}$ is an open subset.
 	
 	To be brief,  if for fixed index $1\leqslant i\leqslant n-1$ the eigenvalue $\lambda_i(P_0')$ lies in $J_{\alpha}$ for some $\alpha$, 
 	then  Lemma \ref{refinement} implies that, for any ${\bf a}>P_0'$, the corresponding eigenvalue  $\lambda_i({\bf a})$ lies in the same  interval $J_{\alpha}$.
 	The computation of $\mathrm{{\bf \widetilde{Card}}}_k$ can be done by setting $\mathrm{\bf a}\rightarrow+\infty$.
 \end{proof}


 \medskip 


 \subsubsection*{Acknowledgements} 
 \small
The author was partially supported by the National Natural Science Foundation of China  (Grant No. 11801587), Guangzhou Science and Technology Program  (Grant No. 202201010451),
 and Guangdong Basic and Applied Basic Research Foundation (Grant No. 2023A1515012121).

 \bigskip


\end{document}